\documentclass[reqno,twoside,11pt]{amsart}
\usepackage{graphicx,subcaption}
\usepackage{hyperref}
\usepackage{enumerate, amsthm}
\usepackage{mathrsfs}
\usepackage{rotating}
\usepackage{cancel}
\usepackage{xcolor}
\usepackage[margin=1 in]{geometry}
\usepackage[margin={1.5cm, 1.5cm},font=small, labelsep=endash]{caption}
\usepackage{mathtools}
\usepackage{verbatim}

\DeclareFontFamily{OT1}{fraktura}{}
\DeclareFontShape{OT1}{fraktura}{m}{n} {<5> <6> <7> <8> <9> <10> <11> <12> <13> <14.4> [1.1] eufm10}{}
\DeclareMathAlphabet{\fraktura}{OT1}{fraktura}{m}{n}

\renewcommand{\section}{\secdef\sct\sect}
\newcommand{\sct}[2][default]{\refstepcounter{section}
\addcontentsline{toc}{section}
{{\tocsection {}{\thesection}{\!\!\!\!#1\dotfill}}{}}
\vspace{0.7cm}
\centerline{ 
\scshape\arabic{section}.\ #1} \nopagebreak \vspace{0.2cm}}
\newcommand{\sect}[1]{
\vspace{0.4cm} \centerline{\large\scshape\rmfamily #1}
\vspace{0.2cm}}

\renewcommand{\subsection}{\secdef\subsct\sbsect}
\newcommand{\subsct}[2][default]{\refstepcounter{subsection}
\addcontentsline{toc}{subsection}
{{\tocsection{\!\!}{\hspace{1.2em}\thesubsection}{\!\!\!\!#1\dotfill}}{}}
\nopagebreak\vspace{0.45\baselineskip} {\flushleft\bf
\thesection.\arabic{subsection}~\bf #1.~}
\\*[3mm]\noindent
\nopagebreak}
\newcommand{\sbsect}[1]{
\vspace{0.1cm}\noindent
\textbf{#1.~}\vspace{0.1cm}}

\renewcommand{\subsubsection}{%
\secdef \subsubsect\sbsbsect}
\newcommand{\subsubsect}[2][default]{%
\refstepcounter{subsubsection} 
\addcontentsline{toc}{subsubsection}{{\tocsection{\!\!}
{\hspace{3.05em}\thesubsubsection}{\!\!\!\!#1\dotfill}}{}}
\nopagebreak
\vspace{0.15\baselineskip} \nopagebreak {\flushleft\rmfamily
\itshape\arabic{section}.\arabic{subsection}.\arabic{subsubsection}
\ \rmfamily #1\/.}\ }
\newcommand{\sbsbsect}[1]{\vspace{0.1cm}\noindent
\rmfamily \itshape
\arabic{section}.\arabic{subsection}.\arabic{subsubsection} \
\sffamily #1\/.\ }

\newif\ifmarek
\marektrue

\ifmarek
\usepackage{mathptmx}

\fi

\newcommand{\bdryin}{\partial^{\text{\rm in}}}
\newcommand{\bdryout}{\partial^{\text{\rm out}}}

\newcommand{\twoeqref}[2]{(\ref{#1}--\ref{#2})}
\newcommand{\LR}{{\text{\rm LR}}}
\newcommand{\UD}{{\text{\rm UD}}}
\newcommand{\cc}{{\text{\rm c}}}
\newcommand{\wt}{\widetilde}

      \usepackage{amssymb}
      \usepackage{amsmath}
      
      \numberwithin{equation}{section}
      \theoremstyle{plain}
      \newtheorem{theorem}{Theorem}[section]
      \newtheorem{lemma}[theorem]{Lemma}
      \newtheorem{corollary}[theorem]{Corollary}
      \newtheorem{proposition}[theorem]{Proposition}

      \theoremstyle{definition}

      \theoremstyle{remark}
      \newtheorem{remark}[theorem]{Remark}

\newenvironment{proofsect}[1]{\vskip0.1cm\noindent{\rmfamily\itshape #1.}}{\qed\vspace{0.15cm}}

      \newcommand{\R}{{\mathbb R}}
      \newcommand{\E}{\mathbb E}
      \newcommand{\e}{\mathrm{e}}
      \renewcommand{\P}{\mathbb P}
      
      \newcommand{\var}{\mathrm{Var}}
      \newcommand{\cov}{\mathrm{Cov}}

      \newcommand{\Z}{\mathbb{Z}}
      
      \newcommand{\N}{\mathbb{N}}

      \newcommand{\eff}{\mathrm{eff}}

      \newcommand{\lb}{\partial_\mathrm{left}}
      \newcommand{\rb}{\partial_\mathrm{right}}
      \newcommand{\ub}{\partial_\mathrm{up}}
      \newcommand{\db}{\partial_\mathrm{down}}
      \renewcommand{\o}{\mathrm{O}}

      \newcommand{\laweq}{\overset{\text{\rm law}}=}
      \newcommand{\cspecial}{\hat{\fraktura c}}

\begin{document}

\title[Random walk on DGFF\hfill]{\large Return probability and recurrence for the random walk driven by two-dimensional Gaussian free field}

\author[\hfill M.~Biskup, J.~Ding, S.~Goswami]{Marek Biskup$^{1,2,*}$, Jian Ding$^{3,\dagger}$ \,and\, Subhajit Goswami$^{4,\dagger}$}
\maketitle

\renewcommand*{\thefootnote}{\fnsymbol{footnote}}

\footnotetext{$^*$ Partially supported by NSF grant DMS-1407558 and GA\v CR project P201/16-15238S}
\footnotetext{$^\dagger$ Partially supported by NSF grant DMS-1455049 and Alfred Sloan fellowship.}

\vspace{-5mm}
\centerline{\textit{$^1$
Department of Mathematics, UCLA, Los Angeles, California}}
\centerline{\textit{$^2$
Center for Theoretical Study, Charles University, Prague, Czech Republic}}
\centerline{\textit{$^3$
Statistics Department, Wharton, University of Pennsylvania, Philadelphia, Pennsylvania}}
\centerline{\textit{$^4$
Institut des Hautes \'Etudes Scientifiques, Bures-sur-Yvette, France}}
\medskip
\centerline{\today}

\vspace{-3mm}
\begin{abstract}
Given any $\gamma>0$ and for $\eta=\{\eta_v\}_{v\in \mathbb Z^2}$ denoting a sample of the two-dimensional discrete Gaussian free field on $\mathbb Z^2$ pinned at the origin, 
we consider the random walk on~$\Z^2$ among random conductances where the conductance of edge
$(u, v)$ is given by $\e^{\gamma(\eta_u + \eta_v)}$. We show that,   for almost every~$\eta$, this random walk  is recurrent and that, with probability tending to~1 as $T\to \infty$, the return probability at time~$2T$  decays as  $T^{-1+o(1)}$.  In addition, we prove a version of subdiffusive behavior by showing that the expected exit time from a ball of radius~$N$ scales as $N^{\psi(\gamma)+o(1)}$ with $\psi(\gamma)>2$ for all~$\gamma>0$.  Our results rely on  delicate control of   the effective resistance for this random network.  In particular, we show  that the effective resistance between two vertices  at  Euclidean distance~$N$  behaves  as~$N^{o(1)}$.
\end{abstract}

\section{Introduction}
\noindent
Let $\eta = \{\eta_v\}_{v \in \Z^2}$  denote a sample of  
the discrete Gaussian free field (GFF) on $\Z^2$ pinned to 0  
at the origin.  Explicitly,  $\{\eta_v\}_{v \in \Z^2}$ is a centered Gaussian process such that 
\begin{equation}
\eta_{ 0 }=0\quad\text{\rm and}\quad\E (\eta_u \eta_v) = G_{\Z^2 \setminus \{ 0 \}}(u,v)
\,\text{\rm\ for all }\, u, v \in \Z^2\,,
\end{equation}
where $G_{\Z^2 \setminus \{ 0 \}}(u, v)$ is  the Green function in $\Z^2\setminus\{0\}$; i.e., 
the expected number of visits to $v$ for the simple random walk on $\Z^2$ started at $u$ and killed 
upon reaching the origin. 
For $\gamma > 0$  and conditional on the sample~$\eta$ of the GFF,  let  $\{X_t\}_{t\geq 0}$ be  a discrete-time Markov chain with  transition  probabilities   
given~by 
\begin{equation}
\label{eq-def-transition}
p_\eta(u, v) := \frac{\e^{\gamma(\eta_v- \eta_u)}}{\sum_{w: |w-u|_1=1}\e^{\gamma(\eta_w-\eta_u)}} \mathbf 1_{|v - u|_1=1} \,,
\end{equation}
 where $|\cdot|_1$ denotes the $\ell^1$-norm on~$\Z^2$. We will write $P_\eta^x$ for the law of the above random walk such that $P_\eta^x(X_0=x)=1$ and use~$E^x_\eta$ to denote the corresponding expectation. We also write~$\P$ for the law of the GFF and use~$\E$ (as above) to denote the expectation with respect to~$\P$.

The transition kernel~$p_\eta$ depends only on the differences $\{\eta_x-\eta_y\colon x,y\in\Z^2\}$ whose law is,  as it turns out,  invariant and ergodic with respect to the  translations  of~$\Z^2$.  (Thanks to the explicit control of correlation decay, the law is actually readily shown to be even strongly mixing.)  The Markov chain $\{X_t\}_{t\ge0}$ is thus an example of a random walk in a stationary random environment. The main conclusion we prove about this random walk is then:

\begin{theorem}
\label{thm_main_spectral}
For each~$\gamma>0$ and each $\delta>0$, 
\begin{equation}
\label{eq-return-probability}
\lim_{T \to \infty}\,\P\Bigl(\e^{-(\log T)^{1/2+\delta}}T^{-1}\leq P_\eta^0(X_{2T} = 0) \leq \e^{ (\log T)^{1/2+\delta}}T^{-1}\Bigr) = 1\,. 
\end{equation}
Furthermore, $\{X_t\}_{t \geq 0}$ is recurrent for $\P$-almost every $\eta$. 
\end{theorem}

The transition probabilities $p_\eta$ are such that the walk prefers to move along the edges where~$\eta$ increases; the walk is thus driven towards larger values of the field. This has been predicted (e.g., in \cite{CLD,CL01}) to result in a subdiffusive behavior. We prove a version of subdiffusivity for the expected exit time from large balls:

\begin{theorem}
\label{thm_main_exit_time}
Let $\tau_{B(N)^\cc}$ denote the first exit time of $\{X_t\colon t\ge0\}$ from the box $B(N):=[-N,N]^2\cap\Z^2$. For each~$\delta>0$, we then have
\begin{equation}
\label{eq-expected-exit-time}
\lim_{N \to \infty}\P\Bigl(N^{\psi(\gamma)} \e^{-(\log N)^{1/2+\delta}} \leq E_\eta^0 \tau_{B(N)^\cc} \leq N^{\psi(\gamma)} \e^{(\log N)^{1/2+\delta}}\Bigr) = 1\,, 
\end{equation}
where
\begin{equation}
\label{E:1.4}
\psi(\gamma):=\begin{cases}
2 + 2(\gamma/\gamma_\cc)^2,\qquad&\text{\rm if }\gamma \le \gamma_\cc := \sqrt{\pi/2},
\\*[1mm]
4\gamma/\gamma_\cc,\qquad&\text{\rm otherwise}.
\end{cases}
\end{equation}
\end{theorem}

The bounds on  the  expected hitting time indicate that $|X_T|$ should scale as $T^{\frac{1}{\psi(\gamma)}+o(1)}$ for large~$T$. Although we expect this to be true, we have so far only been able to prove a corresponding lower~bound:

\begin{theorem}
\label{thm_main_diffusive_exponent}
For $\P$-almost every~$\eta$ and each~$\delta>0$, 
\begin{equation}
\label{eq-diffusive exponent}
\,P_\eta^0\Bigl(|X_{T}| \geq \e^{-(\log T)^{1/2+\delta}}T^{\frac{1}{\psi(\gamma)}}\Bigr)\,\underset{T\to\infty}\longrightarrow\,1 \quad\text{\rm in probability},
\end{equation}
where $\psi(\gamma)$ is as in~\eqref{E:1.4}.
\end{theorem}

\noindent
We note that Theorems~\ref{thm_main_exit_time} and~\ref{thm_main_diffusive_exponent} are consistent with the predictions in \cite{CLD,CL01} for general log-correlated fields. 
In particular, \eqref{eq-diffusive exponent} confirms the prediction for the diffusive exponent of the walk from \cite{CLD,CL01} as a lower bound. The reason why the bounds in~\eqref{eq-expected-exit-time} are not sufficient is that we do not know whether $\tau_{B(N)^\cc}$ scales with~$N$ proportionally to its expectation.  A full proof of subdiffusive behavior thus remains elusive.

The technical approach that makes our analysis possible stems from the following simple rewrite of the transition kernel  achieved by multiplying both the numerator and the denominator in \eqref{eq-def-transition} by $\e^{2\gamma \eta_u}$, 
\begin{equation}
\label{eq-def-transition2}
p_\eta(u, v) = \frac{\e^{\gamma(\eta_v + \eta_u)}}{\sum_{w: |w - u|_1}\e^{\gamma(\eta_w + \eta_u)}} \mathbf 1_{|v - u|_1=1}.
\end{equation}
This represents $\{X_t\}_{t\ge0}$ as a random walk among random conductances, or a Random Conductance Model to which a large body of literature has been dedicated in recent years (see \cite{Biskup11,Kumagai14} for 
reviews). An immediate benefit of the rewrite is that the process is now reversible with respect to the measure $\pi_\eta$ on~$\Z^2$ defined by
\begin{equation}
\label{E:1.8}
\pi_\eta(u):=\sum_{v\colon |u-v|_1=1}\e^{\gamma(\eta_u+\eta_v)}.
\end{equation}
A price to pay is that the conductance  $\e^{\gamma(\eta_u+\eta_v)}$ of edge $(u,v)$ depends on~$\eta$ and not just its gradients, and the law of the conductances is thus not translation invariant. 

As it turns out, the change of the  $N$-dependence  of the expected exit time  $E_\eta^0 \tau_{B(N)^\cc}$  at the critical  value~$\gamma_\cc$  (see Theorem~\ref{thm_main_exit_time}) arises, in its entirety, from the asymptotic
\begin{equation}
\pi_\eta\bigl(B(N)\bigr)
= N^{\psi(\gamma)+o(1)},\qquad N\to\infty\,.
\end{equation}
This is, roughly speaking, because point-to-point effective resistances  in the associated random conductance network  $\Z^2_\eta$  behave, for points at 
distance~$N$, as~$N^{o(1)}$ for every~$\gamma>0$.  The effective resistance and further background on the theory of resistor networks will be discussed in detail in Section~\ref{sec:sec_gen_par_ser}. For now we just recall  that the effective resistance is the voltage needed to induce a unit current between two given points, or sets, in a resistor network.  The precise version of the preceding sentence is then as follows: 

\begin{theorem}
\label{thm-effective-resistance}
Let us regard $B(N):=[-N,N]^2\cap\Z^2$ as a  resistor  network where edge $(u,v)$ has conductance $\e^{\gamma(\eta_u+\eta_v)}$. Let $R_{B(N)_{ \eta}}(u, v)$ denote the effective resistance between~$u$ and~$v$ in network~$B(N)$. Then for each~$\gamma>0$ there are $C,C'\in(0,\infty)$ such that 
\begin{equation}
\label{E:1.10ua}
\max_{u, v \in B(N)}\P\Bigl(R_{B(N)_{ \eta}}(u, v) \geq C\e^{C t \sqrt{\log N}}\Bigr) \leq  C'\e^{-t^2}\,\log N
\end{equation}
holds for each~$N\ge1$ and each~$t\ge0$. 
Moreover, for the corresponding network $\Z^2_\eta$ on all of~$\Z^2$, 
\begin{equation}
\label{E:1.11ua}
\limsup_{N\to\infty}\,\,\frac{\log R_{\Z^2_{\eta}}(0,B(N)^\cc)}{(\log N)^{1/2}(\log\log N)^{1/2}}
<\infty,
\qquad\P\text{\rm-a.s.}
\end{equation}
and, for each~$\gamma>0$, 
also
\begin{equation}
\label{E:1.12ua}
\liminf_{N\to\infty}\,\,\frac{\log R_{\Z^2_{\eta}}(0,B(N)^\cc)}{(\log N)^{1/2}/ (\log\log\log N)^{1/2}}>0.
\qquad\P\text{\rm-a.s.}
\end{equation}
 Both limits are in fact constant $\P$-a.s. 
\end{theorem}

We note that, in light of monotonicity of $N\mapsto R_{\Z^2_{\eta}}(0,B(N)^\cc)$, the bounds in Theorem~\ref{thm-effective-resistance} readily imply  the recurrence claimed in Theorem~\ref{thm_main_spectral}.
The upper bound on the effective resistance in \eqref{E:1.11ua} comes from \eqref{E:1.10ua} while the lower bound \eqref{E:1.12ua} arises, more or less, from its approximate representation using a random walk and invoking Chung's Law of the Iterated Logarithm.

\begin{figure}[t]
\centerline{\includegraphics[width=0.25\textwidth]{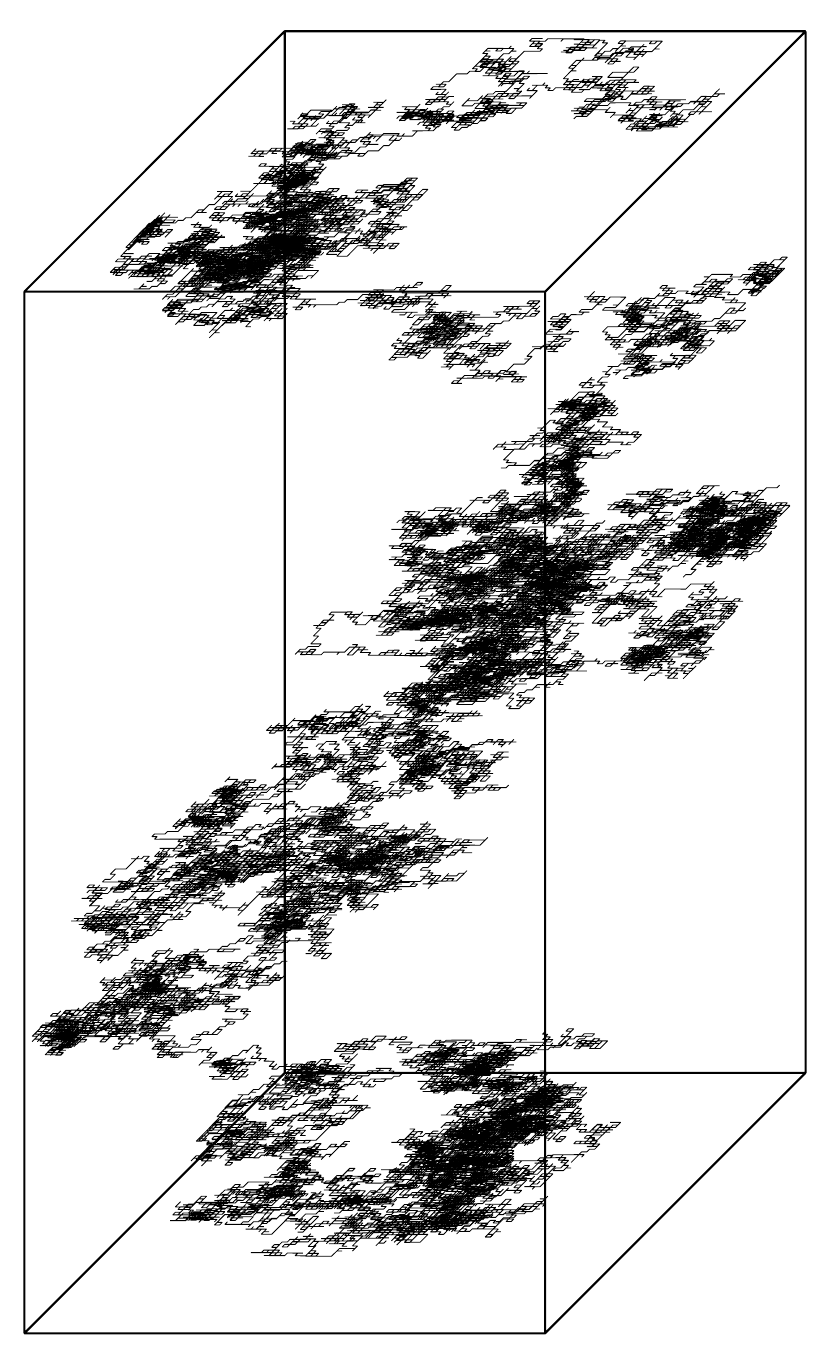}
\includegraphics[width=0.25\textwidth]{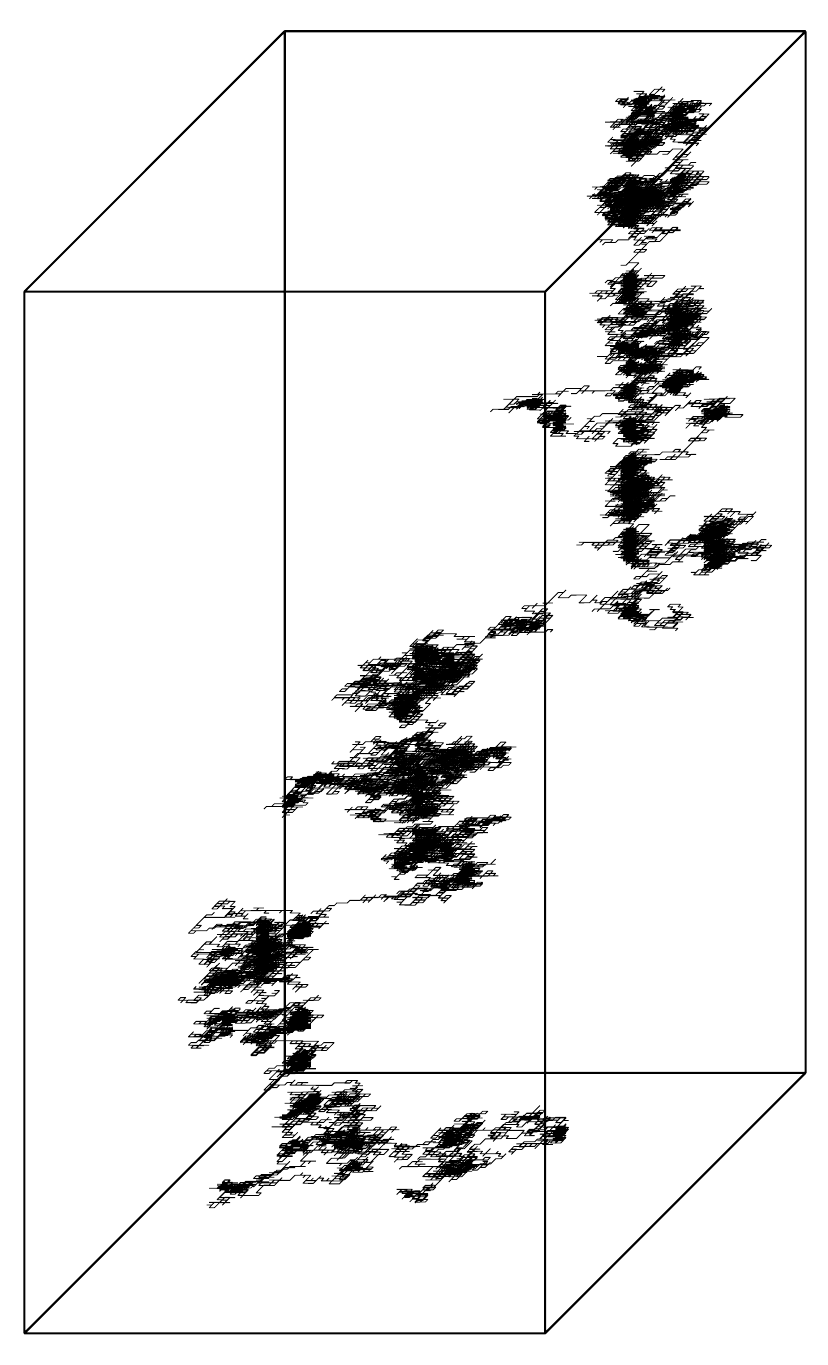}
\includegraphics[width=0.25\textwidth]{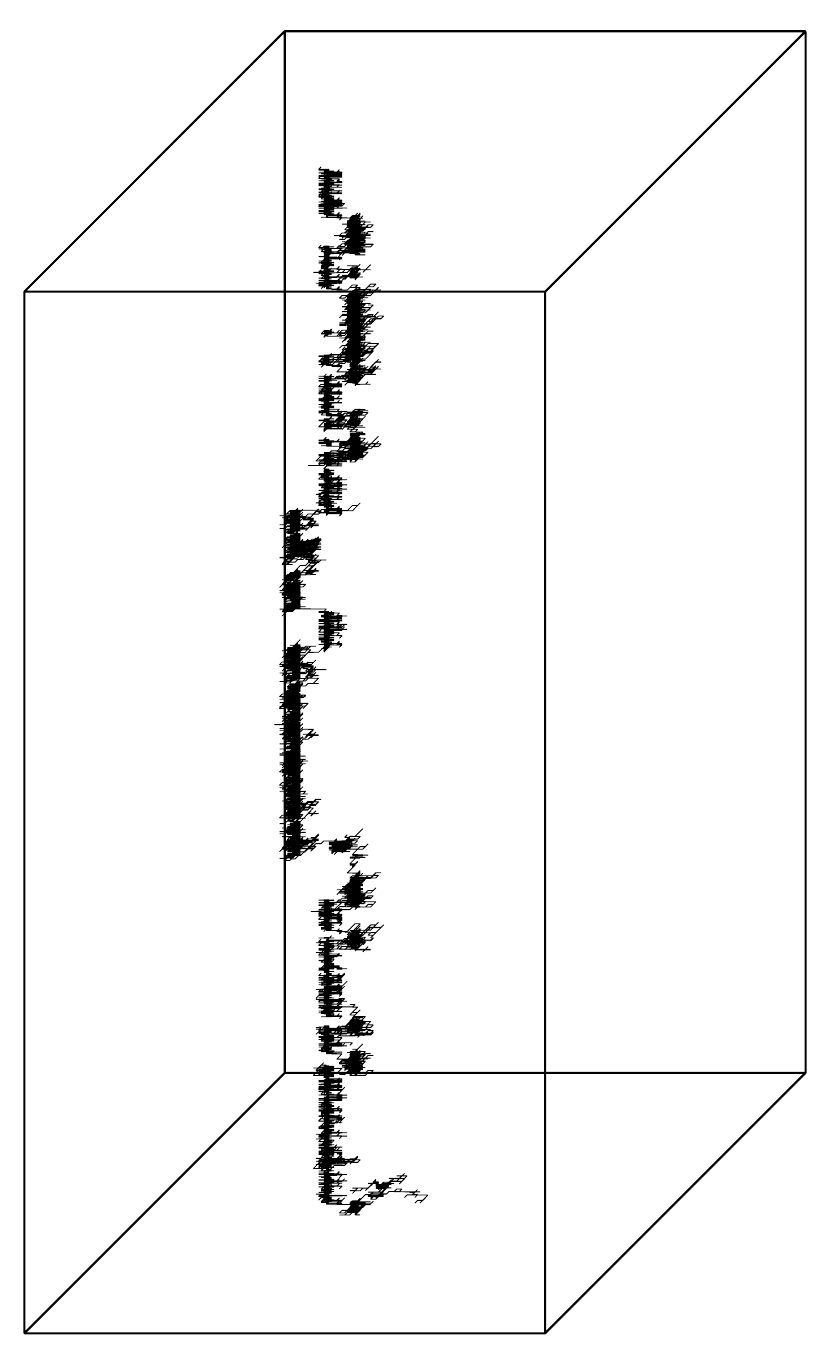}
}
\small 
\vglue0.0cm
\caption{
\label{fig1}
Runs of $100000$ steps of the random walk with transition probabilities~\eqref{eq-def-transition} confined (through reflecting boundary conditions) to a box of side-length~$100$. Labelled left to right, the plots correspond to $\gamma/\gamma_\cc$ equal to~$0.2$, $0.6$ and~$1.2$; time runs upward the vertical axis. Trapping effects are quite apparent.}
\normalsize
\end{figure}

\subsection{Background and related work}
Closely related to our problem is the recently-defined \textit{Liouville Brownian motion} (LBM), which is basically just a time change of the standard Brownian motion by an exponential of the continuum Gaussian free field. The construction of the process was carried out in \cite{GRV13,B14}, with the associated heat kernel constructed in \cite{GRV14}.
 The  spectral dimension (defined as 2 times the exponent for the return probability computed in almost sure sense with respect to the underlying random environment) for LBM was computed  in \cite{RV14} and~\cite{AK}.  Sharp on-diagonal estimates for the LBM heat kernel were proved in~\cite{AK} and some nontrivial bounds for off-diagonal LBM heat kernel were established in \cite{AK,MRVZ14}. 

A random walk naturally associated with LBM is the continuous-time simple symmetric random walk with exponential holding time at~$x$ having parameter $\e^{\beta\eta_x}$  where, in our notation, $\beta:=2\gamma$.  A more natural (albeit qualitatively similar, as far as long-time behavior is concerned) modification is to use $\pi_\eta(x)$ (see~\eqref{E:1.8}) instead of~$\e^{2\gamma\eta_x}$; we will refer to the associated process as the Liouville Random Walk (LRW) below. Formally, this process is a continuous-time Markov chain on $\Z^2$ with generator
\begin{equation}
\label{LWR-gen}
\mathcal L_\eta^{\text{LRW}}f(x):=\frac1{4\pi_\eta(x)}\sum_{y\colon  |x-y|_1=1}\bigl[f(y)-f(x)\bigr].
\end{equation}
Although the LRW is reversible and its formulation using conductances in principle possible, the resulting resistor network is that of the simple random walk and the behavior of the LRW is thus quite different from that studied here. For instance, unlike for the LRW, our random walk moves preferably towards neighbors with a higher potential, emphasizing the trapping effects of the random environment; see Fig.~\ref{fig1}.
The off-diagonal heat kernel computation in \cite{MRVZ14} is also of a different flavor: Our  control  of the return probability  relies crucially on the electric-resistance metric while the off-diagonal LBM heat kernel is expected to be related to the Liouville first passage (Liouville FPP) percolation metric (see \cite{DG16watabiki,DD16}). 

\begin{figure}[t]
\centerline{\includegraphics[width=0.25\textwidth]{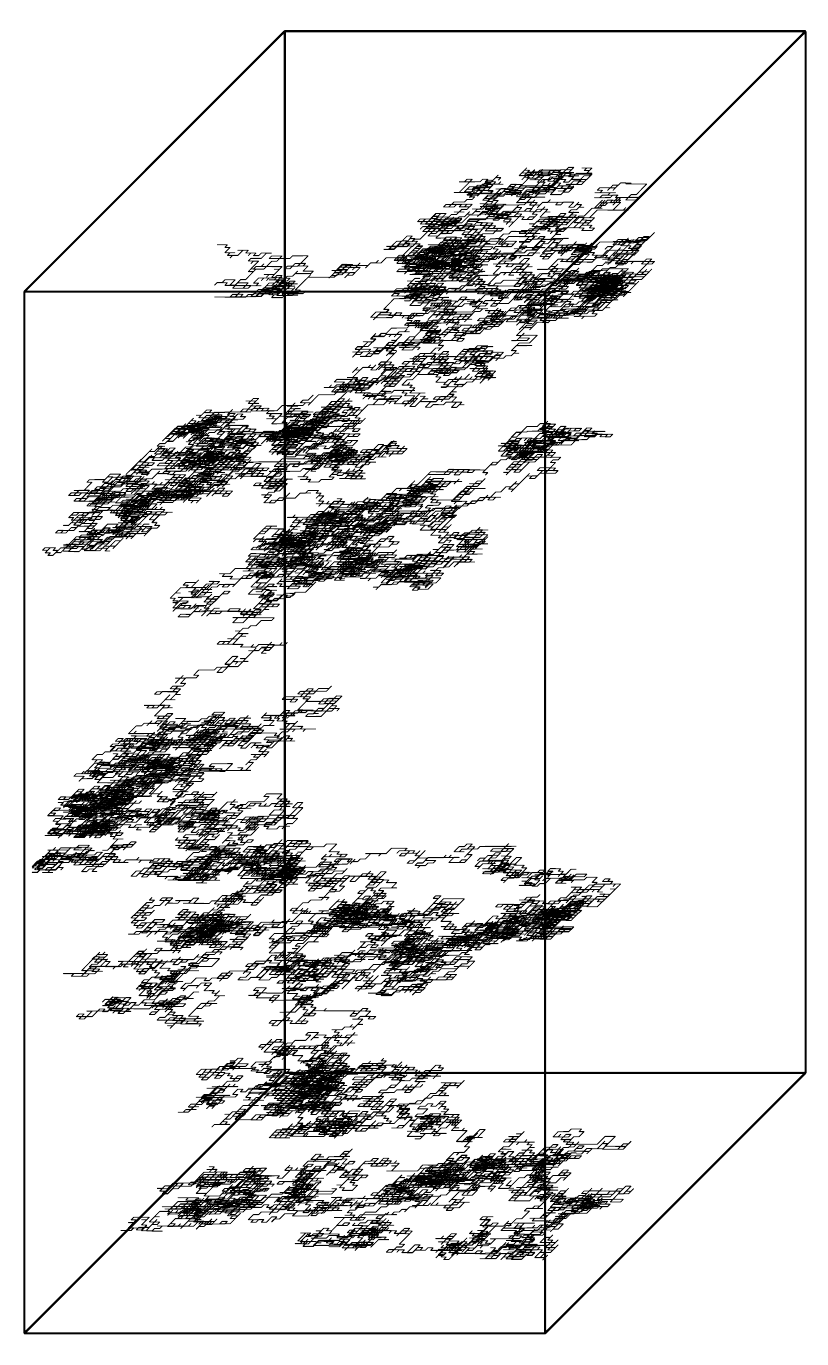}
\includegraphics[width=0.25\textwidth]{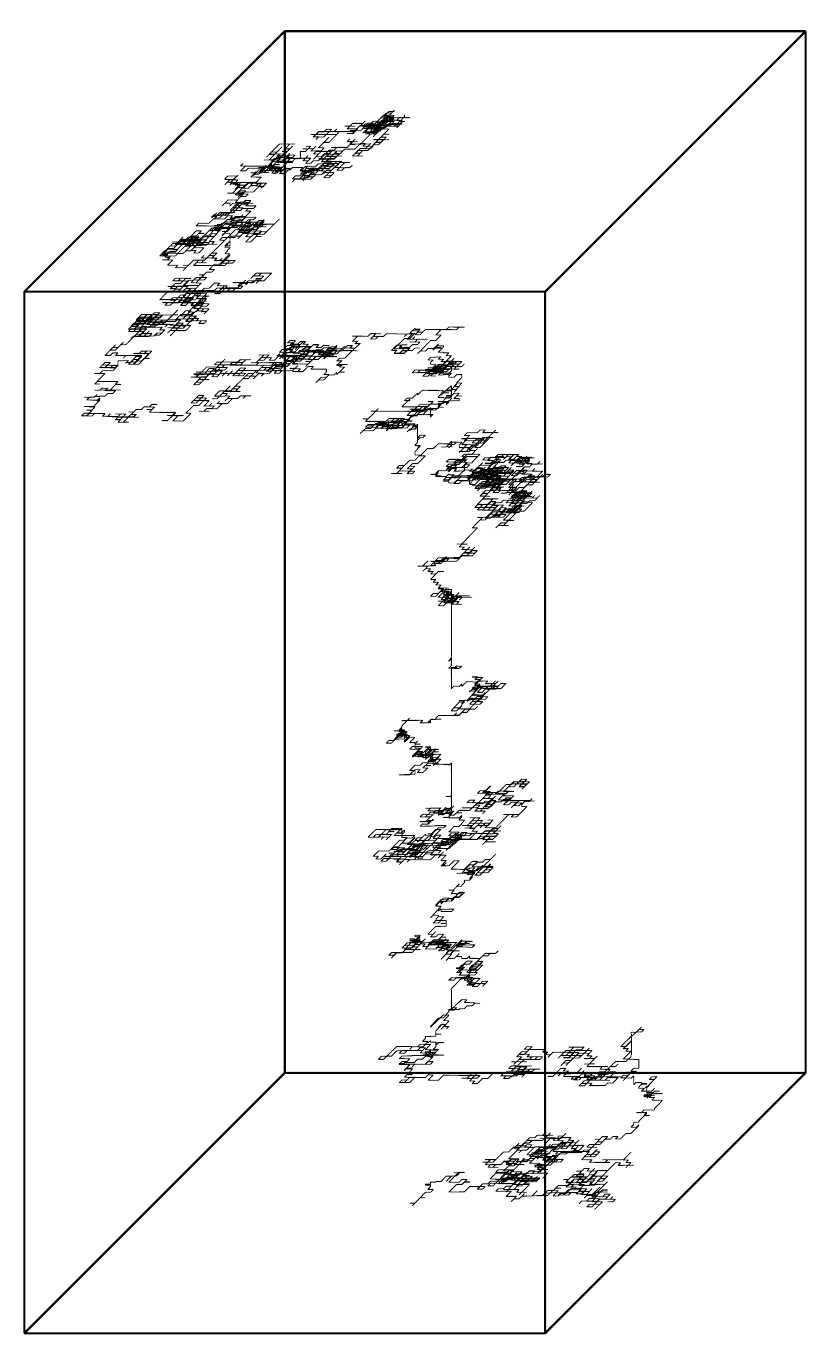}
\includegraphics[width=0.25\textwidth]{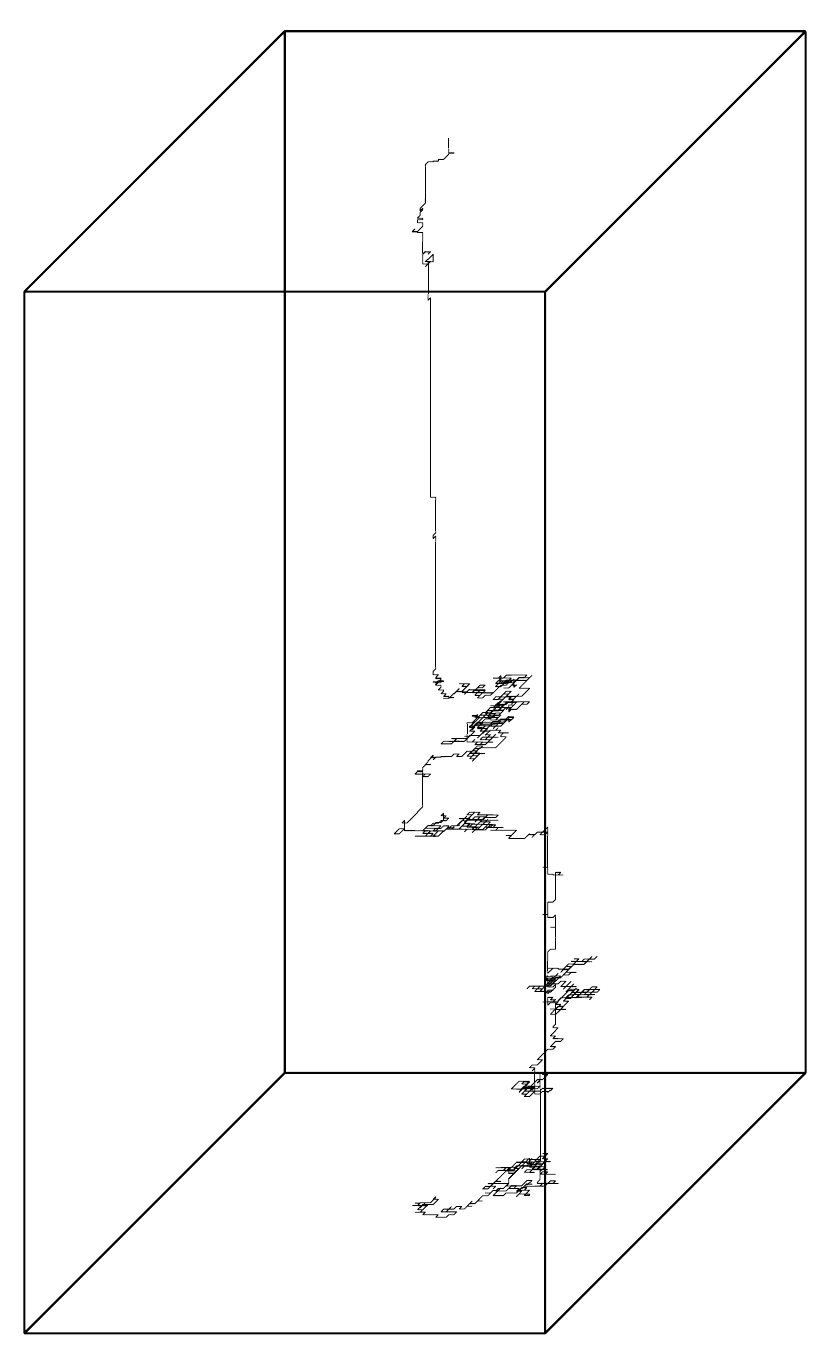}
}
\small 
\vglue0.0cm
\caption{
\label{fig2}
Runs of the Liouville Random Walk (see \eqref{LWR-gen}) for time $100000$ in the same environments, and for the same values of~$\gamma$, as in Fig.~\ref{fig1}. The difference between the two walks is quite obvious, particularly so for larger~$\gamma$.}
\normalsize
\end{figure}

Notwithstanding the above differences, both the LRW and our random walk share the following fact: $x\mapsto\pi_\eta(x)$ defined above is a stationary measure (conditional on $\eta$) for both processes. The same thus applies to any interpolation between the LRW and our random walk; namely, the continuous-time Markov chain with generator
\begin{equation}
\mathcal L_{\eta,\theta}f(x):= \theta\mathcal L_\eta^{\text{LRW}}f(x)+(1-\theta)\sum_{y\colon  |x-y|_1=1}p_\eta(x,y)\bigl[f(y)-f(x)\bigr]
\end{equation}
for any $\theta\in[0,1]$. A question of interest is whether scaling limits of these random walks can eventually be extracted and whether they lead to distinct (for distinct~$\theta$) diffusive processes defined on the background of a continuum Gaussian Free Field.

Another related series of works is on random walks on random planar maps.  This is thanks to  the conjectural relation between LQG and random planar maps (note that part of the conjecture has been established in \cite{MS15,MS16}). Building on  ideas from the theory of circle packings~\cite{BS01},  the authors of \cite{GGN13} proved that the uniform infinite planar triangulation and quadrangulation  are both  almost surely recurrent. In \cite{BC13}, it was shown that the random walk on the uniform infinite planar quadrangulation is sub-diffusive, where an upper bound of $1/3$ on the exponent  was  given while the conjectured exponent is $1/4$.

As noted above, our work relies on estimates of effective resistances, which is a fundamental metric for a graph. Recently, some other metric properties of GFF  have  been studied, including the pseudo-metric defined via the zero-set of the GFF on the metric graph \cite{LW16}, the Liouville FPP metric \cite{DD16,DG16watabiki} (which is roughly the graph distance on the network $\mathbb Z^2_{\eta}$ if we regard edge conductances as passage times) and the chemical distance for the level-set percolation \cite{DL16}. These  studies  reveal different facets of the metric properties of the GFF. In particular, by \cite{DG16watabiki} and the present paper, we see that putting random weights/conductances as exponential of the GFF substantially distorts the graph distance of $\mathbb Z^2$ but  has much less of an effect on  the resistance metric of~$\mathbb Z^2$.

\subsection{A word on proof strategy}
 In light of  the connection between random walks and effective resistances (see, e.g., \cite{LP16} for some background),  the principal step (and the bulk of the paper) is the proof of  Theorem~\ref{thm-effective-resistance}.  This theorem is proved by  a novel combination of  planar and electrostatic  duality, Gaussian concentration inequality and the Russo-Seymour-Welsh theory, as we outline below.

Duality considerations for planar electric networks are quite classical. They invariably boil down to the simple fact that, in a planar network, every harmonic function comes hand-in-hand with its harmonic conjugate. An example of a duality statement, and a source of inspiration for us, is \cite[Proposition 9.4]{LP16}, where it is shown that, for locally-finite planar networks with sufficient connectivity, the wired  effective  resistance across an edge (with the edge removed) is equal to the free  effective  conductance across the dual edge in the dual network (with the dual edge removed). However, the need to deal with more complex geometric settings steered us to develop a version of duality that is phrased in purely geometric terms. In particular, we use that, in planar networks with a bounded degree, cutsets can naturally be associated with paths and \emph{vice versa}. 

The starting point of our proofs is thus a representation of the effective resistance, resp., conductance as a variational minimum of the Dirichlet energy for \emph{families} of paths, resp., cutsets. Although these generalize well-known  upper bounds on these quantities (e.g., the Nash-Williams estimate), we prefer to think of them merely as extensions of the Parallel and Series Laws. Indeed, the variational characterizations are obtained by replacing individual edges by equivalent collections of new edges, connected either in series or parallel depending on the context, and noting that the said  upper  bounds become sharp once we allow for optimization over all such replacements. We refer to Propositions~\ref{prop:fundamental_prop_resistance} and \ref{prop:fundamental_prop_conductance} in Section~\ref{sec:sec_gen_par_ser} for more details.  We note that, in this part, planarity is not needed. 

Another useful fact that we rely on heavily is the symmetry~$\eta \laweq-\eta$ which implies that the joint laws of the conductances are those of the resistances. Using this  along with planarity considerations  we can 
\emph{almost} argue that the law of the effective resistance between the left and right boundaries  of a square centered at the origin  is the same as the law of the effective conductance between the top and bottom boundaries.  The rotation symmetry of~$\eta$ and the (electrostatic) duality between the effective conductance and resistance would then imply that  the law of the effective resistance through a square is  the same as that of its reciprocal value.  Combined with  a  Gaussian concentration inequality (see \cite{ST74,Borell75}), this  would readily show that, for the square of side~$N$,   this effective resistance is typically $N^{o(1)}$. 

However, some care is needed  to make  the ``almost duality''  argument work.  In fact,  we  do not expect  an exact duality of the kind valid  for critical bond percolation on~$\Z^2$ to hold in our case. Indeed, such a duality might for instance  entail that the law of the conductances on a minimal cutset (separating, say, the opposite sides of a square) in the primal network is the same as the law of the resistances on the dual path ``cutting through'' this cutset. Although the GFFs on a graph and its dual are quite closely related (see, e.g., \cite{BK-grad}), we do not see how this property can possibly be true. 
Notwithstanding, we are more than happy to work with just an approximate duality which, as it turns out, requires  only  a uniform bound on the \emph{ratio} of resistances of neighboring edges. This ratio would be unmanageably too large if  applied the duality argument to the network based on  the GFF  itself.  For this reason, we invoke a  decomposition of GFF (see Lemma~\ref{lem:field_decomp_smooth}) into a sum of two independent fields, one of which has small variance and the other is a highly smooth field.  We then apply the  approximate duality to the network derived from  the  smooth field, and we argue that the influence from the other field is small since it has small variance.

 We have so far explained only how to  estimate the effective resistances between the boundaries of a \emph{square}.  However,  in order to prove our theorems, we need to estimate effective resistances between vertices, for which a crucial ingredient is an estimate of the effective resistances between the two shorter sides of a \emph{rectangle}. Questions of this type fall into the framework of the Russo-Seymour-Welsh (RSW) theory.  This is  an important technique
in planar statistical physics, initiated in \cite{Russo78,SW78,Russo81}
 with the aim to prove uniform positivity of the probability of a crossing  of a rectangle  
in critical Bernoulli percolation. Recently, the theory has been
 adapted to include FK percolation, see e.g. \cite{DHN11,BDC12,DST15}, and, in \cite{Tassion14}, also Voronoi percolation. 
In fact, the beautiful method in \cite{Tassion14} is widely applicable
to percolation problems satisfying the FKG inequality, mild symmetry
assumptions, and weak correlation between well-separated regions.
For example, in \cite{DRT16}, this method was used to give a simpler
proof of the result of \cite{BDC12}, and in \cite{DMT},  a RSW theorem was proved  for the crossing probability of level sets of  the 
planar~GFF.

\begin{figure}[t]
\centerline{\includegraphics[width=0.42\textwidth]{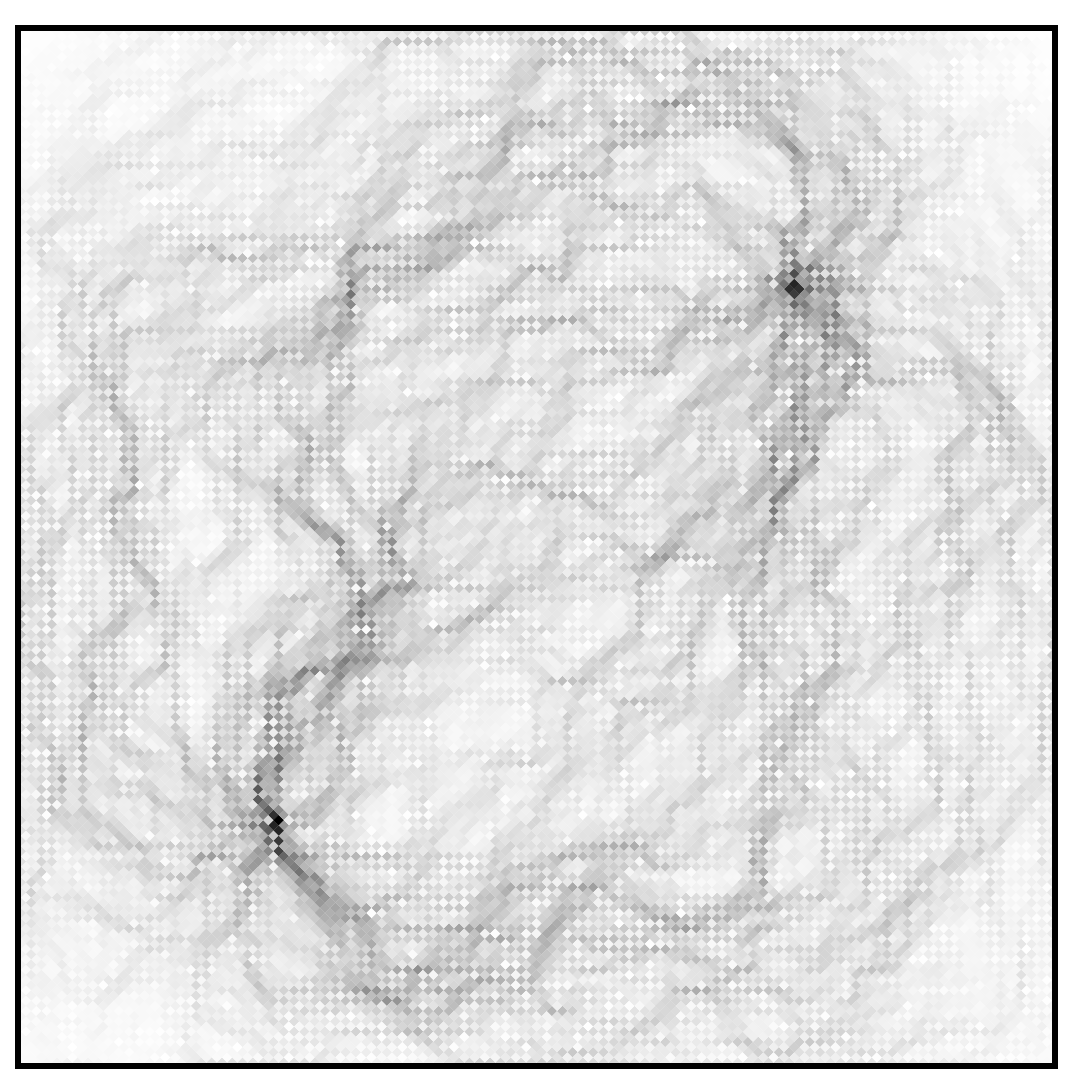}
\hglue3mm\includegraphics[width=0.42\textwidth]{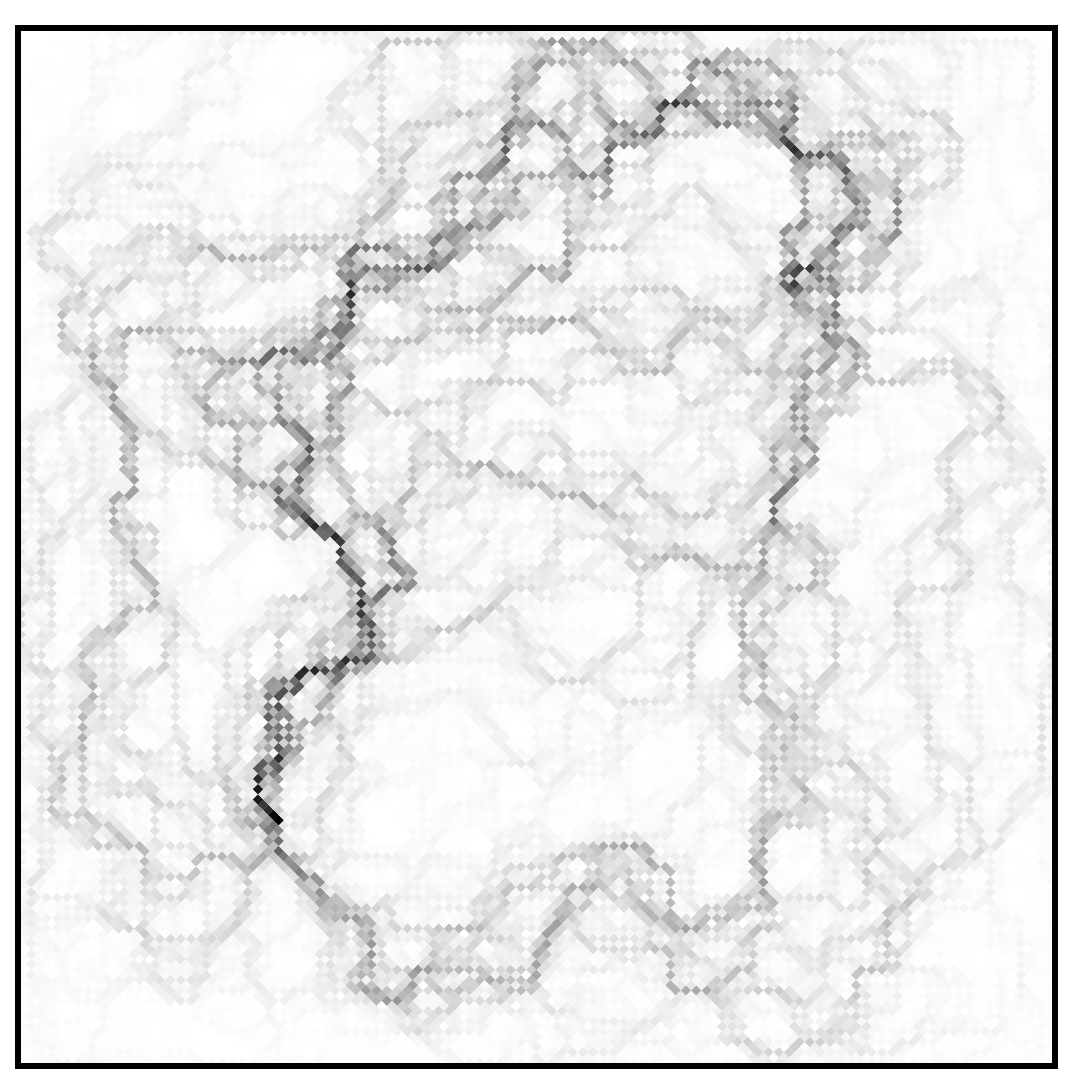}
}
\small 
\vglue0.0cm
\caption{
\label{fig3}
The current flow realizing (through Thomson's Principle) the effective resistance between two points in the network associated with a sample of the GFF in a square of side $N:=100$ and~$\gamma=0.2\gamma_\cc$ (left) and $\gamma=0.5\gamma_\cc$ (right). The source/sink of the current lie on the diagonal of the square (marked by the dark spots in the left figure). The intensity of the shade increases with the value of the current. }
\normalsize
\end{figure}

Our RSW proof is \emph{hugely} inspired by \cite{Tassion14}, with the novelty of incorporating the  (resistance)  metric rather than merely considering connectivity. We remark that in a recent work \cite{DD16}, a RSW result was established for the Liouville FPP metric, again 
inspired by \cite{Tassion14}. It is fair to say that the RSW result in the present paper is less complicated than that in \cite{DD16}, for the reason that we have the approximate duality in our context which was not 
available in \cite{DD16}. However, our RSW  proof has its own subtlety since, for instance, we need to consider 
 crossings by whole collections  of paths simultaneously. The RSW proof is carried out in Section~\ref{sec:RSW}. 
 Once the effective resistances are under control, we move on to the proof of the  results on random walks. The upper bound on the return probability  is proved in  Section~\ref{subsec:spectral_dim_upper}  using the methods drawn from  \cite{Kumagai14}. The lower 
bound on the return probability is more subtle as it requires showing that the effective resistance from $0$ to~$v$ in $B(N)$ is bounded by the sum of the resistances from $0$ to~$\partial B(N)$ and from~$v$ to~$\partial B(N$). This amounts to bounding a \emph{difference} of effective resistances, which is not  immediate from the estimates obtained thus far. 

 We approach this by invoking a concentric decomposition of the GFF along a sequence of annuli, which permits representing of the typical value of the resistance as an exponential of a random walk. The Law of the Iterated Logarithm then shows that the natural fluctuations of the effective resistance (which are of order $\e^{O(\sqrt{\log N})}$) can be beaten in at least one of the annuli. These key steps are the content of Proposition~\ref{prop-5.8} and  Lemma~\ref{lemma-5.11}. As an immediate consequence, we then get recurrence and, in fact, also the bounds in Theorem~\ref{thm-effective-resistance}.

\begin{figure}[t]
\centerline{\includegraphics[width=0.45\textwidth]{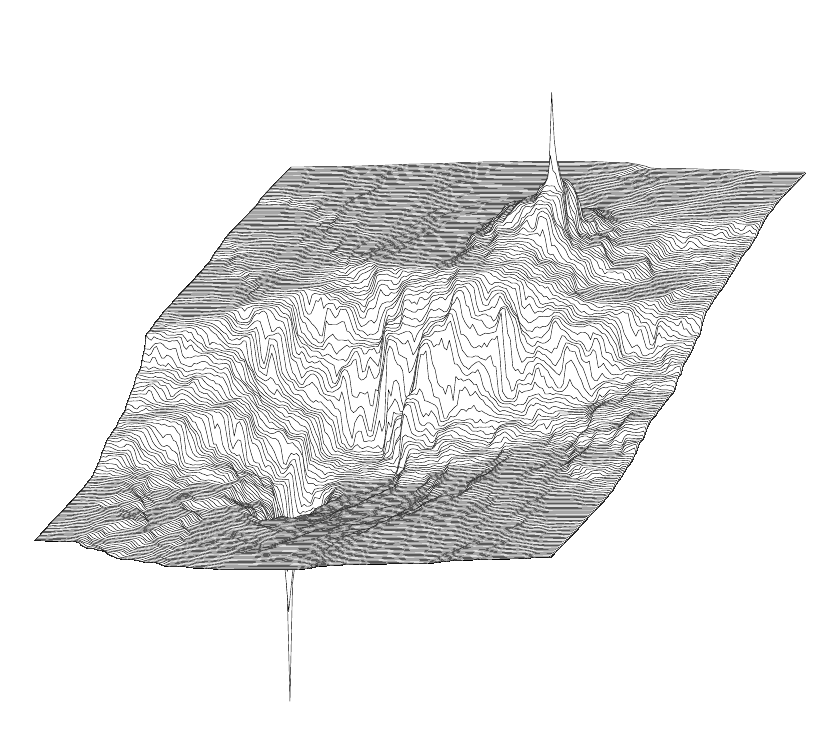}
\hglue-8mm\raise3pt\hbox{\includegraphics[width=0.42\textwidth]{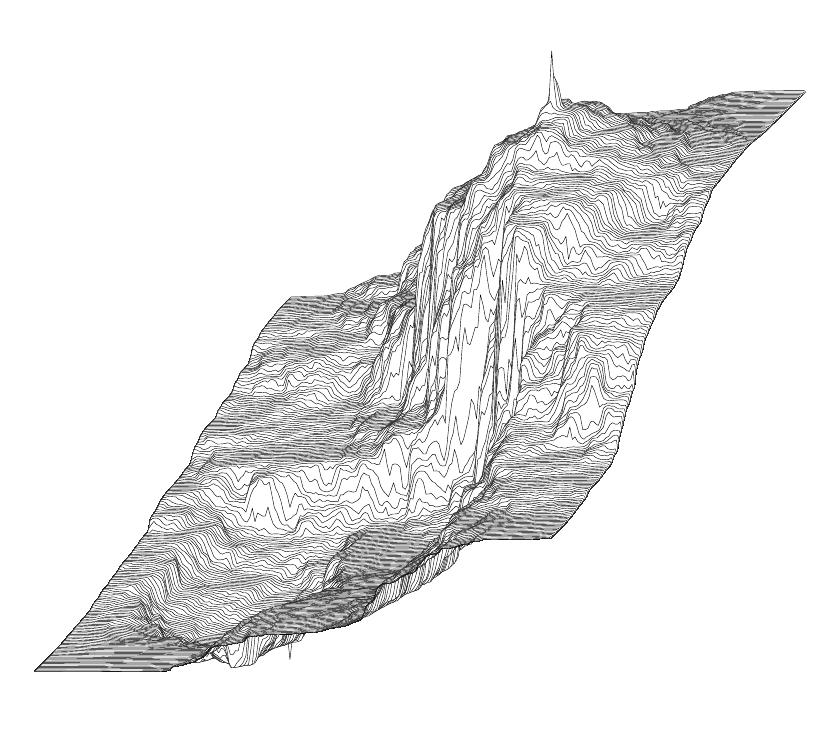}}
}
\small 
\vglue0.0cm
\caption{
\label{fig4}
Voltage profiles for the same geometric setting as in Fig.~4 but for two different samples of GFF at $\gamma=0.3\gamma_\cc$. Notice that the profile represents the probability that a random walk hits the highest point before hitting the lowest point in this  profile.  The plots indicate presence of macroscopic ``barriers'' where this probability undergoes a dramatic change. }
\normalsize
\end{figure}
%
\subsection{Discussions and future directions}
We feel that our method of estimating effective resistances provides a novel framework which may have applications in other planar random media. In fact, from our proofs we should be able to see that our method can be adapted to some other log-correlated Gaussian fields such as those considered in \cite{Madaule15}. We refrain ourselves from doing so, for the reason that we do not  yet  know  how to characterize the class  of log-correlated Gaussian fields  with subpolynomial (i.e., $N^{o(1)}$-like)  growth of the effective resistances.   

One important, and perhaps  less conspicuous,  ingredient  of our proofs  is the estimate of the effective resistance  by means of  the Gaussian concentration inequality. When the underlying random media is not  a function of 
a Gaussian process,  a derivation of such a concentration inequality seems to be a challenge.   
 A natural class of non-Gaussian models where one should try to prove an analogue of Theorem~\ref{thm-effective-resistance} is that of  gradient fields  with uniformly convex interactions.   Indeed, there the required concentration is implied by the Brascamp-Lieb inequality.

 Concerning our  future goals for the problem at hand, our first attempt will aim at the computation of the spectral dimension (which amounts to an almost sure version of \eqref{eq-return-probability}) and an upper bound on the diffusive exponent matching the lower bound in Theorem~\ref{thm_main_diffusive_exponent}. Our ultimate goal is to prove existence of an appropriate scaling limit of the whole problem. This applies not only to the walk itself, but also to the resistance metric as well as the associated current and voltage configurations; see Fig.~\ref{fig3} and~\ref{fig4} for illustrations  and the recent review~\cite[Chapter~16]{Biskup-PIMS} for more specific questions.

\subsection{Acknowledgements}
The work of M.B.\ has been partially supported by NSF grants DMS-1407558 and DMS-1712632 and GA\v CR project P201/16-15238S. The work of J.D.\ and S.G.\ has been partially supported by NSF grant DMS-1455049 and Alfred Sloan fellowship.
The authors  wish to  thank Takashi Kumagai, Hubert Lacoin, R\'emi Rhodes,  Steve Lalley,  Vincent Vargas and Ofer Zeitouni for helpful discussions.  We are also grateful to two anonymous referees for their numerous valuable suggestions that improved greatly the presentation of this work.  The first draft of the paper was completed when both JD and SG were at the University of Chicago.

\section{Generalized parallel and series laws}
\label{sec:sec_gen_par_ser}\noindent
As noted above, our asymptotic statements on the random walk hinge on estimates of effective resistance between various sets in~$\Z^2$. These will in turn rely crucially on a certain duality between the effective resistance and the effective conductance which will itself be based on the distributional equality of~$\eta$ with~$-\eta$.  The  exposition of  our  proofs thus starts with general versions of these duality statements.  The results in this section hold for general networks;  planarity is not required. They  can be viewed as refinements of \cite[Proposition 9.4]{LP16} and are therefore of general interest as well.

\subsection{ Variational characterization of effective resistance}
Let $\mathfrak G$ be a finite,  unoriented,  connected graph where each edge~$e$ is equipped with a resistance $r_e\in\R_+$, where $\R_+$ denotes the set of positive reals.   We will use~$\mathfrak G$ to denote both the corresponding network  as well as the underlying graph. Let~$V(\mathfrak G)$ and~$E(\mathfrak G)$ respectively denote the set of 
vertices and edges of $\mathfrak G$. We assume for simplicity that~$\mathfrak G$ has no self-loops although 
 we allow distinct vertices to be connected by  multiple edges.  For the purpose of counting we identify the two orientations of each edge; $E(\mathfrak G)$ thus includes both orientations as one edge.  

Two edges~$e$ and~$e'$ of~$\mathfrak G$ are said to be \emph{adjacent} to each other, denoted as $e \sim e'$, if they share at least 
one endpoint. Similarly a vertex~$v$ and an edge~$e$ are adjacent, denoted as $v \sim e$, if $v$ is an endpoint 
of the edge $e$.
A \emph{path} $P$ is a sequence of  edges  of $\mathfrak G$ such that any two successive  edges  are 
adjacent.  We say that a path is edge self-avoiding if every edge appears in~$P$ at most once.  We also use $P$ to denote the subgraph  of~$\mathfrak G$  induced by the edge set of~$P$.

For $u, v\in V(\mathfrak G)$, a \textit{flow}~$\theta$ from~$u$ to~$v$ is an assignment of a number~$\theta(x,y)$ to each \textit{oriented} edge $(x,y)$ such that $\theta(x,y)=-\theta(y,x)$ and $\sum_{y\colon y\sim x}\theta(x,y)=0$ whenever~$x\ne u,v$. The \textit{value of the flow}~$\theta$ is then the number $\text{val}(\theta):=\sum_{y\colon y\sim u}\theta(u,y)$; a unit flow then has this value equal to one. With these notions in place, the effective resistance $R_{\mathfrak G}(u, v)$ between $u$ and $v$
is defined by
\begin{equation}
\label{eq-resistance-flow}
R_{\mathfrak G}(u, v) := \inf_{\theta} 
\,\,\sum_{e\in E(\mathfrak G)} r_e\,\theta_\e^2\,,
\end{equation}
where the infimum  (which is achieved because~$\mathfrak G$ is finite)  is over all unit flows from $u$ to $v$.  Note that we sum over each edge~$e\in E(\mathfrak G)$ only once, taking advantage of the fact that $\theta_e$ appears in a square in this, and later expressions.

Recall that a multiset of elements of~$A$ is a set of pairs $\{(a,i)\colon i=1,\dots,n_a\}$ for some $n_a\in\{0,1,\dots\}$  for each~$a\in A$.  We have the following alternative characterization of $R_{\mathfrak G}(u, v)$:

\begin{proposition}
\label{prop:fundamental_prop_resistance}
 Let  $\mathfrak P_{u, v}$ denote the set of 
 all multisets   of  edge self-avoiding  paths from~$u$ to~$v$.
Then
\begin{equation}\label{eq-general-parellel}
R_{\mathfrak G}(u, v) = \inf_{\mathcal P \in \mathfrak P_{u, v}} \,\, \inf_{\{r_{e, P}\colon e\in 
E(\mathfrak G),\, P\in \mathcal P\} \in \mathfrak R_\mathcal P}\Bigl(\,\sum_{P\in \mathcal P} \frac{1}
{\sum_{e\in P} r_{e, P}}\Bigr)^{-1}\,,
\end{equation}
where $\mathfrak R_\mathcal P$ is the set of all  assignments  $\{r_{e, P}: e\in 
E(\mathfrak G),P\in \mathcal P\} \in \mathbb R_+^{E(\mathfrak G)\times \mathcal P}$ such that 
\begin{equation}\label{eq-r-e}
\sum_{P\in \mathcal P} \frac{1}{r_{e, P}} \leq \frac{1}{r_e} \,\text{\rm\ for all }\, e \in E(\mathcal 
G)\,.
\end{equation}
The infima in~\eqref{eq-general-parellel} are (jointly) achieved.
\end{proposition}

\begin{proofsect}{Proof}
Let~$R^\star$ denote the right hand side of \eqref{eq-general-parellel}. We will first prove $R_{\mathfrak G }(u, v) \leq R^\star$. 
Let thus $\mathcal P \in \mathfrak P_{u,v}$ and $\{r_{e, P}\colon e\in E, P\in \mathcal P\} \in \mathfrak R_\mathcal P$ subject to \eqref{eq-r-e} be given. We will view each edge~$e$ in~$\mathfrak G$ as a \emph{parallel} of a collection of edges $\{e_P\colon P\in \mathcal P\}$ where the resistance on~$e_P$ is $r_{e,P}$ and, if the inequality in~\eqref{eq-r-e} for edge~$e$ is strict, introduce a dummy edge $\tilde e$ with resistance $r_{\tilde e}$ such that $1/ r_{\tilde e} = 1/r_e-\sum_{P\in\mathcal P}1/r_{e,P}$. In this new network,~$\mathcal P$ can be identified with a collection of \emph{disjoint} paths where (by the series law) each path $P\in \mathcal P$ has total resistance $\sum_{e\in P} r_{e, P}$. The parallel law guarantees 
\begin{equation}
R_{\mathfrak G}(u, v) \leq \Bigl(\,\sum_{P\in \mathcal P} \frac{1}{\sum_{e\in P} 
r_{e, P}}\Bigr)^{-1}
\end{equation}
which proves $R_{\mathfrak G}(u, v) \leq R^\star$ as desired.

Next, we turn to proving that $R_{\mathfrak G} (u, v) \geq R^\star$ and that the  infima  in~\eqref{eq-general-parellel} are achieved.

Let~$\theta^\star$ be the unit flow achieving the infimum in \eqref{eq-resistance-flow}. We will run an algorithm that inductively identifies a sequence of flows~$\theta_k$ (not necessarily of unit value) and paths~$P_k$ from~$u$ and~$v$. A key fact to note is that, if~$e\mapsto \theta(e)$ is a flow from~$u$ to~$v$ with~$\text{val}(i)>0$, then there is an edge self-avoiding path~$P$ from~$u$ to~$v$ such that $\theta(e)>0$ for each~$e\in P$ oriented in the direction of~$P$. This follows by the fact that, at all vertices except the source and the sink, having an edge with a positive incoming flow forces an edge with a positive outgoing flow.

The algorithm runs as follows. INITIATE by setting $\theta_0:=\theta^\star$. If~$\text{val}(\theta_{k-1})=0$ then STOP, else use the above observation to find a simple path~$P_k$ from~$u$ to~$v$ where $\theta_{k-1}(e)>0$ for each edge~$e\in P_k$ oriented in the direction of the path. Then set $\alpha_k:=\min_{e\in P_k}|\theta_{k-1}(e)|$, let
\begin{equation}
\theta_k(e):=\theta_{k-1}(e)-\alpha_k\, \text{sign}(\theta_{k-1}(e))1_{\{e\in P_k\}}
\end{equation}
and, noting that~$\theta_k$ is a flow from~$u$ to~$v$ with~$\text{val}(\theta)\ge0$, REPEAT. 

As $\{e\in E(\mathfrak G)\colon \theta_k(e)\ne0\}$ is strictly decreasing in~$k$, the algorithm will terminate after a finite number of steps. Also $k\mapsto |\theta_k(e)|$ and~$k\mapsto\text{val}(\theta_k)$ are decreasing, and so 
\begin{equation}
\label{E:12.55}
\forall e\in E(\mathfrak G)\colon\,\,\,
\sum_{k\colon e\in P_k}\alpha_k \le |\theta^\star(e)|\qquad\text{and}\qquad\sum_k\alpha_k = \text{\rm val}(\theta^\star)=1.
\end{equation}
Set $r_{e,P_k}:=|\theta^\star(e)| r_e/\alpha_k$ and compute
\begin{equation}
\begin{aligned}
R_\eff(u,v)&=\sum_{e\in E(\mathfrak G)}r_e\,\theta^\star(e)^2
\ge\sum_{e\in E(\mathfrak G)}\,r_e|\theta^\star(e)|\sum_{k\colon e\in P_k}\alpha_k
\\&=\sum_{e\in E(\mathfrak G)}\,\sum_{k\colon e\in P_k}r_{e,P_k}\alpha_k^2
=\sum_k\,\alpha_k^2\Bigl(\sum_{e\in P_k}r_{e,P_k}\Bigr).
\end{aligned}
\end{equation}
Denoting the quantity in the large parentheses by~$R_k$, of all positive~$\alpha_k$'s subject to the equality in \eqref{E:12.55}, the right-hand side is minimized by~$\alpha_k:=\frac{1/R_k}{\sum_j 1/R_j}$. As \eqref{E:12.55} shows $\sum_k 1/r_{e,P_k}\le 1/r_e$ for each~$e\in E(\mathfrak G)$, the claim follows. 
\end{proofsect}

A slightly augmented version of the above proof in fact yields:

\begin{proposition}
\label{fundamental_cor_resistance}
Let $\mathfrak{T}_{u, v}$ be the set of all  multisets   of edges of~$\mathfrak G$ that, if considered as a graph on $V(\mathfrak G)$, contain a path between~$u$ and~$v$. 
Then
\begin{equation}
\label{eq-general-parellel2}
R_{\mathfrak G}(u, v) = \inf_{\mathcal T \in \mathfrak T_{u, v}}  \inf_{\{r_{e, T}: e\in 
E(\mathfrak G), 
T\in \mathcal T\} \in \mathfrak R_{\mathcal T}}\Bigl(\,\sum_{T \in \mathcal T} \frac{1}{\sum_{e\in T} 
r_{e, T}}\Bigr)^{-1}\,,
\end{equation}
where $\mathfrak R_{\mathcal T}$ is the set of all  assignments  $\{r_{e, T}\colon e\in 
E(\mathfrak G), T \in \mathcal T\} \in \mathbb R_+^{E(\mathfrak G) \times \mathcal T}$ such that 
\begin{equation}
\label{eq-r-e2}
\sum_{T \in \mathcal T} \frac{1}{r_{e, T}} \leq \frac{1}{r_e} \,\text{\rm\ for all }\, e \in E(\mathcal 
G)\,.
\end{equation}
 The infima are jointly achieved for~$\mathcal T$ being a subset of~$\mathfrak P_{u,v}$.
\end{proposition}

\begin{proofsect}{Proof}
Let~$R^\star$ denote the right-hand side of 
\eqref{eq-general-parellel2}. Obviously, $\mathfrak P_{u,v}\subseteq\mathfrak T_{u,v}$ so restricting the first infimum to $\mathcal T\in\mathfrak P_{u,v}$, Proposition~\ref{prop:fundamental_prop_resistance} shows 
$R_{\mathfrak G}(u, v)\ge R^\star$. (This will also ultimately give that the minimum is achieved over 
collections of paths.)  To get $R_{\mathfrak G}(u, v)\le R^\star$, let us consider an assignment $\{r_{e, T}\colon e\in E(\mathfrak G), T \in \mathcal T\}$ satisfying \eqref{eq-r-e2}. For each $T \in \mathcal T$, let $P_T$ denote an arbitrarily chosen  edge self-avoiding  path 
between~$u$ and~$v$ formed by edges in $T$. Then, defining $r_{e, P_T} := r_{e, T}$ for each $T \in \mathcal T$, we find that the assignment $\{r_{e, P_T}\colon e\in E(\mathfrak G), T \in \mathcal T\}$ 
satisfies \eqref{eq-r-e}. Now the claim follows from the simple observation that $\sum_{e \in P_T} r_{e, p_T} 
\leq \sum_{e \in P_T} r_{e, T}$. 
\end{proofsect}

\subsection{ Variational characterization of effective conductance}
An alternative way to approach an electric network is using conductances.  We write $c_e := 
1/r_e$ for the edge conductance on $e$, and define the effective conductance between $u$ and $v$
by 
\begin{equation}
\label{eq-conductance-minimization}
C_{\mathfrak G}(u, v) := \inf_F \,\sum_{e\in E(\mathfrak G)} c_e\,\bigl[F(e_+) - F(e_-)\bigr]^2\,,
\end{equation}
where $e_\pm$ are the two endpoints of the edge $e$  (in some \emph{a priori} orientation)  and the  infimum  is over all functions $F\colon 
V\to \mathbb R$ satisfying $F(u)=1$ and $F(v) = 0$.  The infimum is again achieved by the fact that~$\mathfrak G$ is finite.
The fundamental electrostatic duality  is then expressed as
\begin{equation}
\label{E:2.14}
C_{\mathfrak G}(u, v) = \frac1{R_{\mathfrak G}(u, 
v)}
\end{equation}
and our aim is to capitalize on this relation further by exploiting the geometric duality between paths and cutsets. Here we say that 
a set of edges $\pi$ is a cutset between $u$ and~$v$ (or that $\pi$ separates~$u$ from~$v$) if 
each path from $u$ to~$v$   uses an edge in~$\pi$.  

\begin{proposition}
\label{prop:fundamental_prop_conductance}
 Let $\varPi_{u, v}$ denote the set of all finite collections of cutsets between~$u$ and~$v$. Then 
\begin{equation}\label{eq-general-Nash-William}
C_{\mathfrak G}(u, v) = \inf_{\Pi \in { \varPi}_{u, v}}  \inf_{\{c_{e, \pi}\colon e\in E(\mathfrak G), \pi\in \Pi\} \in \mathfrak C_{\Pi}}\Bigl(\,\sum_{\pi\in \Pi} \frac{1}{\sum_{e\in \pi} c_{e,
\pi}}\Bigr)^{-1}\,,
\end{equation}
where $\mathfrak C_{\Pi}$ is the set of all  assignments  $\{c_{e, \pi} \colon e\in E(\mathfrak G), \pi \in \Pi\} \in \mathbb R_+^{E(\mathfrak G)\times \Pi}$ such that 
\begin{equation}
\label{eq-c-e}
\sum_{\pi \in \Pi} \frac{1}{c_{e, \pi}} \leq \frac{1}{c_e} \,\text{\rm\ for all }\, e \in E(\mathfrak G)\,.
\end{equation}
 The infima in~\eqref{eq-general-Nash-William} are (jointly) achieved.
\end{proposition} 

\begin{proofsect}{Proof}
The proof is structurally similar to that of Proposition~\ref{prop:fundamental_prop_resistance}.
Denote by $C^\star$ the quantity on the right hand side of \eqref{eq-general-Nash-William}. We will 
first prove $C_{\mathfrak G }(u, v) \leq C^\star$. Pick $\Pi \in 
\varPi$ and $\{c_{e, \pi}\colon e\in E(\mathfrak G), \pi \in \Pi\} \in \mathfrak C_{\Pi}$ subject to 
\eqref{eq-c-e}.  Now  view each edge~$e$ as a \emph{series} of a 
collection of edges $\{e_\pi: e\in \pi, \pi\in \Pi\}$ where the conductance on~$e_\pi$ is $c_{e, 
\pi}$ and, if the inequality in~\eqref{eq-c-e} is strict, introduce a dummy edge~$\tilde e$ with conductance $c_{\tilde e}$ such that $1/c_{\tilde e}=1/c_e-\sum_{\pi\in\Pi}1/c_{e,\pi}$. In this new network,~$\Pi$ can be identified with a collection of \emph{disjoint} cutsets, where the 
cutset $\pi \in \Pi$ has total conductance $\sum_{e\in \pi} c_{e, \pi}$. The 
Nash-Williams Criterion then shows
\begin{equation}
C_{\mathfrak G}(u, v) \leq 
\Bigl(\,\sum_{\pi\in \Pi} \frac{1}{\sum_{e\in 
\pi} c_{e, \pi}}\Bigr)^{-1}
\end{equation}
thus proving $C_{\mathfrak G}(u, v) \leq C^\star$ as desired.

Next, we turn to proving $C_{\mathfrak G} (u, v) \geq C^\star$ and that the infima in~\eqref{eq-general-Nash-William} are attained. Let $F^\star$ be a 
function that achieves the infimum in~\eqref{eq-conductance-minimization}. This function is discrete harmonic in the sense that $\mathcal L F^\star(x)=0$ for~$x\ne u,v$, where
\begin{equation}
\mathcal L f(x):=\sum_{y\colon y\sim x}c_{(x,y)}\,\bigl[f(y)-f(x)\bigr].
\end{equation}

This is an important property in light of the following observation: If~$f\colon V(\mathfrak G)\to\R$ is such that~$f(u)>f(v)$ and $\mathcal L f(x)=0$ for~$x\not\in D\cup\{v\}$ where
\begin{equation}
\label{E:D-def}
D:=\bigl\{x\in  V(\mathfrak G)\colon f(x)=f(u)\bigr\},
\end{equation}
then $\pi:=\{(x,y)\in E(\mathfrak G)\colon x\in D,y\not\in D\}$ defines a cut-set~$\pi$ from~$u$ to~$v$ such that $f(x)>f(y)$ holds for each edge~$(x,y)\in\pi$ oriented so that~$x\in D$ and~$y\not\in D$.

We will now define a sequence of functions $F_k\colon V(\mathfrak G)\to\R$ and cuts~$\pi_k$ by the following algorithm: INITIATE by $F_0:=f^\star$. If $F_{k-1}$ is constant then STOP, else use the above observation with $D_{k-1}$ related to~$F_{k-1}$ as~$D$ is to~$F$ in \eqref{E:D-def} to define~$\pi_k:=\{(x,y)\in E(\mathfrak G)\colon x\in D_k,y\not\in D_k\}$. Noting that $F_{k-1}(x)>F_{k-1}(y)$ for each edge $(x,y)\in\pi$ (oriented to point from~$u$ to~$v$), set~$\alpha_k:=\min_{(x,y)\in\pi}|F_{k-1}(x)-F_{k-1}(y)|$ and let
\begin{equation}
F_k(z):=F_{k-1}(z)-\alpha_k\,1_{D_{k-1}}(z).
\end{equation}
Then REPEAT.

As is checked by induction,~$k\mapsto D_k$ is strictly increasing while $k\mapsto F_k$ is non-increasing with $F_k=F^\star$ on~$ V(\mathfrak G)\smallsetminus D_k$. In particular, we have $\mathcal L F_k(x)=0$ for all $x\not\in D_k\cup\{v\}$ and so the above observation can repeatedly be used. The premise of the strict inequality $F_k(u)>F_k(v)$ for all but the final step is the consequence of the Maximum Principle. 

Now we perform some elementary calculations. Let~$\Pi$ denote the set of the cutsets~$\pi_k$ identified above. For each edge~$(x,y)$ and each~$\pi_k\in\Pi$, define
\begin{equation}
\label{E:13.15}
c_{e,\pi_k} := \frac{|F^\star(y)-F^\star(x)|}{\alpha_k}\,c_e.
\end{equation}
The construction and the fact that $F^\star(u)=1$ and~$F^\star(v)=0$ imply 
\begin{equation}
\label{E:13.16}
\sum_k\alpha_k=1\,.
\end{equation}
For any $e=(x,y)$, we also get
\begin{equation}
\sum_{k\colon e\in\pi_k}\alpha_k = \bigl|F^\star(y)-F^\star(x)\bigr|.
\end{equation}
In particular, the collection $\{c_{e,\pi}\colon\pi\in\Pi,\,e\in\pi\}$ obeys \eqref{eq-c-e}. Moreover, \eqref{E:13.15} then shows, for any $e=(x,y)$,
\begin{equation}
\sum_{k\colon e\in\pi_k} c_{e,\pi_k}\alpha_k^2 = c_e\bigl|F^\star(y)-F^\star(x)\bigr|\sum_{k\colon e\in\pi_k}\alpha_k = c_e\bigl|F^\star(y)-F^\star(x)\bigr|^2.
\end{equation}
Summing over~$e\in E(\mathfrak G)$ and rearranging the sums yields
\begin{equation}
C_{\mathfrak G}(u,v) = \sum_{e\in E(\mathfrak G)}\,\,\sum_{k\colon e\in\pi_k} c_{e,\pi_k}\alpha_k^2 =  
\sum_k\alpha_k^2\Bigl(\sum_{e\colon e\in\pi_k}c_{e,\pi_k}\Bigr)\,.
\end{equation}
Denoting the sum in the large parentheses by~$C_k$, among all non-negative~$\alpha_k$'s satisfying \eqref{E:13.16}, the right-hand side is minimal for~$\alpha_k:=\frac{1/C_k}{\sum_j1/C_j}$. This shows ``$\ge$'' in \eqref{eq-general-Nash-William}.
\end{proofsect}

Propositions~\ref{prop:fundamental_prop_resistance} and \ref{prop:fundamental_prop_conductance} seem to be closely related to various variational characterizations of effective resistance/conductance by way of optimizing over \emph{random} paths and cutsets. These are rooted in the Nash-Williams criterion and Terry Lyons' random-path method for bounding effective resistance (which can be shown to be sharp). The ultimate statements of these characterizations can be found in Berman and Konsowa~\cite{BK90}.

\subsection{ Restricted notion of effective resistance}
Propositions~\ref{prop:fundamental_prop_resistance} and \ref{prop:fundamental_prop_conductance}  naturally lead to restricted notions  of resistance and conductance  obtained by limiting the optimization to  only \emph{subsets}  of paths and cutsets, respectively.   For the purpose of  the  current paper we will only be concerned with effective resistance. To this end, 
 for each collection~$\mathcal A$ of finite sets of elements from~$E(\mathfrak G$), we define 
\begin{equation}
\label{resist_restrict}
R_{\mathfrak G}(\mathcal A) := \inf_{\{r_{e, A}: e\in E(\mathcal A),\, A \in \mathcal 
A\} \in \mathfrak R_\mathcal A}\Bigl(\,\sum_{A \in \mathcal A} \frac{1}{\sum_{e\in A} r_{e, A}}\Bigr)^{-1}\,,
\end{equation}
where $E(\mathcal A):=\bigcup_{A\in\mathcal A}A$ and where $\mathfrak R_\mathcal A$ is the set of all $\{r_{e, A}: e\in 
E(\mathcal A),\, A \in \mathcal A\} \in \mathbb R_+^{ E(\mathcal A)\times \mathcal A}$  such that 
\begin{equation}
\label{eq-r-e-restrict}
\sum_{A\in \mathcal A} \frac{1}{r_{e, A}} \leq \frac{1}{r_e} \,\text{\rm\ for all }\, e \in  E(\mathcal A)\,.
\end{equation}
 We refer to $R_{\mathfrak G}(\mathcal A)$ as the effective resistance \textit{restricted to~$\mathcal 
A$}. By taking suitable $r_{e,P}$,  the map  $\mathcal A\mapsto R_{\mathfrak G}(\mathcal A)$ is shown to be 
non-increasing with respect to the set inclusion. We will mostly be interested in $R_{\mathfrak G}(\mathcal A)$ when $\mathcal A$ is a set of  edge self-avoiding  paths from~$u$ to~$v$.  In particular, $R_{\mathfrak G}(\mathcal A) = R_{\mathfrak G}(u, v)$ when $\mathcal A$ is the set of all 
the edge self-avoiding paths between~$u$ and $v$.   The following result is a generalization of the triangle inequality for the effective resistance.

\begin{lemma}
\label{lem:series_law}
Let $\mathcal P_1, \mathcal P_2, \cdots, \mathcal P_k$ be collections of paths such that for any choice of $P_i$ from $\mathcal P_i$ for each $1 \leq i \leq k$, the graph union $\bigcup_{1 \leq i \leq k} P_i$ contains a path 
between $u$ and~$v$. Then 
\begin{equation}
R_{\mathfrak G}(u, v) \leq \sum_{i=1}^kR_{\mathfrak G}(\mathcal P_i)\,.
\end{equation}
\end{lemma}

\begin{proofsect}{Proof}
Define the edge sets $E_1, E_2, \cdots, E_k$ recursively by setting $E_1 := \bigcup_{P \in \mathcal P_1}E(P)$ and  letting  $E_j := \bigcup_{P \in \mathcal P_j}E(P) \setminus \bigcup_{i < j}E_i$ for $k \geq j > 1$. Let $\{r_{e, P}\colon e\in E(\mathfrak G), P\in \mathcal P_i\}$ be a vector in 
${\mathbb R}_+^{E(\mathfrak G)\times \mathcal P_i}$ satisfying \eqref{eq-r-e-restrict} for all~$i$.  For each~$i=1,\dots,k$ and  each $P \in \mathcal P_i$, define $\rho_{i, P}$ by
\begin{equation}
\rho_{i, P} := \frac{\Bigl(\,\sum_{e \in E(P)}r_{e, P}\bigr)^{-1}}{\sum_{P \in \mathcal P_i}\Bigl(\,\sum_{e 
\in E(P)}r_{e, P}\bigr)^{-1}}\,.
\end{equation}
Also for $e \in E_i$ and $P_1, P_2,\cdots, P_k\,$ in $\mathcal P_1, \mathcal 
P_2, \cdots, \mathcal P_k$ respectively, define
\begin{equation}
r_{e; P_1, P_2, \cdots, P_k} := r_{e,P_i}\prod_{j \neq i} \frac{1}{\rho_{j, P_j}}\,.
\end{equation}
Notice that for any $e \in E_i$,
\begin{equation}
\label{E:2.34}
\sum_{\substack{P_j \in \mathcal P_j, \\ 1 \leq j \leq k}} \frac{1}{r_{e; P_1, P_2, \cdots, P_k}} = 
\sum_{P_i \in \mathcal P_i} \frac{1}{r_{e, P_i}} \leq \frac{1}{r_e}\,,
\end{equation}
where the first equality follows from the fact that $\sum_{P \in \mathcal P_j}\rho_{j, P} = 1$ for 
all $j$ and the last inequality is a consequence of \eqref{eq-r-e-restrict}. 

The above definitions also immediately give
\begin{equation}
\begin{aligned}
\sum\limits_{e \in \bigcup_{1 \leq i \leq 
  k}E(P_j)}r_{e; P_1, P_2, \cdots,P_k}
&\le\sum\limits_{1 \leq i \leq k}\sum \limits_{e \in 
   E(P_i)}\frac{r_{e,P_i}}{\prod_{j \neq i}\rho_{j, 
  P_j}}
\\
&=\sum\limits_{1 \leq i \leq k}\frac{\big(\sum _{P 
  \in \mathcal P_i}\frac{1}{\sum_{e \in E(P)}r_{e, 
  P}}\big)^{-1}}{\prod_{1 \leq j \leq k}\rho_{j,P_j}}
\end{aligned}
\end{equation}
As \eqref{E:2.34} holds, Proposition~\ref{fundamental_cor_resistance} with~$T$ being the 
  set   of edges in $P_1,\dots,P_k$ yields
\begin{equation}
\begin{aligned}
\label{eq:series_law}
R_{\mathfrak G}(u, v) 
&\leq  \Biggl(\,\sum_{\substack{P_j \in \mathcal P_j, \\ 1 \leq j \leq 
k}} \frac{1}{\sum\limits_{e \in \bigcup_{1 \leq i \leq k}E(P_j)}r_{e; P_1, P_2, \cdots,P_k}}\Biggr)^{-1}
\\
&\leq  
\Biggl(\biggl[\,\sum\limits_{1 \leq i \leq k}\Bigl(\sum _{P \in \mathcal P_i}\frac{1}{\sum_{e \in E(P)}r_{e, P}}\Bigr)^{-1}\biggr]^{-1}
\sum_{\substack{P_j \in \mathcal P_j, \\ 1 \leq j \leq k}} \prod_{1 \leq j \leq k}\rho_{j,P_j}
\Biggr)^{-1}
\\
&= \sum\limits_{1 \leq i \leq k}\Bigl(\,\sum _{P \in \mathcal P_i}\frac{1}{\sum_{e \in E(P)}r_{e, P}}\Bigr)^{-1}\,,
\end{aligned}
\end{equation}
where we again used $\sum_{P \in \mathcal P_j}\rho_{j, P} = 1$ in the last step. 
Since \eqref{eq:series_law} holds for all choices of $\{r_{e, P}\colon e\in E(\mathfrak G), P\in \mathcal P_i\}$ satisfying \eqref{eq-r-e-restrict}, the claim follows from \eqref{resist_restrict}.
\end{proofsect}

A similar upper bound holds also for the effective conductance.

\begin{lemma}
\label{lem:parallel_law}
Let $\mathcal P_1,\dots,\mathcal P_k\in\mathfrak P_{u, v}$ be such that every path from~$u$ to~$v$ lies in the union $\bigcup_{1 \leq i \leq k} \mathcal P_i$. 
Then
\begin{equation}
C_{\mathfrak G}(u, v) \leq \sum_{1 \leq i \leq k}R_{\mathfrak G}(\mathcal P_i)^{-1}\,.
\end{equation}
\end{lemma}

\begin{proofsect}{Proof}
This is a direct consequence of Proposition~\ref{prop:fundamental_prop_resistance}.  Indeed, write $R_{\mathfrak G}(u,v)^{-1}$ as the suprema of $\sum_{P\in\mathcal P}(\sum_{e\in P}r_{e,P})^{-1}$ over~$\mathcal P$ and~$r_{e,P}$ satisfying \eqref{eq-r-e}. Next bound the sum over~$P$ by the sum over~$i=1,\dots,k$ and the sum over~$P\in\mathcal P\cap\mathcal P_i$ and observe, since $\sum_{P\in\mathcal P\cap\mathcal P_i}1/r_{e,P}\le\sum_{P\in\mathcal P}1/{r_{e,P}}\le1/r_e$, we have
\begin{equation}
\sum_{i=1}^k\sum_{P\in\mathcal P\cap P_i}\frac1{\sum_{e\in P}r_{e,P}}\le \sum_{i=1}^k R_{\mathfrak G}(\mathcal P_i)^{-1}.
\end{equation}
As this holds for all~$\mathcal P$ and all admissible $r_{e,P}$, the claim follows from \eqref{E:2.14}.\end{proofsect}

We note (and this will be useful later) that, in standard treatments of electrostatic theory on graphs, the notions of effective resistance/conductance are naturally defined between subsets  (as opposed to just single vertices)  of the underlying network. A simplest way to reduce this to our earlier definitions is by `` shorting''  the vertices in these sets together. Explicitly, given two non-empty disjoint sets $A,B\subseteq V(\mathfrak G)$ consider a network~$\mathfrak G'$ where all edges in $(A\times A)\cup(B\times B)$ have been removed and the vertices in~$A$ identified as one vertex~$\langle A\rangle$ --- with all edges in~$\mathfrak G$ with exactly one endpoint in~$A$ now ``pointing'' to~$\langle A\rangle$ in~$\mathfrak G'$ --- and the vertices in~$B$ similarly identified as one vertex~$\langle B\rangle$. Then we define
\begin{equation}
R_{\mathfrak G}(A,B):=R_{\mathfrak G'}\bigl(\langle A\rangle,\langle B\rangle\bigr)
\quad\text{\rm and}\quad
C_{\mathfrak G}(A,B):=C_{\mathfrak G'}\bigl(\langle A\rangle,\langle B\rangle\bigr).
\end{equation}
Note that, for one-point sets, $R_{\mathfrak G}(\{u\},\{v\})$ coincides with~$R_{\mathfrak G}(u,v)$, and similarly for the effective conductance.   The electrostatic duality also holds, $R_{\mathfrak G}(A,B) = 1/C_{\mathfrak G}(A,B)$. 

\subsection{Self-duality}
 The results in this section are consequences of Propositions~\ref{fundamental_cor_resistance} and 
\ref{prop:fundamental_prop_conductance}.  Observe that the similarity of the two formulas  \eqref{eq-general-parellel2} (also \eqref{eq-general-parellel})
 and  
 \eqref{eq-general-Nash-William}   naturally leads to  the  consideration of self-dual situations --- i.e., those in which the resistances~$r_e$ can somehow be exchanged for the conductances~$c_e$  and the paths can be exchanged for the cutsets.  An example of this is the network~$\Z^2_{\eta}$ where the distributional identity $\eta\laweq-\eta$  (recall from the introduction that $\eta$ is a centered Gaussian field)  makes the associated resistances $\{r_e\colon e\in E(\Z^2)\}$ equidistributed to the conductances $\{c_e\colon e\in E(\Z^2)\}$;  planarity then associates cutsets with paths.  To formalize this situation, given a network~$\mathfrak G$ we define its \emph{reciprocal}~$\mathfrak G^\star$ as the network with the same underlying graph but with the resistances swapped for the conductances. An edge~$e$ in network~$\mathfrak G^\star$ thus has resistance $r_e^\star:=1/r_e$, where~$r_e$ is the resistance of~$e$ in network~$\mathfrak G$.

\begin{lemma}
\label{lem:dual_couple}
Let $\mathfrak D$ denote the maximum vertex degree in~$\mathfrak G$ and let $\rho_{\max}$ denote the maximum ratio of the resistances of any pair of adjacent edges in~$\mathfrak G$. Given two pairs $(A, B)$ and $(C, D)$ of disjoint, nonempty subsets of~$V(\mathfrak G)$,  suppose that every path between $A$ and $B$ shares a vertex with every path between $C$ and $D$. Then
\begin{equation}
\label{ineq_dual_couple}
R_{\mathfrak G}(A, B) {R}_{\mathfrak G^\star}(C, D)\geq  \frac{1}{4 \mathfrak{D}^2\rho_{\max} } \,.
\end{equation}
\end{lemma}

\begin{proofsect}{Proof}
The proof is based on the fact that every   path $P$ between $C$ and $D$  
 defines a cutset~$\pi_P$ between~$A$ and~$B$ by taking~$\pi_P$ to be the set of all edges adjacent to any edge in~$P$, but not including the edges in $(A\times A)\cup(B\times B)$.  By the electrostatic 
 duality we just need to show 

\begin{equation}
\label{E:2.36}
C_{\mathfrak G}(A, B)\le 4 \mathfrak{D}^2\rho_{\max} {R}_{\mathfrak G^\star}(C, D)\,.
\end{equation}
To this end, given any $\mathcal P \in \mathfrak P_{C, D}$  (i.e., $\mathcal P$ is a multiset of edge self-avoiding paths between~$C$ and~$D$)  let us pick positive numbers $\{r'_{e,P}: e \in E(\mathcal P), P\in \mathcal P\}$ such~that
\begin{equation}\label{eq-r-e3}
\sum_{P\in \mathcal P} \frac{1}{r'_{e, P}} \leq \frac{1}{c_e} \,\text{\rm\ for all }\, e \in  E(\mathcal P)\,.
\end{equation}
For any edge~$e$ and any path~$P\in\mathcal P$, let $N_P(e)$, 
$N_{\mathcal P}(e)$ and $N(e)$ denote the sets of all edges in $E(P)$, $E(\mathcal P)$ and 
$E(\mathfrak G)$ that are adjacent to~$e$, respectively. For any $e \in E(\mathcal P)$ 
and any $P \in \mathcal P$, let $\theta_{e, P} := c_e / r'_{e,P}$  and note  that $\theta_{e, P}$'s are  positive  
numbers satisfying $\sum_{P \in \mathcal P} \theta_{e, P} \le  1$ for all $e\in E(\mathcal P)$.  As a consequence, if we define 
\begin{equation}
c_{e, \pi_P} := \frac{c_e}{\sum\limits_{e' \in N_{\mathcal P'}(e)} \theta_{e', P}}| N_{\mathcal 
P}(e) |
\end{equation}
 then $\{c_{e,\pi_P}\colon e\in 
 \bigcup_{P \in \mathcal P} \pi_P,\,P\in\mathcal P\}$ satisfies 
 \eqref{eq-c-e}.  Now fix  a path $P$ in~$\mathcal P$ and compute, invoking the definitions of~$\mathfrak D$, $\rho_{\max}$ and also Jensen's inequality in the second step: 
\begin{equation}
\begin{aligned}
\sum\limits_{e \in \pi_P} c_{e, \pi_P} &= 
\sum\limits_{e \in \pi_P}\frac{c_e}{\sum\limits_{e' \in 
N_{\mathcal P}(e)} \theta_{e', P}}| N_{\mathcal P}(e)| \leq  2\mathfrak D \sum\limits_{e 
\in \pi_P}\frac{c_e}{\sum\limits_{e' \in N_P(e)} \theta_{e', P}}| N_P(e)| 
\\
&\leq  2\mathfrak D\sum\limits_{e \in 
\pi_P}\Big(\frac{c_e}         
{| N_P(e)|} \sum_{e' \in N_P(e)} \frac{1}
{\theta_{e',P}}\Big) \leq 2 \mathfrak D\sum
_{\begin{subarray}{c}
 e \in 
E(\mathfrak G), e' \in P \\  e\sim e' 
\end{subarray}}
\frac{c_e }{\theta_{e', P}}  
\\ 
&= 2\mathfrak D\sum\limits_{e' \in P}\frac{\sum\limits_{e 
\in 
N(e')}c_e}{\theta_{e', P}} \leq 4\mathfrak D^2 \rho_{\max} 
\sum\limits_{e' 
\in P}\frac{c_{e'}}{\theta_{e',P}} = 4\mathfrak D^2 
\rho_{\max}\sum\limits_{e' \in P}r'_{e',P}\,.
\end{aligned}
\end{equation}
 Hence we get by Proposition~\ref{prop:fundamental_prop_conductance}
 \begin{equation}
\label{conduc_resist_ineq}
C_{\mathfrak G}(A, B)   \leq \Bigl(\,\sum_{P \in \mathcal P'} \frac{1}{\sum_{e \in \pi_P} c_{e, 
\pi_P}}\Bigr)^{-1} \leq 4{\mathfrak D}^2\rho_{\max}\Bigl(\,\sum_{P \in \mathcal P'} \frac{1}{\sum_{e \in P} r'_{e, P}}\Bigr)^{-1}\,.
\end{equation}

As this holds for any choice of $\mathcal P$ and positive numbers $\{r'_{e, P}: e \in E(\mathcal P), P\in \mathcal P\}$ satisfying \eqref{eq-r-e3}, we get~\eqref{E:2.36} as desired.
\end{proofsect}

 A  crucial fact  underlying  the proof of the previous lemma was that one could obtain a cut 
set for~$\mathcal P$ from a path $P$ in $\mathcal P$ by taking union of all edges adjacent to vertices in $P$. 
In the same setup,  
we get a corresponding result also for effective conductances. Indeed, we have:

\begin{lemma}
\label{remark_dual_couple}
For the same setting and notation as in Lemma~\ref{lem:dual_couple}, assume that for every cutset~$\pi$ between~$C$ and~$D$, the subgraph induced by the set of all edges that are adjacent to some edge in~$\pi$ contains a path in $\mathfrak P_{A, B}$. Then 
\begin{equation}
\label{eq:dual_couple_conduct}
C_{\mathfrak G}(A , B) {C}_{\mathfrak G^\star}(C , D)\geq \frac{1}{4{\mathfrak D}^2\rho_{\max}}\,.
\end{equation}
\end{lemma}

\begin{proofsect}{Proof}
For any cutset $\pi$ between $C$ and $D$, let $T_\pi$ denote the set of all edges that are adjacent to some 
edge in $\pi$. Thus $T_\pi$ contains a path in $\mathfrak 
P_{A, B}$ by the hypothesis of the lemma.  Also recall from the statement of  Proposition~\ref{prop:fundamental_prop_conductance} that $\varPi(C, D)$ is the set of all finite collections of 
cutsets between~$C$ and~$D$.  Now given any $\Pi \in \varPi_{C, D}$, we pick a collection of positive numbers $\{c_{e, \pi}^\star \colon e\in \bigcup_{\pi \in \Pi} T_\pi, \pi \in \Pi \}$ such that
\begin{equation}
\label{eq:dual_couple_conduct1}
\sum_{\pi \in \Pi}\frac{1}{c_{e, \pi}^\star} \leq \frac{1}{r_e}\,.
\end{equation}
Following the exact same sequence of steps as { in the proof of Lemma~\ref{lem:dual_couple}, we now find} $\{r_{e, T_\pi}\colon e \in \pi, \pi \in \Pi\}$ satisfying \eqref{eq-r-e2} such that
\begin{equation}
\label{eq:dual_couple_conduct2}
\Bigl(\,\sum_{\pi \in \Pi} \frac{1}{\sum_{e \in T_\pi} r_{e, 
T_\pi}}\Bigr)^{-1} \leq 4{\mathfrak D}^2\rho_{\max}\Bigl(\,\sum_{\pi \in \Pi} \frac{1}{\sum_{e \in \pi} c_{e, \pi}^\star}\Bigr)^{-1}\,.
\end{equation}
Proposition~\ref{fundamental_cor_resistance} then implies
\begin{equation}
R_{\mathfrak G}(A, B) \leq \Bigl(\,\sum_{\pi \in \Pi} \frac{1}{\sum_{e \in T_\pi} r_{e, 
T_\pi}}\Bigr)^{-1} \leq 4{\mathfrak D}^2\rho_{\max}\Bigl(\,\sum_{\pi \in \Pi} \frac{1}{\sum_{e \in \pi} c_{e, \pi}^\star}\Bigr)^{-1}\,.
\end{equation}
As this holds for all choices of $\Pi$ and $\{c_{e, \pi}^\star \colon e\in \pi, \pi \in \Pi \}$ satisfying \eqref{eq:dual_couple_conduct1}, we get the desired inequality \eqref{eq:dual_couple_conduct}.
\end{proofsect}

\section{Preliminaries on Gaussian processes}
\noindent
 Before we move on to the main line of the proof, we need to develop some preliminary control on the underlying Gaussian fields. The goal of this section is to amass the relevant technical claims concerning Gaussian processes and, in particular, the GFF. An impatient, or otherwise uninterested, reader may consider only skimming through this section and returning to it when the relevant claims are used in later proofs.

\subsection{Some standard inequalities}
We start by recalling, without proof, a few standard facts about general Gaussian processes:

\begin{lemma}[Theorem 7.1 in \cite{L01}]
\label{lem:Borell_ineq}
 Given a finite set $A$,  consider a centered Gaussian process $\{X_v: v \in A\}$. Then, for $x > 0$,
\begin{equation}
\P\biggl(\,\Bigl|   \max_{v \in A}X_v  - \E  \max_{v \in A}X_v\Bigr|  \geq x\biggr) \leq 2 \e^{-x^2 / 2 \sigma^2}\,,
\end{equation}
where $\sigma^2 := \max_{v \in A}\E (X_v^2)$.
\end{lemma}

\begin{lemma}[Theorem 4.1 in \cite{A90}]
\label{lem:basic_chaining}
 Suppose~$(S,d)$ is a finite metric space such that  $\max_{s, t \in S} d(s, t) = 1$ and assume there are $\beta,K_1\in(0,\infty)$ such that for every $\epsilon \in (0, 1]$, the $\epsilon$-covering number  $N_\epsilon(S, d)$ of $(S, d)$ obeys $N_\epsilon(S, d) \leq K_1\epsilon^{-\beta}$.   Then for any $\alpha,K_2\in(0,\infty)$ and any centered Gaussian process~$\{X_s\}_{s \in S}$ satisfying 
\begin{equation}
\sqrt{\E(X_s - X_{s'})^2} \leq K_2\,d(s, s')^\alpha,\qquad s,s'\in S,
\end{equation}
we have 
\begin{equation}
\E \bigl(\,\max_{s \in A} |X_s|\bigr) \leq K
\quad\text{\rm and}\quad
\E \bigl(\,\max_{s,t \in A} |X_s-X_t|\bigr) \leq K,
\end{equation}
where $K:=K_2(\sqrt{\beta\log 2} + \sqrt{\log (K_1 + 1)})K_\alpha$ with $K_\alpha := \sum_{n \geq 0}2^{-n\alpha}\sqrt{n + 1}$. 
\end{lemma}

As a consequence of Lemma~\ref{lem:basic_chaining} we get the following  result which we will use in the next subsection.

\begin{lemma} 
\label{lem:max_expectation}
Let $B_1, B_2, \ldots, B_N$ be squares  in~$\Z^2$  of side lengths $b_1, b_2, \ldots, b_N$ respectively and let $B := \cup_{j 
\in [N]}B_j$.  There exists an absolute constant $C' > 0$ such that, if 
$\{X_v\}_{v \in B}$ is a centered Gaussian process satisfying
\begin{equation}
\E(X_u - X_v)^2 \leq \frac{|u -v|}{b_j},\qquad  (u, v)\in\bigcup_{j=1}^N (B_j\times B_j),
\end{equation}
then
\begin{equation}
\label{E:3.5x}
\E \max_{v \in B} X_v \leq C'\sqrt{\log N}\Bigl(1 + \max_{v \in B}\sqrt{\E X_v^2}\,\Bigr) + C'\,.
\end{equation}
\end{lemma}

The following lemma, taken from \cite{Pitt82}, is the 
FKG inequality for Gaussian random variables. We will refer to this as the FKG in the rest of the paper.

\begin{lemma}
\label{lem:fkg_gaussian}
Consider a Gaussian process $\mathbf{X} = \{X_v\}_{v \in A}$  on a finite set~$A$,  and suppose that 
\begin{equation}
\label{E:3.6}
\cov(X_u, X_v)\ge0,\qquad u,v\in A.
\end{equation}
Then
\begin{equation}
\cov \bigl(f(\mathbf{X}), g(\mathbf{X})\bigr) \geq 0
\end{equation}
holds for any bounded, Borel measurable functions $f, g$ on $\R^A$ that are increasing  separately  in each coordinate.
\end{lemma}

As a corollary to FKG, we get:

\begin{corollary}
\label{cor:square_root_trick}
Consider a Gaussian process $\mathbf{X} = \{X_v\}_{v \in A}$ on a finite set $A$ such that \eqref{E:3.6} holds.  If $\mathcal E_1, \mathcal E_2, \cdots, \mathcal E_k\in\sigma(\mathbf X)$ are all increasing (or all decreasing), then 
\begin{equation}
\max_{i \in [k]}\P(\mathcal E_i) \geq 1 - \biggl(1 - \P\Bigl(\bigcup_{i \in [k]}\mathcal E_i\Bigr)\biggr)^{1/k}\,.
\end{equation}
\end{corollary}

\noindent
This is known as the ``square root trick'' in percolation literature (see, e.g., \cite{GR99}).

\subsection{Smoothness of harmonic averages of the GFF}
\label{subsec:smoothness}
Moving to the specific example of the GFF we note that one of the  most  important properties that makes the GFF amenable to analysis is its behavior under restrictions to a subdomain. This goes by the name Gibbs-Markov, or domain-Markov, property.  In order to give a precise statement (which will happen in Lemma~\ref{lem:Markov} below) we need some notations.

Given a set~$A\subseteq\Z^2$, let~$\partial A$ denote the set of vertices in~$\Z^2\smallsetminus A$ that have a neighbor in~$A$. Recall that a GFF in $A\subsetneq\Z^2$ with Dirichlet boundary condition is a centered Gaussian process $\chi_A=\{\chi_{A,v}\}_{v\in A}$ such that 
\begin{equation}
\chi_{A, v} = 0 \,\,\text{\rm\ for }\,\, v \in \Z^2\smallsetminus A\qquad\text{and}\qquad
\E (\chi_{A, u}\chi_{A, v}) = G_{A}(u, v)\,\,\text{\rm\ for }\,\, u,v\in A,
\end{equation}
where $G_{A}(u, v)$ is the Green function in~$A$; i.e., the expected number of visits to $v$ for the simple random walk on $\Z^2$ started at $u$ and killed upon entering $\Z^2\smallsetminus A$. 
We then have:

\begin{lemma}[Gibbs-Markov property]
\label{lem:Markov}
Consider the GFF $\chi_A=\{\chi_{A,v}\}_{v\in A}$ on a set $A\subsetneq\Z^2$   with Dirichlet boundary condition   and let~$B\subseteq A$ be finite. Define  the random fields  $\chi_A^c=\{\chi_{A,v}^c\}_{v \in B}$ and $\chi_A^f=\{\chi_{A,v}^f\}_{v \in B}$ by
\begin{equation}
\chi_{A,v}^c = \E\bigl(\chi_{A,v} \,\big|\, \chi_{A,u}\colon u\in A\smallsetminus   B \bigr)\quad\text{\rm and}\quad \chi_{A,v}^f = \chi_{A,v} - \chi_{A,v}^c.
\end{equation}
Then $\chi_A^f$ and $\chi_A^c$ are independent with $\chi_A^f\laweq\chi_B$. Moreover, $\chi_A^c$ equals~$\chi_A$ on~$A\smallsetminus  B $  and its sample paths are discrete harmonic on~$B$. 
\end{lemma} 

\begin{proofsect}{Proof}
This is verified directly by writing out the probability density of~$\chi_A$ or, alternatively, by noting that the covariance of $\chi^c_A$ is  $G_{A} - G_{B}$,  which is harmonic in both variables throughout~$B$. We leave further details to the reader.
\end{proofsect}

By way of reference to the spatial scales that these fields will typically be defined over, we refer to $\chi_A^f$ as the \emph{fine} field and $\chi_A^c$ as the \emph{coarse} field. However, this should not be confused with the way their actual sample paths look like. Indeed, the samples of $\chi_A^f$ will typically be quite rough (being those of a GFF), while the samples of~$\chi_A^c$ will be rather smooth (being discrete harmonic on~$B$). Our next goal is to develop a good control of the smoothness of~$\chi_A^c$ precisely. A starting point is the following estimate:

\newcommand{\fraka}{\mathfrak a}

\begin{lemma}
\label{lemma-3.7}
There is an absolute constant~$c\in(0,\infty)$ such that, given any $\emptyset\ne \tilde B\subseteq B\subseteq A \subsetneq  \Z^2$ with~$\tilde B$ connected and denoting
\begin{equation}
\label{E:3.11ua}
N:=\sup \bigl\{M \in \N \colon \tilde B + [-M, M]^2\cap \Z^2 \subseteq B\bigr\}, 
\end{equation}
the coarse field $\chi_A^c$ on~$B$ obeys
\begin{equation}
\label{E:3.10}
\var\bigl(\chi_{A,u}^c-\chi_{A,v}^c\bigr) \le c\Bigl(\frac{\text{\rm dist}_{\tilde B}(u,v)}N\Bigr)^2,\qquad u,v\in \tilde B,
\end{equation}
where~$\text{\rm dist}_{\tilde B}(x,y)$ denotes the length of the shortest path in~$\tilde B$ connecting~$x$ to~$y$.
\end{lemma}

\begin{proofsect}{Proof}
Let~$u,v\in \tilde B$ first be nearest neighbors and let $M:=\lfloor N/2\rfloor$. Using $(f,g)$ to denote the canonical inner product in~$\ell^2(\Z^2)$ with respect to the counting measure, the Gibbs-Markov property gives 
\begin{equation}
\label{E:3.13ui}
\var\bigl(\chi_{A,u}^c-\chi_{A,v}^c\bigr)
=\Bigl(\delta_u-\delta_v,(G_{A}-G_{B})(\delta_u-\delta_v)\Bigr)
\end{equation}
Since $A\mapsto G_A$ is increasing (as an operator $\ell^2(\Z^2)\to\ell^2(\Z^2)$) with respect to the set inclusion,  the largest right-hand side of \eqref{E:3.13ui} for the current setting is achieved by~$A$ being the complement of  a single point and~$B$ being  the square  $u + B(M) = u + [-M,M]^2\cap\Z^2$.   Focusing on such~$A$ and~$B$ from now on and shifting the domains suitably, we may  assume $A := \Z^2 \smallsetminus \{0\}$.  Then
\begin{equation}
\label{E:3.13}
 G_{A}(x,y)  =\fraka(x)+\fraka(y)-\fraka(x-y),
\end{equation}
where $\fraka(x)$ is the potential  kernel  defined, e.g., as the limit value of $G_{B(N)}(0,0)-G_{B(N)}(0,x)$ as~$N\to\infty$. The relevant fact for us is that  $\fraka$  admits the asymptotic form
\begin{equation}
\label{E:3.14}
\fraka(x) = g\log|x|+c_0+O\bigl(|x|^{-2}\bigr),\qquad|x|\to\infty,
\end{equation}
where $g:=2/\pi$ and~$c_0$ is a (known) constant.

There is another representation of $\var\bigl(\chi_{A,u}^c-\chi_{A,v}^c\bigr)$ in terms of harmonic measures which follows from the discrete harmonicity of the coarse field.  Let $H^B(x,y)$, for $x\in B$ and~$y\in \partial B$, denote the harmonic measure; i.e., the probability that the simple random walk started from~$x$ first  enters~$\Z^2\smallsetminus B$  at~$y$. Then
\begin{equation}
\var\bigl(\chi_{A,u}^c-\chi_{A,v}^c\bigr)
=\bigl(f, G_{A} f\bigr)
\end{equation}
 where 
\begin{equation}
\label{E:3.11}
f(\cdot):=\sum_{z\in\partial B}\bigl[H^B(u,z)-H^B(v,z)\bigr]\delta_z(\cdot).
\end{equation}
 In order to make use of this expression, we  will need suitable estimates for the harmonic measure: There are constants~$c_1,c_2\in(0,\infty)$ such that for all~$M\ge1$, any neighbor~$v$ of~$u$ and~$B:=u+B(M)$, from, e.g.,   \cite[Proposition 8.1.4]{Lawler10}  , we have
\begin{equation}
\label{E:3.15}
H^B(u,z)\le\frac{c_1}M,\qquad z\in\partial B,
\end{equation}
and
\begin{equation}
\label{E:3.16}
\bigl|H^B(u,z)-H^B(v,z)\bigr|\le\frac{c_2}M\,H^B(u,z),\qquad z\in\partial B.
\end{equation}
For our special choice of~$A$, using \eqref{E:3.11} we now write
\begin{multline}
\label{E:3.18}
\var\bigl(\chi_{A,u}^c-\chi_{A,v}^c\bigr)
\\
=\sum_{z,\tilde z\in\partial B}
\bigl[H^B(u,z)-H^B(v,z)\bigr]\bigl[H^B(u,\tilde z)-H^B(v,\tilde z)\bigr]\bigl(\fraka(z)+\fraka(\tilde z)-\fraka(z-\tilde z)\bigr).
\end{multline}
Since $z\mapsto H^B(u,z)$ is a probability measure for each~$u$, the contribution of the terms $\fraka(z)$ 
and~$\fraka(\tilde z)$ vanishes.  For the same reason, we { may} replace $\fraka(z - \tilde z)$ with $\fraka(z - \tilde z) - g\log M$ in~\eqref{E:3.18}.  
 Now we apply \eqref{E:3.16} with the result 
\begin{equation}
\var\bigl(\chi_{A,u}^c-\chi_{A,v}^c\bigr)
\le\Bigl(\frac{ c_2}M\Bigr)^2
\sum_{z,\tilde z\in\partial B}
H^B(u,z)H^B(u,\tilde z)|\fraka(z-\tilde z) - g\log M| \,.
\end{equation}
Invoking \eqref{E:3.14}  and  \eqref{E:3.15}, the two sums are  bounded by a constant independent of~$M$. This gives \eqref{E:3.10} for any neighboring pair of vertices. For the general case we apply the triangle inequality  for the intrinsic (pseudo)metric $u,v\mapsto [\var(\chi_{A,u}^c-\chi_{A,v}^c)]^{1/2}$  along the shortest path in~$\tilde B$ between~$u$ and~$v$  in the graph-theoretical metric. 
\end{proofsect}

Using the above variance bound, we now get:

\begin{corollary}
\label{cor:field_smoothness1}
For each set~$A\subseteq \Z^2$, let us write~$\text{\rm diam}_A(A)$ for the diameter~$A$ in the graph-theoretical metric on~$A$.
For each~$\delta>0$ there are constants $c,\tilde c\in(0,\infty)$ such that for all sets $\emptyset\ne \tilde B\subseteq B\subseteq A\subsetneq\Z^2$ with~$\tilde B$ connected and obeying 
\begin{equation}
\label{E:3.19}
\sup \bigl\{M \in \N \colon \tilde B + [-M, M]^2\cap \Z^2 \subseteq B\bigr\} \ge\delta\,\text{\rm diam}_{\tilde B}(\tilde B)
\end{equation} 
 and for $\chi_A^c$ denoting the coarse field on~$B$ for the GFF~$\chi_A$ on~$A$, we have
\begin{equation}
\label{E:3.20}
\P\Bigl(\,\sup_{u,v\in \tilde B}\bigl|\chi_{A,u}^c-\chi_{A,v}^c\bigr|>c+t\Bigr)\le 2\e^{-\tilde c t^2}
\end{equation}
for each~$t\ge0$.
\end{corollary}

\begin{proofsect}{Proof}
The condition \eqref{E:3.19} ensures, via Lemma~\ref{lemma-3.7}, that the variance of $\chi_{A,u}^c-\chi_{A,v}^c$ is bounded by a constant times $\text{dist}_{\tilde B}(u,v)/N$ with  
 $N$ as in~\eqref{E:3.11ua}.  The assumption \eqref{E:3.19}  then  ensures that this is at most a $\delta$-dependent constant. Writing this constant as~$2/\tilde c$ and denoting
\begin{equation}
M^\star:=\sup_{u,v\in \tilde B}\bigl|\chi_{A,u}^c-\chi_{A,v}^c\bigr|\,,
\end{equation}
 Lemma~\ref{lem:Borell_ineq} gives
\begin{equation}
\P\bigl(|M^\star-\E M^\star|>t\bigr)\le2\e^{-\tilde c t^2}.
\end{equation}
It remains to bound $\E M^\star$ uniformly in~$A$ and~$B$ satisfying \eqref{E:3.19}. For this we note that, by Lemma~\ref{lemma-3.7}, an $\epsilon$-ball in the intrinsic metric $\rho(u,v):=[\var(\chi_{A,u}^c-\chi_{A,v}^c)]^{1/2}$ on~$\tilde B$ contains an order-$N\epsilon$ ball in the graph-theoretical metric on~$\tilde B$ which in turn contains an order-$(N\epsilon)^2$ ball in the $\ell^1$-metric on~$B$.  Lemma~\ref{lem:max_expectation} then applies with $\alpha:=1$ and $\beta:=2$ and the bound follows from \eqref{E:3.5x}.
\end{proofsect}


\subsection{A LIL for averages on concentric annuli}
The proof of the RSW estimates will require controlling the expectation of the GFF on concentric annuli, conditional on the values of the GFF on the boundaries thereof. We will conveniently represent the sequence of these expectations by a random walk. Annulus averages and the associated random walk have been central to the study of the local properties of nearly-maximal values of the GFF in~\cite{BL3}. However, there the emphasis was on estimating the probability that the random walk stays above a polylogarithmic curve for a majority of time, while here we are interested in a different aspect; namely, the Law of Iterated Logarithm. The conclusions derived here will be applied in the proof of Proposition~\ref{lem:Tassion_main}.

We begin with a quantitative version of the law of the iterated logarithm for a specific class of Gaussian random walks.

\begin{lemma}
\label{lem:finite_LIL}
 Set $\phi(x):=\sqrt{2x\log \log x}$ for $x \geq 3$ and let  $Z_1, Z_2, \cdots, Z_n$ be independent random variables with $Z_i \laweq \mathcal N(0, \sigma_i^2)$ for 
some $\sigma_i^2>0$. Let $s_k^2 := \sum_{1 
\leq i \leq k}\sigma_i^2$ and suppose that there are $\sigma > 0$ and  $d > 0$  such that
\begin{equation}
\sigma^2k -  d  \leq s_k^2 \leq \sigma^2k +  d ,\qquad k\ge1.
\end{equation}
Then there are $c_{\sigma, d}>0$, $C_{\sigma,d}>0$ and $N_{\sigma, d}>0$,  depending only on~  $d$   and~$\sigma$,  such that for all $n \geq N_{\sigma, d}$, the random walk $S_k := \sum_{1 \leq i \leq k}Z_i$ obeys
\begin{equation}
\label{eq:finite_LIL}
\P\Bigl(\#\bigl\{\e^{\sqrt{\log n}} \leq k \leq n: S_k \geq \phi(s_k^2)/2\bigr\} \geq c_{\sigma, d}\log \log n \Bigr) \geq 1 - \frac{C_{\sigma, d}}{\log \log n}\,.
\end{equation}
\end{lemma}

\begin{proofsect}{Proof}
 Since $\phi$ is regularly varying at infinity with exponent~$1/2$ and~$k\mapsto s_k^2$ is within distance~$d$ of a linear function, one can find~$a>1$ and~$k_1$ sufficiently large (and depending only on~$\sigma$ and~$d$) such that
\begin{equation}
\label{eq:finite_LIL1}
\phi(s_{a^k}^2 - s_{a^{k -1}}^2) \geq \frac67 \phi(s_{a^k}^2), \qquad k\ge k_1,
\end{equation}
and
\begin{equation}
\label{eq:finite_LIL2}
\phi(s_{a^{k -1}}^2) \leq \frac29\phi(s_{a^k}^2), \qquad k\ge k_1,
\end{equation}
 hold true. 
Now define a sequence of random variables as
\begin{equation}
T_1 := S_a - S_1,\quad T_2 := S_{a^2} - S_a,\,\,\,\dots\quad T_{\lfloor \log_a n\rfloor} := S_{a^{\lfloor \log_a 
n\rfloor}} - S_{a^{\lfloor \log_a n\rfloor - 1}}\,.
\end{equation}
 Then  $T_1, T_2, \cdots, T_{\lfloor \log_a n 
\rfloor}$ are independent  with $T_k\laweq\mathcal N(0,s^2_{a^k}-s^2_{a^{k-1}})$.  Then, for each $k$ with $k_1 \leq k \leq \lfloor \log_a n\rfloor$,  the inequality \eqref{eq:finite_LIL1} and a straightforward Gaussian tail estimate show  
\begin{equation}
\P\bigl(T_k \geq \tfrac34\phi(s_{a^k}^2)\bigr) \geq \P\bigl(T_k \geq \tfrac78\phi(s_{a^k}^2 - s_{a^{k-1}}^2)\bigr)\geq \frac{c}{\log (s_{a^k}^2 - s_{a^{k-1}}^2)}\,,
\end{equation}
 for some constant $c>0$ 
depending only on $\sigma$ and $d$.  Thus, whenever~$n$ is such that $\sqrt{\lfloor \log_a n\rfloor} \geq k_1$ holds true, we have 
\begin{equation}
\label{eq:finite_LIL3}
\sum_{\sqrt{\lfloor \log_a n\rfloor} \leq k \leq \lfloor \log_a n\rfloor}
\P\bigl(T_k \geq \tfrac34\phi(s_{a^k}^2)\bigr) \geq  c'\log \log n - c''\,,
\end{equation}
for some $c',c''>0$. By independence of $T_1, T_2, \cdots,$ $T_{\lfloor \log_a n \rfloor}$,  the Chebyshev inequality gives 
\begin{equation}
\label{eq:finite_LIL4}
\P\biggl(\#\Bigr\{\sqrt{\lfloor \log_a n\rfloor} \leq k \leq \lfloor \log_a n\rfloor\colon T_k \geq \tfrac34\phi(s_{a^k}^2)\Bigr\} \geq \frac{c'\log \log n}{2}\biggr) \geq 1 - \frac{\tilde c}{\log \log n}
\end{equation}
 for some constant~$\tilde c\in(0,\infty)$.  
A  computation using a Gaussian tail estimate gives  
\begin{equation}
\P\bigl(S_{a^k} \leq -\tfrac98\phi(s_{a^k}^2)\bigr) \leq (\log s_{a^k}^2)^{-81 / 64}\,
\end{equation}
for all $k \geq 1$. Therefore 
\begin{equation}
\label{eq:finite_LIL5}
\P\biggl(\,\,\bigcup_{\sqrt{\lfloor \log_a n\rfloor} \leq k \leq \lfloor \log_a n\rfloor}\bigl\{S_{a^k} \leq -\tfrac98\phi(s_{a^k}^2)\bigr\}\biggr) \leq \tilde c'(\log n)^{-17/128}\,,
\end{equation}
for some constant~$\tilde c'\in(0,\infty)$.  On $\{S_{a^{k-1}}\ge-\tfrac98\phi(s^2_{a^{k-1}})\}\cap\{T_k\ge\tfrac34\phi(s_{a^k}^2)\}$, \eqref{eq:finite_LIL2} gives
\begin{equation}
S_{a^k}=S_{a^{k-1}}+T_k\ge-\tfrac98\phi\bigl(s^2_{a^{k-1}}\bigr)+\tfrac34\phi\bigl(s_{a^k}^2\bigr)
\ge\tfrac12\phi\bigl(s_{a^k}^2\bigr) 
\end{equation} 
and so the  bounds \eqref{eq:finite_LIL4} and \eqref{eq:finite_LIL5} imply \eqref{eq:finite_LIL}.
\end{proofsect}

We will apply Lemma~\ref{lem:finite_LIL} to a special sequence of random variables which arise from averaging the GFF along concentric   squares.   For integers $N\ge1$,~$n\ge1$ and~$b\ge2$, denote $N' := 
b^{n}N$ and,
for each $k \in\{1,\dots,n\}$, define
\begin{equation}
\label{E:3.66}
M_{ n,k} := \E\biggl(\chi_{N',0} \,\bigg|\, \sigma\Bigl(\chi_{N', v}: v \in \bigcup_{n-k \leq j \leq n}\partial B(b^j N)\Bigr)\biggr)\,,
\end{equation}
Notice that  we can also write  $M_{ n,k} = \E\big(\chi_{N',0} | \sigma(\chi_{N', v}\colon v\in \partial B(b^{n-k}N))\big)$ due to the Gibbs-Markov property of~the~GFF. We then have: 

\begin{lemma}
\label{lem:ellipicity_upper}
For each integer  $b\ge1$  as above, there are constants $\sigma>0$ and~$d>0$  such that for all $N\ge1$ and all $n\geq 1$ the sequence
$\{M_{ n,k} - M_{ n,k-1}\}_{k =1,\dots,n-1}$ (with $M_{ n,0} := 0$) satisfies the conditions of Lemma~\ref{lem:finite_LIL} with these $(\sigma,  d )$.
\end{lemma}

\begin{proofsect}{Proof}
Since the $M_{ n,k} - M_{ n,k-1}$'s are differences of a Gaussian 
martingale sequence, they are independent normals. So we only need to verify the constraints on the variances. Denoting   $N'':=b^{n - k} N$,   the Gibbs-Markov property of the GFF implies
\begin{equation}
\label{eq:ellipticity_upper1}
\var (M_{ n,k}) = G_{B(N')}(0,0)-G_{B(N'')}(0,0).
\end{equation}
Recalling our notation $H^B(x,y)$ for the harmonic measure, the representation 
\begin{equation}
G_B(x,y) = -\fraka(x-y)+\sum_{z\in\partial B}H^B(x,z)\fraka(y-z)
\end{equation}
gives
\begin{equation}
\var (M_{ n,k})= \sum_{z\in\partial B(N')}H^{B(N' )}(0,z)\fraka(z)-\sum_{z\in\partial B(N'')}H^{B(N'')}(0,z)\fraka(z).
\end{equation}
Now substitute the asymptotic form \eqref{E:3.14} and notice that the terms arising from~$c_0$ exactly cancel, while those from the error $O(|x|^{-2})$ are uniformly bounded. Concerning the terms arising from the term $g\log|x|$, here we note that
\begin{equation}
\sup_{N\ge1}\,\Bigl|\sum_{z\in\partial B(N)}H^{B(N)}(0,z)\log|z| - \log N\Bigr|<\infty,
\end{equation}
which follows by using  $\log|x+r|-\log|x|=O(|r|/|x|)$ to approximate the sum by an integral. Hence we get
\begin{equation}
\begin{aligned}
G_{B(N')}(0,0)-G_{B(N'')}(0,0) &= g\log(N')-g\log(N'')+O(1) 
\\&= g\log(b)(n-k)+O(1)
\end{aligned}
\end{equation}
with $O(1)$ bounded uniformly in~$N\ge1$, $n\ge1$ and~$k=1,\dots,n-1$.
\end{proofsect}

Using the above setup, pick two (possibly real) numbers~$1<r_1<r_2<b$ and define 
\begin{equation}
A_{ n,k}:= B\bigl(\lfloor r_2 b^{k} N\rfloor \bigr)\smallsetminus B\bigl(\lceil r_1 b^{k} N\rceil\bigr)^{\circ}\,.
\end{equation} 
  The point of working with the conditional expectations of $\chi_{N'}$ evaluated at the origin is that these expectations represent very well the typical value of the same conditional expectation anywhere on~$A_{ n,k}$. Namely, we have:

\begin{lemma}
\label{lem:anchoring}
Denote
\begin{equation}
\Delta_{n} := \max_{k=1,\dots,n-1}\,\,\max_{v \in A_{ n,k}}\,\Bigl|
M_{ n,k} - 
\E\bigl(\chi_{N', v} \,\big|\,\chi_{N', v}: v \in {\textstyle\bigcup_{n \geq j \geq n - k}}\partial B(b^jN)\bigr)\Bigr|.
\end{equation}
For each~$b\ge2$  (and each~$r_1,r_2$ as above)  there are $\wt C>0$  and $N_0\ge1$  such that for all  $N\ge N_0$ and all  $n\ge1$,
\begin{equation}
\label{lem:anchoring_eq}
\P\bigl(\Delta_{n} \geq \wt C\sqrt{\log n}\,\bigr) \leq 1/ n^2\,.
\end{equation}
\end{lemma}

\begin{proofsect}{Proof}
Denote $ A_{n,k}'  := B(b^{ k+1}N) \smallsetminus B(b^{ k}N)$ and for~$v\in A'_{n,k}$ abbreviate
\begin{equation}
\label{E:3.46}
\tilde\chi_{k,v} := \E\bigl(\chi_{N', v} \,\big|\,\chi_{N', v}\colon v \in {\textstyle\bigcup_{n \geq j \geq n-k}}\partial B(b^jN)\bigr).
\end{equation}
From the Gibbs-Markov property we also have 
\begin{equation}
\tilde\chi_{k,v} = \E\bigl(\chi_{N', v} \,\big|\,\chi_{N', v}\colon v \in \partial A_{n,k}'\bigr)\,, 
\qquad v\in A_{n,k}'.
\end{equation}
As soon as~$N$ is sufficiently large, the domains $A:=B(N')$, $B:= A_{n,k}'$ and $\tilde B:=A_{ n,k}$ obey condition~\eqref{E:3.19}  with some~$\delta\ge1$ for all~$n\ge1$ and all $k\in\{1,\dots,n-1\}$. 
Corollary~\ref{cor:field_smoothness1}  then  gives
\begin{equation}
\label{E:3.48}
\P\Bigl(\,\max_{u,v\in A_{ n,k}}\bigl|\tilde\chi_{k,v}-\tilde\chi_{k,u}\bigr|>c+t\Bigr)\le2\e^{-\tilde c t^2}
\end{equation}
for some constants $c,\tilde c>0$ independent of~$N$, $n$ and~$k$. This shows that the oscillation of $\tilde\chi_k$ on~$A_{ n,k}$ has a uniform Gaussian tail, so in order to bound $M_{ n,k}-\tilde\chi_{k,v}=\tilde\chi_{k,0}-\tilde\chi_{k,v}$ uniformly for~$v\in A_{ n,k}$, it suffices to show that, for just one~$v\in A_{ n,k}$, also $\tilde\chi_{k,v}-\tilde\chi_{k,0}$ has such a tail. Since this random variable is a centered Gaussian, it suffices to estimate its variance. Here \eqref{E:3.46} gives
\begin{equation}
\label{}
\var\bigl(\tilde\chi_{k,v}-\tilde\chi_{k,0}\bigr)\le \var\bigl(\tilde\chi_{k-1,v}-\tilde\chi_{k-1,0}\bigr).
\end{equation}
Corollary~\ref{cor:field_smoothness1} can now be applied with $A:=B(N')$, $B:= B(b^{k+1}N)$ and $\tilde B:=B(\lfloor r_2 b^k N\rfloor)$ to bound the right-hand side by a constant uniformly in~$N$, $n$ and~$k=1,\dots,n-1$. Combined with \eqref{E:3.48}, the union bound shows
\begin{equation}
\P\Bigl(\,\max_{v\in A_{ n,k}}\bigl|\tilde\chi_{k,v}-M_{n,k}\bigr|>c'+t\Bigr)\le2\e^{-\tilde c' t^2}
\end{equation}
with $c',\tilde c'\in(0,\infty)$ independent of~$N$, $n$ and~$k$.
Another use of the union bound now yields \eqref{lem:anchoring_eq}, thus proving the claim.
\end{proofsect}

\subsection{Cardinality of the level sets}
 In this subsection, we estimate the cardinality of the sets of points where the GFF  equals  (roughly) a  prescribed  multiple of its absolute maximum. 
Recall that  from \cite{BZ10,BDZ14} we know that the family of random variables 
\begin{equation}
\label{E:3.51}
\max_{v\in B(N)}\chi_{N,v} - 2\sqrt g\log N -   \frac{3}{4}  \sqrt g\log\log N
\end{equation}
 is  tight as~$N\to\infty$. The level sets we are interested in are of the form 
\begin{equation}
\label{eq-def-level-set}
\mathcal A_{N, \alpha} := \Bigl\{v\in B(\lfloor N/2\rfloor)\colon \chi_{N,v} \in( \alpha \tilde m_N, \alpha 
\tilde m_N + 1)\Bigr\}\,,
\end{equation}
 where $\tilde m_N:=2\sqrt g\log N$ and~$\alpha\in(0,1)$. Our conclusion about these is as follows:

\begin{theorem} \label{thm-levelset}
For any $\alpha_0\in(0,1)$ there are  $c=c(\alpha_0)>0$ and  $\kappa=\kappa(\alpha_0)>0$  such that for all $0\leq \alpha_N\leq \alpha_0$ and all $\delta \geq \mathrm{e}^{-(\log N)^{1/4}}$ the bound
\begin{equation}
\label{E:3.52}
\P\bigl(|\mathcal A_{N, \alpha_N}| \leq \delta \E |\mathcal A_{N, \alpha_N}|\bigr)
\leq c\delta^{ \kappa}\,
\end{equation}
holds for all~$N$ sufficiently large.
The same holds  also for the GFF on $ B(N) \smallsetminus \{0\}$.
\end{theorem}

 The exponent linking  the cardinality of the level set  to  the  linear  size of the  underlying  domain has been computed in~\cite{daviaud} building on \cite{BDG}  where the  leading-order growth-rate of the absolute maximum  was determined.  While much progress on the maxima of the GFF has been made recently, notably with the help of modified branching random walk (MBRW) introduced in 
\cite{BZ10}, the methods  used in these studies  do not seem to  be of much use  here. Indeed, in order to make use of the modified branching random walk one needs to invoke a comparison between the GFF and MBRW, which is conveniently available for the maximum (using Slepian's lemma \cite{Slepian62}), but does not seem to extend to the cardinality of the level sets. 

Another possible approach  to consider is the intrinsic dimension of the level sets (see \cite{CDD13}), but this would not give a sharp estimate as we desire. Our approach  to  Theorem~\ref{thm-levelset} is  much simpler, being  a combination of the second moment method (which directly applies to GFF) and the ``sprinkling method'' which was employed in \cite{Ding11} in the context of the GFF. We remark that the second moment method has  recently been used to prove that a suitably-scaled size of the whole level set admits a non-trivial distributional limit~\cite{BL4}.

\begin{proofsect}{Proof of Theorem~\ref{thm-levelset}}
 The proof is actually quite easy when $\alpha<1/\sqrt2$, but becomes more complicated in the complementary regime of~$\alpha$. This is due to well known failure of the second-moment method in these problems and the need for a suitable truncation to make it work again. The first half of the proof thus consists of the set-up, and control, of the truncation.

Pick $N\ge1$ large and let $n:=\max\{k\colon 2^k< N/8\}$.
For $v\in B(\lfloor N/2\rfloor)$, write $B(v, L):=v+B(L)$ and, for $k = 1, \ldots, n$, set, abusing of our earlier notation, $A_{n,k}(v) := B(v, 2^{k+1}) \smallsetminus B(v, 2^k)$. Note that $A_{n,k}(v)\subset B(\lfloor 3N/4\rfloor)$ for all~$k=1,\dots,n$. Then for all $x,y\in A_{n,k}(v)$ and with~$g:=2/\pi$,  
\begin{equation}
\E (\chi_{N,v} \chi_{N,x}) = g(\log 2)(n-k) + O(1)
\end{equation}
and
\begin{equation}
\E (\chi_{N,x} \chi_{N,y}) \ge g(\log 2)(n-k) + O(1)
\end{equation}
hold with $O(1)$ uniformly bounded in~$N$ and~$x,y$ as above. 
Next denote 
\begin{equation}
\bar \chi_{N,k,v} := \frac{1}{|A_{n,k}(v)|} \sum_{u\in A_{n,k}(v)} \chi_{N,u}\,.
\end{equation}
A straightforward calculation then shows that
\begin{equation}
\label{eq-var-bar-X}
\var (\bar \chi_{N,k,v}) = g(\log 2)(n-k) + O(1)\,
\end{equation}
and
\begin{equation}
\E (\bar \chi_{N,k,v} \chi_{N,v})= g(\log 2)(n-k) + O(1),
\end{equation}
again, with $O(1)$ uniform in~$N$. It follows that there are numbers $a_x=a_{N,k,v,x}$ with $|a_{x} - 1| = 
O(1/(n-k))$ and a Gaussian process $Y_x=Y_{N,k,v,x}$ which is independent of $\bar \chi_{N,k,v}$ and obeys $\var(Y_x)=g(\log 2)k+O(1)$ such that
\begin{equation}
\label{eq-X-Y}
\chi_{N,x} = a_x\bar \chi_{N,k,v} + Y_x,\qquad x\in \{v\} \cup 
A_{n,k}(v).
\end{equation}
Further, we have that 
\begin{equation}\label{eq-cov-Y}
\max_{x\in A_{n,k}(v)}\E (Y_v Y_x)= O(1)
\end{equation}
again with~$O(1)$ uniform in~$N$.

For $\epsilon>0$, $r>0$ and $0\leq \alpha_N \leq \alpha_0$, define the event
\begin{multline}
\qquad
E_{v, \epsilon, r, \alpha_N} := \bigl\{\chi_{N,v} \in (\alpha_N \tilde m_N, \alpha_N \tilde m_N+1)\bigr\} 
\\
\cap \bigcap_{k=1}^n \Bigl\{\bar \chi_{N,k,v} 
\leq \alpha_N \frac{n-k}{n} \tilde m_N  + \epsilon [k \wedge (n-k)]+ r\Bigr\}\,.
\qquad
\end{multline}
We claim that for $\epsilon := \frac{(1-\alpha_0)}{10} $ and $r := r_{\alpha_0}$ sufficiently large, we have
\begin{equation}
\label{eq-E-v}
\P\bigl(E_{v, \epsilon, r, \alpha_N}\bigr) \geq \tfrac{1}{2}\P\bigl(\chi_{N,v} \in (\alpha_N \tilde m_N, \alpha_N \tilde m_N+1)\bigr)\,.
\end{equation}
In order to prove \eqref{eq-E-v}, note that by \eqref{eq-var-bar-X}
\begin{equation}
\E \bigl(\bar \chi_{N,k,v}\,\big|\, \chi_{N,v} \in (\alpha_N \tilde m_N, \alpha_N \tilde m_N+1)\bigr) =  \alpha_N \frac{n-k}{n}\tilde m_N + O(1)
\end{equation}
and
\begin{equation}
\var \bigl(\bar \chi_{N,k,v}\,\big|\, \chi_{N,v}\bigr) \leq \frac{4(n-k) k}{n}\,.
\end{equation}
Abbreviating $s_k := \alpha_N \frac{n-k}{n} \tilde m_N  + \epsilon [k \wedge (n-k)]+ r$, from these observations we have 
\begin{multline}
\qquad
\sum_{k=1}^n\P\Bigl( \bar \chi_{N,k,v} \geq s_k  \,\Big|\, \chi_{N,v} \in (\alpha_N 
\tilde m_N, \alpha_N \tilde m_N+1)\Bigr) 
\\
\leq \sum_{k=1}^n\e^{-\epsilon((n-k)\wedge k + r + O(1))/100} \leq 1/2\,,
\qquad
\end{multline}
 where
the last inequality holds for all~$r\ge r(\alpha_0)$ where $r(\alpha_0)\in(0,\infty)$.  This yields \eqref{eq-E-v}. 

Now we are ready to apply the second moment method. We will work with 
\begin{equation}
\mathcal Z := \sum_{v\in B(\lfloor N/2\rfloor)} \mathbf 1_{E_{v, \epsilon, r, \alpha_N}}
\end{equation}
From \eqref{eq-E-v} and a calculation for the Gaussian distribution we get
\begin{equation}\label{eq-expectation-Z}
\E \mathcal Z \geq \frac{1}{2}\E |\mathcal A_{N, \alpha_N}| \ge \frac{c}{\sqrt{n}} 4^{(1-\alpha_N^2) n}\,
\end{equation}
for some constant~$c>0$.
Our next task is a derivation of a suitable upper bound on $\var \mathcal Z$.  From \eqref{eq-X-Y} and \eqref{eq-cov-Y} we get that, for any $v\in B(\lfloor N/2\rfloor)$ and with~$c_r>0$ a constant depending on~$r$ but not on~$v$ or~$N$,
\begin{equation}
\begin{aligned}
\sum_{u\in B(\lfloor N/2\rfloor)}\P(&E_{u, \epsilon, r, \alpha_N}\cap E_{v, \epsilon, r, \alpha_N})
\\
\leq& \sum_{k = 1}^n \sum_{u\in A_{n,k}(v)} \P\Bigl(\chi_{N,u},\chi_{N,v} \in (\alpha_N \tilde m_N, \alpha_N \tilde m_N + 1),  \bar \chi_{N,k,v} 
\leq x_{k} \Bigr)  
\\
\leq& \sum_{k=1}^n \sum_{u\in A_{n,k}(v)} \int_{-\infty}^{x_\ell } \P\Bigl(Y_{v} \wedge Y_{u} \geq \alpha_N \tilde m_N - s\Bigr) \P(\bar \chi_{N,k,v}  \in \text{d}s) 
\\
\leq & \,c_r\sum_{k = 
1}^n \frac{1}{\sqrt{n-k}} \Bigl(\frac{1}{\sqrt{k}}\Bigr)^2 4^{-\alpha_N^2(n-k)} 4^{(1-2\alpha_N^2) k}  4^{2\epsilon \alpha_N [(n-k)\wedge k]}\,.
\end{aligned}
\end{equation}
Here the last inequality follows from the fact that, once we write the integral using the explicit form of the law of $\bar\chi_{N,k,v}$, the integrand is maximized at $s := s_{k}$ and decays exponentially when~$s$ is away from $s_k$. Combined with \eqref{eq-expectation-Z}, the preceding inequality implies that
\begin{equation}
\frac{\var \mathcal Z}{(\E \mathcal Z)^2} 
\leq c_r\sum_{k = 1}^n \frac{n}{\sqrt{n-k}} \Bigl(\frac{1}{\sqrt{k}}\Bigr)^2 4^{-(1-\alpha_N^2)(n-k)} 4^{2\epsilon \alpha_N [(n-k)\wedge k]} 
= O(1)\,.
\end{equation}
This implies
\begin{equation}\label{eq-second-moment-level-set-first-bound}
\P(\mathcal Z \geq \E \mathcal Z) \geq c\,
\end{equation}
for some $c=c(\alpha_0)>0$ sufficiently small uniformly in~$N\ge N_1$ for some~$N_1$ large.

It remains to enhance the lower bound in~\eqref{eq-second-moment-level-set-first-bound} to a number sufficiently close to one.  To this end, pick an integer $M$ with~$N_1\leq M \leq \e^{(\log N)^{1/4}}$, let $L:=\lfloor N/(2M)\rfloor$ and consider a collection of boxes $V_1,\dots,V_{L^2}$ of the form~$V_i:=v_1+B(M)$ contained in~$B(\lfloor N/2\rfloor)$. For $u\in V_i$, $i=1, \ldots, L^2$, define the coarse fields 
\begin{equation}
\label{E:3.73}
\chi^c_{N,i,u} = \E \bigl(\chi_{N,u} \,\big|\, \chi_{N,x}\colon x\in \partial V_i \bigr)\,.
\end{equation}
By Lemma~\ref{lem:max_expectation} and \cite[Lemma 3.10]{BDZ14}, we get that 
\begin{equation}
\E \max_{v\in V_i} |\chi^c_{N,i,v} - \chi^c_{N, i,v_i}| 
\leq O(1)\,.
\end{equation}
In addition, as is easy to check, $\var \chi^c_{N,i,v_i} \leq 4 \log M$. Introducing the event
\begin{equation}
\mathcal E := \bigl\{\chi^c_{N,i,v} \geq - 40 \log M\colon  v\in V_i, 1\leq i\leq L^2\bigr\},
\end{equation}
we obtain that 
\begin{equation}\label{eq-Lambda-c}
\P(\mathcal E^c) 
= O(M^{-1})\,.
\end{equation}
Conditioning on $\mathcal E$ and on the values $\{\chi_{N,v}\colon v\in \partial V_i, 1\leq i\leq L^2\}$, the GFF in each square of~$V_i$ are independent of each other. Further, the Gaussian field on $V_i$ dominates the field obtained from subtracting $40\log M$ from the GFF on $V_i$ with Dirichlet boundary condition on $\partial V_i$. Write
\begin{equation}
\mathcal A_{N, \alpha_N, i} := \bigl\{v\in V_i\colon \chi_{N,v} \in( \alpha_N \tilde m_N, \alpha_N 
\tilde m_N + 1)\bigr\}.
\end{equation}
 By a straightforward first moment computation, we see that
\begin{equation}
\E |\mathcal A_{N, \alpha_N}| \leq M^{400} \E |\mathcal A_{N, \alpha_N + 40\log M/\tilde m_N, i}|\,.
\end{equation}
Therefore, applying \eqref{eq-second-moment-level-set-first-bound} to $V_i$ we get that
\begin{equation}
\P\bigl( |\mathcal A_{N, \alpha_N, i}| \geq M^{-400} \E |\mathcal A_{N, \alpha_N}|\,\big| \,\mathcal E\bigr) \geq c\,.
\end{equation}
By conditional independence, we then get that
\begin{equation}
\P\bigl( \max_{1\leq i\leq L^2}|\mathcal A_{N, \alpha_N, i}| \geq M^{-400} \E |\mathcal A_{N, \alpha_N}|\,\big|\,\mathcal E) \geq 1  - (1 -c)^{L^2}\,.
\end{equation}
Combined with \eqref{eq-Lambda-c}, it gives
\begin{equation}
\P\bigl(|\mathcal A_{N, \alpha_N}| \geq M^{-400} \E |\mathcal A_{N, \alpha_N}|\bigr) \geq 1 - O(M^{-1}) -  (1 -c)^{L^2}.
\end{equation}
Choosing $M$ so large that $\delta<M^{-400}<2\delta$ (assuming that~$\delta$ is sufficiently small), this readily gives the claim for the GFF on~$B(N)$ with Dirichlet boundary condition. 

In the case that the GFF on $B(N)\smallsetminus\{0\}$, the same calculation goes through by considering instead the level set restricted to the square $(\lfloor N/4\rfloor,0)+B(\lfloor N/2\rfloor)$ and replacing~$\chi_N$ in~\eqref{E:3.73} by~$\eta$. We leave further details to the reader. 
\end{proofsect}

\subsection{A  non-Gibbsian  decomposition of GFF on a square}
As a final item of concern in this section we note that, apart from the Gibbs-Markov property, our proofs will also make use of another decomposition of the GFF which is based on a suitable decomposition of the Green function. This decomposition will be  
of crucial importance  for the development of the RSW  theory  in Section~\ref{sec:RSW}.

\begin{lemma}
\label{lem:field_decomp_smooth}
Let $\{\chi_{N,v}\}_{v \in B(N)}$ be the GFF on~$B(N)$ with Dirichlet boundary condition. Then there exist two independent, centered Gaussian fields $\{Y_{N,v}\}_{v \in B(N)}$ and $\{Z_{N,v}\}_{v \in B(N)}$ such that the following hold:
\settowidth{\leftmargini}{(1111)}
\begin{enumerate}
\item[(a)] $\chi_{N} = Y_{N} + Z_{N}$ a.s.
\item[(b)] $\var(Y_{N,v}) = O(\log \log N)$ uniformly for all $v \in B(N)$.
\item[(c)] $\var(Z_{N,v} - Z_{N,v}) = O(1/\log N)$ uniformly for all $u, v \in B(\lceil N/2 \rceil)$ such that $u \sim v$.
\end{enumerate}
The distribution of $\{Z_{v,N}\}_{v \in B(N)}$ is invariant  under reflections and rotations  that preserve~$B(N)$.
\end{lemma}

\begin{proofsect}{Proof}
Throughout the proof of the current lemma, we let $\{S_t: t \geq 0\}$ be the \emph{lazy} discrete-time simple symmetric random walk on $\mathbb Z^2$ that, at each time, stays put at its current position with probability~$1/2$ or  moves to one of its neighbors with probability~$1/8$.  
We denote by~$P^v$  the law of the walk with~$P^v(S_0 := v)=1$ and write~$E^v$ to denote the expectation with respect to~$P^v$.  Let $\tau$ be the first hitting time to 
the boundary $\partial B(N)$. It is clear~that 
\begin{equation}
\E (\chi_{N, v} \chi_{N, u}) = \frac{1}{2} \sum_{t=0}^\infty P^v(S_t = u, \tau \geq t)\,.
\end{equation}
In addition, thanks to laziness of~$S_t$, the matrix $(P^v(S_t = u, \tau \geq t))_{u, v\in B(N)}$ 
is non-negative definite for each $t\geq 0$,  and so is its sum over~$t$ in any subset of non-negative integers.  Therefore, there are independent centered 
Gaussian fields $\{Y_{N, v} \colon v\in B(N)\}$ and $\{Z_{N, v}\colon v\in B(N)\}$ such that
\begin{equation}
\E (Y_{N, v}  Y_{N, u}) = \frac{1}{2} \sum_{t=0}^{\lfloor\log N\rfloor^2}P^v(S_t = u, \tau \geq t)
\end{equation}
and
\begin{equation}
\E (Z_{N, v}  Z_{N, u}) = \frac{1}{2} \sum_{t=\lfloor\log N\rfloor^2 + 1}^{\infty}P^v(S_t = u, \tau \geq t)\,.
\end{equation}
At this point, it is clear that we can couple the 
processes together so that Property (a) holds. Property (b) holds by crude computation  which shows   
\begin{equation}
\var Y_{N, v} \leq \sum_{t=0}^{\lfloor\log N\rfloor^2}P^v(S_t = v)  
\leq O(1)\sum_{t=0}^{\lfloor\log N\rfloor^2} \frac{1}{t+1} 
= O(\log\log N)\,.
\end{equation}
It remains to verify Property~(c). For any $u, v\in B(\lceil N/2\rceil)$ and $u\sim v$, we have that 
\begin{equation}
\begin{aligned}
|\E &Z_{N, v}^2 - \E Z_{N, v} Z_{N, u}| 
\\
&= \Bigl |\sum_{t=\lfloor\log N\rfloor^2 + 1}^{\infty}P^v(S_t = v, \tau \geq t) -  \sum_{t=\lfloor\log N\rfloor^2 + 1}^{\infty}P^v(S_t = u, \tau \geq t) \Bigr| 
\\
 &\leq  \sum_{t=\lfloor\log N\rfloor^2 + 1}^{\infty} \bigl|P^v(S_t = v)  - P^v(S_t = u)\Big| 
 +   \sum_{t=0}^\infty E^v\bigl|\,P^{S_\tau} (S_t = v) - P^{S_\tau} (S_t = u)\bigr|\,.
\end{aligned}
\end{equation}
 Since 
\begin{equation}
\bigl|P^v(S_t = v)  - P^v(S_t = u)\bigr| = O(t^{-3/2})\,,
\end{equation}
(see, e.g., \cite[Exercise~2.2]{Lawler10}), the first term on the right hand side is bounded by $O(1/\log N)$. 
The second term is $O(1/N)$ by \cite[Theorem 4.4.6]{Lawler10} and the fact that $u \in B(\lceil N/2 \rceil)$.  This completes the verification of Property (c).
\end{proofsect}




\section{A RSW result for effective resistances} 
\label{sec:RSW}\noindent
Having dispensed with preliminary considerations, we are now ready to develop a RSW theory for effective resistances across rectangles. Throughout we write, 
for $N,M\ge1$, 
 \begin{equation}
B(N,M):=\bigl([-N,N] \times [-M,M] \bigr)\cap\Z^2
\end{equation} 
for the rectangle of  $(2N+1) \times (2M+1) $ vertices centered at the origin. Recall that $B(N, N) = B(N)$.  The principal outcome of this section are Corollary~\ref{cor-4.3} and Proposition~\ref{prop-4.7}.  In Corollary~\ref{cor-4.18}, these  yield the proof of one half of Theorem~\ref{thm-effective-resistance}.  The proof of the other half  comes only at the very end of the paper (in Section~\ref{sec-proofs}).

\subsection{Effective resistance across squares}
\label{sec:resistance_bound}
In Bernoulli percolation, the RSW theory is a loose term for a collection of methods for extracting uniform lower bounds on the probability that any rectangle of a given aspect ratio is crossed by an occupied path along its longer dimension. The starting point is a duality-based lower bound on the probability of a left-right crossing of a square. In the present context, the crossing probability is replaced by resistance across a square and duality by consideration of a reciprocal network.
An additional complication is that our problem is intrinsically spatially-inhomogeneous. This means that all symmetry arguments, such as rotations and reflections, require special attention to where the underlying domain is located. In particular, it will be advantageous to work with the GFF on finite squares instead of the pinned field in all of~$\Z^2$.

If $S$ is a rectangular domain in~$\Z^2$, we will write $\lb S$, $\db S$, $\rb S$ and~$\ub S$ to denote the sets of vertices in  $ S$  that have a neighbor in  $\Z^2\smallsetminus S$  to the left, down, right and up of them, 
respectively.  (Notice that, unlike~$\partial S$, these ``boundaries'' are subsets of~$S$.)  Given any field $\chi  = \{\chi_v\}_{v \in S}$ recall that $S_\chi$ denotes the network on~$S$ associated with~$\chi$. We then abbreviate
\begin{equation}
R_{\LR;S,\chi}:=R_{S_{\chi}}\bigl(\lb S,\rb S\bigr)
\end{equation}
and
\begin{equation}
R_{\UD;S,\chi}:=R_{S_{\chi}}\bigl(\ub S,\db S\bigr).
\end{equation}
Our first estimate concerning these quantities is:

\begin{proposition}[Duality lower bound]
\label{prop:zero_mean}
Let $\chi_M$ denote  the {\rm GFF} on~$B(M)$ with Dirichlet boundary conditions.
There is $\cspecial  = \cspecial (\gamma) \in (0, \infty)$ and for each~$\epsilon>0$ there is~$N_0=N_0(\epsilon,\gamma)$ such that for all~$N\ge N_0$ and all $M\ge 2N$,
\begin{equation}
\label{zero_mean_bound1}
\P\Bigl(R_{\LR;B(N),\chi_M} \leq \e^{\cspecial \log\log M}\Bigr) \geq \frac12-\epsilon\,.
\end{equation}
The same result holds also for $R_{\UD;B(M),\chi_M}$, which is equidistributed to $R_{\LR;B(N),\chi_M}$. 
\end{proposition}

The proof requires some elementary observations that will be useful later as well:

\begin{lemma}
\label{lem:resistance_decomp}
Let $A$ be a finite subset of $\Z^2$ and $\chi_1 = \{\chi_{1, v}\}_{v \in A}$, $\chi_2 = \{\chi_{2, v}\}_{v 
\in A}$ be two random fields on $A$. Then for any $u, v \in A$ we have,
\begin{equation}
\label{E:4.1}
R_{A_{\chi_1 + \chi_2}}(u, v) \leq R_{A_{\chi_1}}(u, v) \,\,\max_{\begin{subarray}{c}
u', v' \in A\\ u' \sim v'
\end{subarray}} \e^{-\gamma(\chi_{2, u'} + \chi_{2, v'})}\,.
\end{equation}
Furthermore,
\begin{equation}
\label{E:4.2}
\E\bigl(R_{A_{\chi_1 + \chi_2}}(u, v) \,\big|\, \chi_1\bigr) \leq R_{A_{\chi_1}}(u, v) \,\,\max_{\begin{subarray}{c}
u', v' \in A\\ u' \sim v'
\end{subarray}} \E\bigl(\e^{-\gamma(\chi_{2, u'} + \chi_{2, v'})} \,\big|\, \chi_1\bigr)\,
\end{equation}
and
\begin{equation}
\label{E:4.3}
\E\bigl(C_{A_{\chi_1 + \chi_2}}(u, v) \,\big|\, \chi_1\bigr) \leq C_{A_{\chi_1}}(u, v) \,\,\max_{\begin{subarray}{c}
u', v' \in A\\ u' \sim v'
\end{subarray}}
 \E\bigl(\e^{\gamma(\chi_{2, u'} + \chi_{2, v'})} \,\big|\, \chi_1\bigr)\,.
\end{equation}
\end{lemma}

\begin{proofsect}{Proof}
Let $\theta$ be a unit flow from~$u$ to~$v$. Then \eqref{eq-resistance-flow} implies 
\begin{equation}
R_{A_{\chi_1 + \chi_2}}(u, v) \leq \sum_{u', v' \in A, u' \sim v'}[\theta_{(u', v')}]^2 \e^{-\gamma(\chi_{1, u'} + \chi_{1, v'})} \e^{-\gamma(\chi_{2, u'} + \chi_{2, v'})}\,.
\end{equation} 
Hereby we get \eqref{E:4.1} by bounding the second exponential by its maximum over all pairs of nearest neighbors in~$A$ and optimizing over~$\theta$. The estimate \eqref{E:4.2} is obtained similarly; just take the conditional expectation before optimizing over~$\theta$. The proof of \eqref{E:4.3} exploits the similarity between \eqref{eq-resistance-flow} and \eqref{eq-conductance-minimization} and is thus completely analogous. 
\end{proofsect}

\begin{proofsect}{Proof of Proposition~\ref{prop:zero_mean}}
Our aim is to use the fact that, in any Gaussian network, the resistances are equidistributed to the conductances. We will apply this in conjunction with the estimate in Lemma~\ref{remark_dual_couple}.  Unfortunately, this estimate requires a hard bound on the maximal ratio of resistances at neighboring edges. These ratios would be undesirably too large if we work with the GFF network directly; instead we will invoke the decomposition of~$\chi_{M}$ into the sum of Gaussian fields $Y_{M} = \{Y_{M, v}\}_{v \in B(N)}$ and $Z_{M} = \{Z_{M, v}\}_{v \in B(N)}$ as stated in Lemma~\ref{lem:field_decomp_smooth} and apply Lemma~\ref{remark_dual_couple} to the network associated with $Z_{M}$ only.

We begin by estimating the oscillation of $Z_{M}$ across neighboring vertices. 
From property~(c) in the statement of Lemma~\ref{lem:field_decomp_smooth} and a standard bound on the expected maximum of centered Gaussians, we first get 
\begin{equation}
\sup_{N\ge1}\,\E\Bigl(\, \max_{\begin{subarray}{c}
u, v \in B(N)\\ |u - v|_1 \leq 2
\end{subarray}}
(Z_{M, u} - Z_{M, v}) \Bigr)<\infty\,.
\end{equation}
Using this bound and property~(c), Lemma~\ref{lem:Borell_ineq} shows that  for each~$\epsilon>0$ there is $c_1\in\R$ such that for all $N\ge1$,
\begin{equation}
\label{E:4.11a}
\P\Bigl(\,\,\max_{\begin{subarray}{c}
u, v \in B(N)\\ |u - v|_1 \leq 2
\end{subarray}}(Z_{M, u} - Z_{M, v}) \geq c_1\Bigr) \leq \epsilon\,.
\end{equation}
Now observe that the pairs $(\lb B(N), \rb B(N))$ and $(\ub B(N), \db B(N))$ satisfy the conditions of 
Lemma~\ref{remark_dual_couple}.  Using $R^\star_{\UD; B(N),{Z_{M}}}$ to denote the top-to-bottom resistance in the reciprocal network, combining \eqref{eq:dual_couple_conduct} with the last display yields 
\begin{equation}
\P\Bigl(R_{\LR; B(N),{Z_{M}}} R^\star_{\UD; B(N),{Z_{M}}} \leq 64 \e^{2c_1\gamma}\Bigr) \geq 1-\epsilon\,.
\end{equation}
A key point of the proof is that, since the law of $Z_M$ is symmetric with respect to rotations of~$B(M)$, the fact that $Z_M\laweq -Z_M$ implies
\begin{equation}
R^\star_{\UD; B(N),{Z_{M}}} \laweq R_{\LR; B(N),{Z_{M}}}. 
\end{equation}
The union bound then shows
\begin{equation}
\label{E:4.10}
\P\Bigl(R_{\LR; B(N),{Z_{M}}} \leq 8\e^{c_1\gamma}\Bigr) \geq  \frac{1-\epsilon}2\,.
\end{equation}
Lemma~\ref{lem:resistance_decomp} and the independence of~$Y_{M}$ and $Z_{M}$ now give 
\begin{equation}
\label{E:4.11}
\E\bigl(R_{\LR; B(N),{\chi_{M}}}\,\big|\, Z_{M}\bigr) \leq R_{\LR; B(N),{Z_{M}}}\,\,\max_{\begin{subarray}{c}
u, v \in B(N)\\ u \sim v
\end{subarray}}
\E \e^{-\gamma (Y_{M,u} + Y_{M,v})}.
\end{equation}
Lemma~\ref{lem:field_decomp_smooth} shows $\var Y_{M,v} \leq c'\log \log M$ for some constant~$c'\in(0,\infty)$ and so the maximum on the right of \eqref{E:4.11} is at most $\e^{2c'\gamma^2\log \log M}$. 
Taking $\cspecial >2c'\gamma^2$,  
we get \eqref{zero_mean_bound1} (for~$N$ sufficiently large) from \twoeqref{E:4.10}{E:4.11} and Markov's inequality.
\end{proofsect}

With only a minor amount of additional effort, we are able to conclude a uniform \emph{lower} bound for the resistance across rectangles.

\begin{corollary}
\label{cor-4.3}
Let~$\cspecial $ be as in Proposition~\ref{prop:zero_mean}. For each~$\epsilon>0$ there is $N_0'=N_0'(\gamma,\epsilon)$ such that for all $N\ge N_0'$, all $M\ge 16N$ and all translates~$S$ of $B(4N,N)$ contained in~$B(M/ 2)$, we have
\begin{equation}
\P\bigl(R_{\LR;S,{\chi_{M}}} \geq \e^{-2\cspecial \log\log M}\bigr) \geq \frac12-\epsilon\,.
\end{equation}
The same applies to $R_{\UD;S,{\chi_{M}}}$ for any translate~$S$ of $B(N,4N)$ contained in~$B(M/ 2)$.
\end{corollary}

\begin{proofsect}{Proof}
Replacing effective resistances by effective conductances in the proof of Proposition~\ref{prop:zero_mean} (and  relying on Lemma~\ref{lem:dual_couple} instead of Lemma~\ref{remark_dual_couple}) yields
\begin{equation}
\label{zero_mean_bound1a}
\P\Bigl(R_{\LR;B(N),\chi_M} \geq \e^{-\cspecial \log\log M}\Bigr) \geq \frac12-\epsilon
\end{equation} 
for all~$N\ge N_0$. Since
\begin{equation}
R_{\LR;B(4N),\chi_M}\le R_{\LR;B(4N,N),\chi_M}
\end{equation}
this bound extends to the rectangle $B(4N,N)$. Now consider a translate~$S$ of this rectangle that is contained in~$B(M/ 2)$. Taking $M':=8N$ and let $\tilde S$ be the  translate of $B(M')$ that is centered at the same point as~$S$. Considering the Gibbs-Markov decomposition into a fine field $\chi_{\tilde S}^f$ and a coarse field $\chi_{\tilde S}^c$ on $\tilde S$, we then get
\begin{multline}
\qquad
\P\Bigl(R_{\LR;S,\chi_M} \geq \e^{\tilde c\gamma}\e^{-\cspecial \log\log M}\Bigr)
\\
\ge
\P\Bigl(R_{\LR;S,\chi_{\tilde S}^f} \geq \e^{-\cspecial \log\log M'}\Bigr)-\P\Bigl(\max_{u\in S}|\chi_{{\tilde S},u}^c|\le\tilde c\Bigr)\,.
\qquad
\end{multline}
Since~$S$ and~$\tilde S$ are centered at the same point, the first probability is at least $\frac12-\epsilon$ by our extension of \eqref{zero_mean_bound1a} to rectangles. The second probability can be made arbitrarily small uniformly in~$N$ by taking~$\tilde c$ large. The claim follows.
\end{proofsect}

\begin{remark}
Despite our convention that constants such as~$c,\tilde c, c'$, etc may change meaning line to line, the constant~$\cspecial $ will denote the quantity from Proposition~\ref{prop:zero_mean} throughout the rest of this~paper.
\end{remark}

\subsection{Restricted resistances across squares}
As noted already in the introduction, our  approach to the RSW theory is strongly inspired by~\cite{Tassion14} which is itself  based on inductively controlling  the crossing probability (in Bernoulli percolation)  between $\lb B(N)$ and a \emph{portion} of~$\rb B(N)$. We will now setup the relevant objects and notations and prove estimates that will later serve in an argument by contradiction.

For the square $B(N)$ and $\alpha,\beta\in[-N,N]\cap\Z$ with $\alpha\le\beta$,  consider the subset of $\rb B(N)$ defined by 
\begin{equation}
\rb^{[\alpha,\beta]}B(N):=
(\{N\}\times[\alpha,\beta])\cap\Z^2
\end{equation}
Let $\mathcal P_{N;[\alpha,\beta]}$ denote the set of paths in~$B(N)$ that use only the vertices in $((-N,N)\times[-N,N])\cap\Z^2$ except for the initial vertex, which lies in $\lb B(N)$, and the terminal vertex, which lies in~$\rb^{[\alpha,\beta]}B(N)$.
With these notions in place, we now introduce the shorthand
\begin{equation}
\label{E:5.16}
R_{N,[\alpha,\beta],\chi}:=R_{B(N)_\chi}\bigl(\mathcal P_{N;[\alpha,\beta]}\bigr) 
 = R_{B(N)_\chi}\bigl(\lb B(N),\rb^{[\alpha,\beta]}B(N)\bigr). 
\end{equation}
Our first goal is to define a quantity~$\alpha_N$ which will mark, in rough terms, the point of transition of $\alpha\mapsto R_{N,[0,\alpha],\chi_{2N}}$ from large to small values.

\begin{figure}[t]
\centerline{\includegraphics[width=0.45\textwidth]{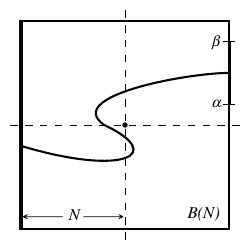}
}
\small 
\vglue0.0cm
\caption{
\label{fig5}
An illustration of the geometric setting underlying the definition of the restricted effective resistance $R_{N,[\alpha,\beta],\chi}$ in~\eqref{E:5.16}.}
\normalsize
\end{figure}

We first need a couple of simple observations. Note that $\mathcal P_{N;[0,N]}\cup \mathcal P_{N;[-N,0]}$ includes all paths starting on $\lb B(N)$ and terminating on $\rb B(N)$. Lemma~\ref{lem:parallel_law} then shows
\begin{equation}
\frac1{R_{\LR; B(N),\chi_{2N}}}
\le \frac1{R_{N,[0,N],\chi_{2N}}}+\frac1{R_{N,[-N,0],\chi_{2N}}}
\end{equation}
while the symmetry of both the law of~$\chi_{2N}$ and  the  square $B(N)$ with respect to the reflection through the~$x$ axis implies $R_{N,[0,N],\chi_{2N}} \laweq R_{N,[-N,0],\chi_{2N}}$.
By Proposition~\ref{prop:zero_mean}, there is~$N_0$ such that
\begin{equation}
\P\bigl(R_{\LR; B(N),\chi_{2N}}>\e^{\cspecial \log\log (2N)}\bigr)\le2/3
\end{equation}
as soon as~$N\ge N_0$. The square-root trick in Corollary~\ref{cor:square_root_trick} then shows
\begin{equation}
\label{eq:Tassion1_1}
\P\bigl(R_{N,[0,N],\chi_{2N}} > 2\e^{\cspecial \log\log(2N)}\bigr) \leq \sqrt{2/3} < 0.82
\end{equation}
as soon as~$N\ge N_0$. 

Next we note that,  by Lemma~\ref{lemma-3.7}, 
\begin{equation}
\sup_{N\ge1}\max_{\begin{subarray}{c}
v\in B(3N/2)\\u\sim v
\end{subarray}}
\var(\chi_{2N,v}-\chi_{2N,u})<\infty.
\end{equation}
Hence, there is~$C'\in(0,\infty)$ such that $\chi:=\chi_{2N}$ obeys
\begin{equation}
\label{E:5.22}
\max_{v \in B(N)}\P\Bigl(\max\{\chi_{v - e_2} - \chi_{v + e_1}, \chi_{v - e_2 + e_1} + \chi_{v - e_2} - \chi_{v} - \chi_{v + e_1}\} \geq C'\Bigr) \leq 0.005
\end{equation}
for all~$N\ge1$. Now set $C_1:=2(2\e^{C'\gamma}+1)$, define $\phi_N\colon\{0,\dots,N\}\to[0,1]$ by
\begin{equation}
\phi_N(\alpha):=\P\bigl(R_{N,[ \alpha ,N],\chi_{2N}}>(4+C_1)\,\e^{\cspecial \log\log(2N)}\bigr)
\end{equation}
and, noting that $\alpha\mapsto\phi_N(\alpha)$ is non-decreasing with $\phi_N(0)<0.82$ (cf \eqref{eq:Tassion1_1}), let
\begin{equation}
\alpha_N := \begin{cases}
\min\bigl\{\alpha\in\{0,\dots,\lfloor N/2\rfloor\}\colon \phi_N(\alpha) > 0.99\bigr\}  &\text{if } \phi_N(\lfloor N/2\rfloor) > 0.99,
\\
\lfloor N/2\rfloor, &\text{otherwise}\,.
\end{cases}
\end{equation}
This definition implies the following inequalities:
 
\begin{lemma}
\label{lem:Tassion1}
For~$C'$ as in~\eqref{E:5.22}, define $C_2:=4(2\e^{C'\gamma}+1)^2$ and let $ \cspecial $,~$N_0$ and~$C_1$ be as above. Then the following two properties hold for all~$N\ge N_0$:
\settowidth{\leftmargini}{(1111)}
\begin{enumerate}
\item[(\rm P1)] For all $\alpha \in\{0,\dots, \alpha_N\}$,
\begin{equation}
\label{E:5.25}
\P\bigl(R_{N,[\alpha,N],\chi_{2N}} \leq 5C_2\e^{\cspecial \log\log(2N)}\bigr) \geq 0.005.
\end{equation}
\item[(\rm P2)] If $\alpha_N < \lfloor N/2\rfloor$, then for all $\alpha \in\{ \alpha_N,\dots,N\}$,
\begin{equation}
\label{E:5.25a}
\P\bigl(R_{N,[\alpha,N],\chi_{2N}} \geq (4 + C_1)\,\e^{\cspecial \log \log (2N)}\bigr) > 0.99
\end{equation}
and
\begin{equation}
\label{E:5.26}
\P\bigl(R_{N,[0,\alpha],\chi_{2N}} \leq  4\e^{\cspecial \log\log(2N)}\bigr) \geq 0.17\,.
\end{equation}
\end{enumerate}
\end{lemma}

\begin{proofsect}{Proof}
We begin with (P1). Since $\phi_N(\alpha) \leq 0.99$ for $\alpha \in \{0,\dots, \alpha_N-1\}$, for all such~$\alpha$ we have
\begin{equation}
\label{eq:Tassion1_2}
\P\bigl(R_{N,[\alpha,N],\chi_{2N}}\leq (4 + C_1)\,\e^{\cspecial \log\log(2N )}) \geq  0.01\,.
\end{equation} 
 In order to deal with $\alpha = \alpha_N$, we will  will need: 

\begin{lemma}
\label{lemma-reroute}
For $\chi:=\chi_{2N}$ and~$v$ being the point with coordinates $(N-1,\alpha_N)$, we have
\begin{equation}
\begin{aligned}
\label{eq:Tassion1_3}
\bigl\{R_{N,[\alpha_N,N],\chi_{2N}} &>  \,C_1 R_{N,[\alpha_N-1,N],\chi_{2N}}\bigr\}
\\
&\!\!\!\subseteq
\Bigl\{\max\bigl\{\chi_{v - e_2} - \chi_{v + e_1}, \chi_{v - e_2 + e_1} + \chi_{v - e_2} - \chi_{v} - \chi_{v + e_1}\bigr\}\ge C'\Bigr\}.
\end{aligned}
\end{equation}
\end{lemma}

\noindent
Deferring the proof of this lemma until after this proof,  we now combine \eqref{eq:Tassion1_2} for $\alpha:=\alpha_N-1$ with \eqref{E:5.22} to get 
\begin{equation}
\begin{aligned}
\label{eq:Tassion1_4}
\P\Bigl(&R_{N,[\alpha_N,N],\chi_{2N}}\le \,(4+C_1)C_1\e^{\cspecial \log\log(2N)}\Bigr)
\\
&\qquad\ge \P\Bigl(R_{N,[\alpha_N-1,N],\chi_{2N}}\le (4+C_1)\,\e^{\cspecial \log\log(2N)},\,\,R_{N,[\alpha_N,N],\chi_{2N}} \le C_1R_{N,[\alpha_N-1,N]}\Bigr)
\\
&\qquad\geq 0.01 - 0.005 = 0.005\,.
\end{aligned}
\end{equation}
Since $(4+C_1)C_1\le 5C_2$, the bound \eqref{E:5.25} holds for~$\alpha:=\alpha_N$  as well.  Thanks to the upward monotonicity of $\alpha\mapsto R_{N,[\alpha,N],\chi_{2N}}$  and also of $\alpha \mapsto \phi_N(\alpha)$,  the inequality extends to all $\alpha\le\alpha_N$.

The first inequality in $\mathrm{(P2)}$ evidently holds by our choice of~$\alpha_N$. As for the second inequality, Lemma~\ref{lem:parallel_law} shows
\begin{equation}
\frac1{R_{N,[0,N],\chi_{2N}}}\le\frac1{R_{N,[0,\alpha],\chi_{2N}}}+\frac1{R_{N,[\alpha,N],\chi_{2N}}}
\end{equation}
and this then implies 
\begin{multline}
\quad
\bigl\{R_{N,[0,N],\chi_{2N}} \leq 2\e^{\cspecial \log\log(2N)}, R_{N,[\alpha,N],\chi_{2N}} > (4 + C_1)\,\e^{\cspecial \log\log(2N)}\bigr\} 
\\
\subseteq \bigl\{R_{N,[0,\alpha],\chi_{2N}} \leq 4\e^{\cspecial \log\log(2N)}\bigr\}.
\quad
\end{multline}
Invoking \eqref{eq:Tassion1_1} and the definition of~$\alpha_N$, the probability of the event on the right is then at most $0.99 - 0.82 = 0.17$.
\end{proofsect}

 We still owe to the reader: 

\begin{proofsect}{Proof of Lemma~\ref{lemma-reroute}}
 Suppose~$\chi$ is such that the complementary event to that on the right of \eqref{eq:Tassion1_3} occurs. We will show that then the complement of the event on the left occurs as well. For this, let 
$\theta$ be the optimal flow  realizing the effective resistance in~\eqref{E:5.16} and let $\theta(x,y)$ denote  its value on edge $(x,y)$. To reduce clutter of indices, write $r(x,y)$ for the resistance of edge $(x,y)$. Abbreviate $t:=v+e_1$, $u:=v-e_2$ and $w:=u+e_1 = (N,\alpha_N-1)$. Our aim is to reroute $\theta(v,t)$ through~$u$ to $w$. Define a flow~$\tilde\theta$ by  setting  $\tilde \theta(v, u) := \theta(v,u) + \theta(v, t)$, $\tilde \theta(u,w) :=  \theta(u,w) + \theta(v, t)$ and $\tilde \theta(v, t): = 0$  and letting~$\tilde\theta_e:=\theta_e$ for all other edges~$e$.  The only edges where $\tilde \theta$ might expend more energy than $\theta$ are the edges $(v, u)$ and $(u,w)$. To bound the change in energy, we note 
\begin{equation}
\begin{aligned}
r(v, u) \tilde \theta(v, u)^2 
&\leq r(v, u)\bigl[\theta(v, u) + \theta(v, t)\bigr]^2 
\\
&\leq 2r(v, u)\theta(v, u)^2  + 2r(v, t)\,\e^{C'\gamma}\theta(v, t)^2 
\end{aligned}
\end{equation}
with the second inequality due to  the containment in the complement of the event on the right of~\eqref{eq:Tassion1_3}.  Similarly we have 
\begin{equation}
r(u, w) 
\tilde \theta(u,  w)^2  
\leq 2r(u,w)\theta(u,w)^2  
+ 2r(v, t)\,\e^{C'\gamma}\theta(v, t)^2 .
\end{equation}
Hence we get $R_{N, [\alpha_N-1, N],\chi_{2N}} \leq (2 + 4\e^{C'\gamma})R_{N, [\alpha_N, N],\chi_{2N}} = C_1 R_{N, [\alpha_N, N],\chi_{2N}}$, thus proving \eqref{eq:Tassion1_3}.
\end{proofsect}

\subsection{From squares to rectangles}
We now move to bounds on resistance across rectangular domains. As in Bernoulli percolation, a fundamental tool in this endeavor is the FKG inequality which, in our case, will be used in the following~form:

\begin{lemma}
\label{lem:resistance_chaining1} Consider a finite set $S \subseteq \Z^2$  and a Gaussian process $\{\chi_v\}_{v \in \mathcal R}$ such that $\cov(\chi_u, \chi_v) \geq 0$ for all $u, v \in S$. 
Suppose that $\mathcal P_1, \mathcal P_2,\cdots, \mathcal P_{n}$ are collections of paths in~$S$ that satisfy the conditions of Lemma~\ref{lem:series_law} for a pair of 
disjoint subsets $(A, B)$~of $S$. Then for any $r > 0$, we have
\begin{equation}
\P\bigl(R_{S_\chi}(A, B) \leq nr\bigr) \geq \prod_{i =1}^n\P\bigl(R_{S_\chi}(\mathcal P_i) \leq r\bigr)\,.
\end{equation}
\end{lemma}

\begin{proofsect}{Proof}
This is an immediate consequence of Lemma~\ref{lem:series_law},  the monotonicity of $R_{S_\chi}(\mathcal P_i)$ in individual edge resistances, and the FKG inequality in Lemma~\ref{lem:fkg_gaussian}. 
\end{proofsect}

The principal  outcome of this subsection is: 

\begin{proposition}
\label{lem:Tassion2}
There are $c_0,C_3\in(0,\infty)$ such that for all $N\ge N_0$ for which $\alpha_N \leq 2\alpha_{\lfloor 4N/7 \rfloor}$ holds, all $M\ge8N$ and any shift~$S$ of $B(4N,N)$ satisfying $S\subseteq B(M/2)$,
\begin{equation}
\label{E:5.32}
\P\bigl(R_{\LR;S,{\chi_{M}}} \leq C_3\e^{\cspecial \log\log M}\bigr) \geq c_0\,.
\end{equation}
The same applies to $R_{\UD;S,{\chi_{M}}}$ for any shift~$S$ of $B(N,4N)$ that obeys~$S\subseteq B(M/2)$.
\end{proposition}

By Proposition~\ref{prop:zero_mean} the bound holds for left-to-right resistance of centered squares. We will  employ a geometric argument combined with the FKG inequality to extend the bound from squares to rectangular domains. The main technical tool is Lemma~\ref{lem:series_law}  which,  in a sense, permits us to bound resistance by path-connectivity considerations only. We will actually use a different argument depending on whether $\alpha_N$ equals, or is less than $\lfloor N/2\rfloor$.

\begin{proofsect}{Proof of Proposition~\ref{lem:Tassion2}, case $\alpha_N = \lfloor N/2\rfloor$}
Here we will need the bound \eqref{E:5.25}, but for the underlying domain not necessarily centered at the box which defines the underlying field. Thus, for~$S$ a translate of the square $B(N)$ such that $S\subseteq B(M/2)$, let~$R_{S,[\alpha,\beta],\chi_M}$ denote the quantity corresponding to $R_{N,[\alpha,\beta],\chi_M}$ for the square~$S$ and the underlying field  given by~$\chi_{M}$. In light of \eqref{E:5.25}, Corollary~\ref{cor:field_smoothness1} and Lemma~\ref{lem:resistance_chaining1} show that, for some constant $C_3'\in(0,\infty)$ depending only on~$C_1$ and~$C_2$, 
\begin{equation}
\label{E:5.33}
\P\bigl(R_{S,[\alpha_N,N],\chi_M} \leq C_3'\e^{\cspecial \log\log M}\bigr)\ge0.001
\end{equation}
holds for all~$N\ge N_0$, all $M\ge8N$ and all squares~$S$ as above that are contained in $B(M/2)$. Thanks to invariance of the law of~$\chi_{M}$ under rotations of $B(M)$, the same bound holds also for the ``rotated'' quantities; namely, those dealing with ``up-down' resistances.

Now let~$S$ be a translate by~$x\in\Z^2$ of the rectangle $B(4N,N)$ such that $S\subseteq B(M/2)$ and let us regard~$S$ as the union of the squares
\begin{equation}
S_i:=x+ (i-5) Ne_1+B(N),\qquad i=1,\dots,7.
\end{equation}
For each $i\in\{1,\dots,7\}$, consider the following collections of paths: First, let $\mathcal P_i$ be the set of all paths in~$S_i$ that cross~$S_i$ left to right (with only the initial and terminal point visiting the left and right boundaries of~$S_i$). Then (referring to parts of the boundary as if~$S_i$ were the square~$B(N)$), let~$\mathcal P_i'$ be the collection of paths that connects the bottom of the square to the $[-N,-\alpha_N]$ portion of the top boundary, and let $\mathcal P_i''$ be the path between the bottom of the square to the $[\alpha_N,N]$ portion of the top boundary. The key point (implied by the fact that $\alpha_N = \lfloor N/2\rfloor$) is now that, for any choice of paths $P_i\in\mathcal P_i$, $P_i'\in\mathcal P_i'$ and $P_i''\in\mathcal P_i''$ and any $i=1,\dots,7$, the graph union of the triplet of paths $(P_i,P_i',P_i'')$ is connected and, for each $i=1,\dots,6$, the graph union of $(P_i,P_i',P_i'')$ is connected to the graph union of $(P_{i+1},P_{i+1}',P_{i+1}'')$;  see Fig.~\ref{fig6}.

\begin{figure}[t]
\centerline{\includegraphics[width=0.98\textwidth]{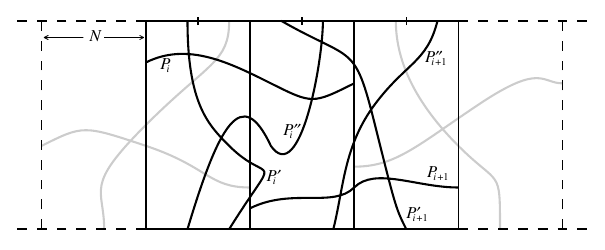}
}
\small 
\vglue0.0cm
\caption{
\label{fig6}
The setting of the proof of Proposition~\ref{lem:Tassion2}, case $\alpha_N = \lfloor N/2\rfloor$. The collection of paths shown suffices to ensure a left-to right crossing through the four shown translates of~$B(N)$. The key points to observe are that $P_i$ intersects both~$P_i'$ and~$P_i''$ while $P_i''$ intersects~$P_{i+1}'$, for each~$i$.}
\normalsize
\end{figure}

It follows that the graph union of the seven triplets of paths contains a left-to-right crossing of the rectangle~$S$ and, by Lemma~\ref{lem:series_law}, we thus get
\begin{equation}
R_{\LR;S,{\chi_{M}}}
\le\sum_{i=1}^7\Bigl(R_{S_i,\chi_{M}}(\mathcal P_i)+R_{S_i,\chi_{M}}(\mathcal P_i')+R_{S_i,\chi_{M}}(\mathcal P_i'')\Bigr).
\end{equation}
In light of the definition \eqref{E:5.16} (and, for simplicity of computation, restricting~$\mathcal P_i$ to paths that terminate only at the top $[\alpha_N,N]$ portion of the right boundary), \eqref{E:5.33} and the FKG inequality now give \eqref{E:5.32} with $C_3:=21C_3'$ and $c_0:=10^{-63}$.
\end{proofsect}

 \begin{proofsect}{Proof of Proposition~\ref{lem:Tassion2}, case $\alpha_N < \lfloor N/2\rfloor$}
Here, in addition to \eqref{E:5.25a} which, as before, we bring to the form \eqref{E:5.33}, we will also need \eqref{E:5.26} --- this is why we need $\alpha_N<\lfloor N/2\rfloor$ --- which we extend using Corollary~\ref{cor:field_smoothness1} and Lemma~\ref{lem:resistance_chaining1} to the form
\begin{equation}
\label{E:5.35}
\P\bigl(R_{S,[0,\alpha_N],\chi_M} \leq C_3''\e^{\cspecial \log\log M}\bigr)\ge0.01
\end{equation}
for some $C_3''\in(0,\infty)$, all~$N\ge N_0$ and all translates~$S$ of $B(N)$ such that $S\subseteq B(M/2)$. The same bound holds also for all rotations and reflections of these quantities.

\begin{figure}[t]
\centerline{\includegraphics[width=0.65\textwidth]{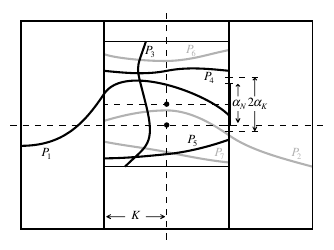}
}
\small 
\vglue0.0cm
\caption{
\label{fig7}
An illustration of the geometric setting underlying the key argument in the proof of Proposition~\ref{lem:Tassion2}, case $\alpha_N < \lfloor N/2\rfloor$. Here $K:=\lfloor 4N/7\rfloor$ and $\alpha_N\le 2\alpha_K$. Examples of paths $P_1\in\mathcal P_1$, $P_3\in\mathcal P_3$, $P_4\in\mathcal P_4$ and $P_5\in\mathcal P_5$ are shown in black. Together with any choice of paths $P_2\in\mathcal P_2$, $P_6\in\mathcal P_6$ and $P_7\in\mathcal P_7$ (shown in gray), these enforce a left-to-right crossing of the rectangle.  We note that, although in the above drawing one of~$P_5$ and~$P_7$ can be dropped without changing the overall effect, all the depicted paths are generally needed to achieve the desired crossing. 
}
\normalsize
\end{figure}

Abbreviate $K:=\lfloor 4N/7\rfloor$ and note that $K<N<2K$ for~$N$ large enough. Let us first deal with~$S$ being a translate of the rectangle~$([-N,3N-2K]\times[-N,N])\cap\Z^2$ by some~$x\in\Z^2$ subject to the restriction $S\subseteq B(4K)$. Consider the squares
\begin{equation}
S_1:=x+B(N),\quad S_2:=x+2(N-K)e_1+B(N)
\end{equation}
and
\begin{equation}
S_3:=x+(N-K)e_1+\alpha_K e_2+[-K,K]^2\cap\Z^2
\end{equation}
and note that $S_1\cup S_2=S$ and $S_3\subseteq S_1\cap S_2$;  see Fig.~\ref{fig6}.  Define the following collections of paths: First, let~$\mathcal P_1$ be all paths in~$S_1$ from the left side to the $[0,\alpha_N]$ portion of the right side. Similarly, let $\mathcal P_2$ be all paths in~$S_2$ from the $[0,\alpha_N]$ portion of the left side to the right side of~$S_2$.  Next  we define the following collections of paths in~$S_3$:
\settowidth{\leftmargini}{(1111)}
\begin{enumerate}
\item[(1)] the set $\mathcal P_3$ of all paths from the top to the bottom sides of~$S_3$,
\item[(2)] the set $\mathcal P_4$ of all paths from the left side of~$S_3$ to the $[\alpha_K,K]$ portion of the right side,
\item[(3)] the set $\mathcal P_5$ of all paths from the left side of~$S_3$ to the $[-K,-\alpha_K]$ portion of the right side,
\item[(4)] the set $\mathcal P_6$ of all paths from the $[\alpha_K,K]$ portion of the left side of~$S_3$ to the right side, and
\item[(5)] the set $\mathcal P_7$ of all paths from the $[-K,-\alpha_K]$ portion of the left side of~$S_3$ to the right side.
\end{enumerate}
The key point is that, thanks to the assumption  $\alpha_N\le 2\alpha_K$,  for any choice of paths $P_i\in\mathcal P_i$, the graph union of these paths will contain a left-to-right path crossing~$S$;  see Fig.~\ref{fig6}. By Lemma~\ref{lem:series_law}, 
\begin{equation}
R_{\LR,S,\chi_{M}}\le\sum_{i=1}^7 R_{S_i,\chi_{M}}(\mathcal P_i),
\end{equation}
where $S_4=\dots=S_7:=S_3$. From here we get \eqref{E:5.32} for all $2(2N-K)\times2N$ rectangles~$S\subseteq B( M/2)$ with $C_3:=21\max\{C_3',C_3''\}$ and~$c_0:=10^{-14}$.

In order to prove the desired claim, consider a translate~$S$ of $B(4N,N)$ by~$x\in\Z^2$ entirely contained in~$B( M/2)$ and note that, letting $k:= \lceil\frac {4N}{N-K}\rceil$, and we can cover~$S$ by the family of rectangles $S_0',\dots,S_k'$ and~$S_1'',\dots,S_{k-1}''$ defined as follows:
\begin{equation}
S_j':= x_j+\bigl([0,2(2N-K)]\times[-N,N]\bigr)\cap\Z^2,\quad j=0,\dots,k,
\end{equation}

where $x_j:=x+2(N-K)je_1$ for all $j=0,\dots,k-1$ and~$x_k:=x+[8N-2k(N-K)]e_1$, which ensures that all~$S_i'$ lie inside~$S$ (and thus inside~$B(M/2)$), 
and
\begin{equation}
S_j'':= y_j+\bigl([-N,N]\times[0,2(2N-K)]\bigr)\cap\Z^2,\quad j=1,\dots,k-1,
\end{equation}

where $y_j-x_j$ are such that all~$S_j''$ lie in~$B( M/2)$ (this is possible because $2(2N-K)<16N$) and such that $S_j'\cap S_j''\subseteq S_{j+1}'$ for each~$j=1,\dots,k-1$. Assuming each $S_j'$ and~$S_j''$ contains a path connecting the shorter sides of the rectangle, the graph union of these paths then contains a left-to-right crossing of~$S$. Lemma~\ref{lem:series_law} then  gives
\begin{equation}
R_{\LR,S,\chi_{M}}
\le\sum_{j=0}^k R_{\LR,S_j',\chi_{M}}+\sum_{j=1}^{k-1}R_{\UD,S_j'',\chi_{M}}\,.
\end{equation}
In light of our earlier proof of \eqref{E:5.32} for rectangles of dimensions $2N\times2(2N-K)$, we get \eqref{E:5.32} for $2N\times8N$ rectangles as well with $C_3:=21(2k+1)\max\{C_3',C_3''\}$ and $c_0=10^{-14(2k+1)}$.
\end{proofsect}

\subsection{Bounding the growth of~$\alpha_N$}
It appears that Proposition~\ref{lem:Tassion2} could  be  more than sufficient for proving uniform upper bound on resistance across rectangles, provided we can somehow guarantee that $N\mapsto\alpha_N$ does not grow faster than exponentially with~$N$. This is the content of:

\begin{proposition}
\label{lem:Tassion_main}
For each~$c_0\in(0,1)$ and each~$C_3\in(0,\infty)$, there exists an integer~$C_5>8$ such that if, for some $N\ge1$,
\begin{equation}
\label{E:5.32a}
\P\bigl(R_{\LR;S,{\chi_{16N}}} \leq C_3\e^{\cspecial \log\log (16N)}\bigr) \geq c_0
\end{equation}
holds for all translates or rotates~$S$ of $B(4N,N)$ contained in~$B(8N)$, then we have $\alpha_{N'}\ge N$ for at least one~$N'\in\{8N,\dots,C_5N\}$.
\end{proposition}

The proof will be based on the following lemma:
  
\begin{lemma}
\label{lem:Tassion3}
Suppose that, for some~$c_0,C_3\in(0,\infty)$ and some~$N\ge1$, \eqref{E:5.32a} holds for all translates and rotates of $B(4N,N)$ contained in~$B(8N)$. There are $c_1$ and $C_4$, depending only on~$c_0$ and~$C_3$, respectively, such that whenever $K > 2N$ is such that $\alpha_K \leq N$ and $M\ge16K$,
\begin{equation}
\label{E:4.48}
\P\bigl(R_{\LR,S,\chi_{M}} \leq C_4\e^{\cspecial \log\log M}\bigr) \geq c_1
\end{equation}
holds for all translates and rotates of $B(4K,K)$ contained in~$B(8K)$.
\end{lemma}

\begin{figure}[t]
\centerline{\includegraphics[width=0.85\textwidth]{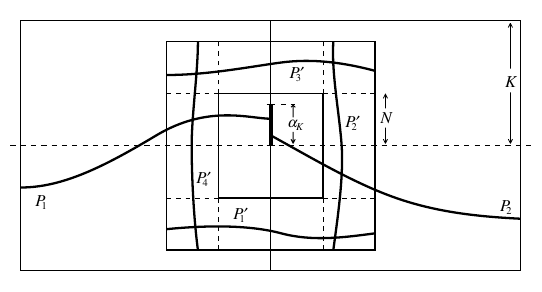}
}
\small 
\vglue0.0cm
\caption{
\label{fig8}
The geometric setup for the proof of Lemma~\ref{lem:Tassion3}. The graph union of paths $P_1,P_2,P_1',\dots,P_4'$ contains a left-to-right crossing of the $4K\times2K$-rectangle.}
\normalsize
\end{figure}

\begin{proofsect}{Proof}
We will first prove this for rectangles~$S$ of the form $B(2K,K)$. Consider the squares $S_1:=-Ke_1+[-K,K]^2\cap\Z^2$ and $S_2:=Ke_1+B(K)$ and let~$S_1',\dots,S_4'$ be the four maximal rectangles of dimensions $N\times4N$, labeled counterclockwise starting from the one at the bottom, contained in the annulus $B(2N)\smallsetminus B(N)^{ \circ}$. Let~$P_1$ be a path in~$S_1$ connecting the left-hand side to the $[0,\alpha_K]$ portion of the right-hand side and, similarly, $P_2$ is the path in~$S_2$ connecting the $[0,\alpha_K]$-portion of the left-hand side to the right hand side. Let $P_1',\dots,P_4'$ be paths (in $S_1',\dots,S_4'$, respectively) between the shorter sides of $S_1',\dots,S_4'$, respectively. Then the assumption $\alpha_K\le N$ implies that the graph union of $P_1,P_2,P_1',\dots,P_4'$ contains a path in~$S$ connecting the left side to the right side;  see Fig.~\ref{fig8}.  Combining \eqref{E:4.48} with \eqref{E:5.35} (in which~$N$ is replaced by~$K$), we get
the claim for~$S$ with $C_4:=2C_3''+4C_3$ and $c_1:=10^{-4}(c_0)^4$.

To extend this to rectangles~$S$ of the form $B(4K,K)$, we note that these can be covered by four translates and two rotates of $B(2K,K)$ such that the existence of a crossing between the shorter sides in each of these rectangles forces a crossing of~$S$. Thanks to Lemma~\ref{lem:resistance_chaining1}, the desired bound then holds for~$S$ as well; we just need to multiply the above~$C_4$ by~$6$ and raise the above~$c_1$ to the sixth power.
\end{proofsect}

We are now ready to give:

\begin{proofsect}{Proof of Proposition~\ref{lem:Tassion_main}}
The proof is by way of contradiction; indeed, we will prove that if such~$N'$ does not exist, then we will ultimately violate the first inequality in (P2) in Lemma~\ref{lem:Tassion1} for a sufficiently large square. This will be done by showing that a path from the left side of the square~$B(N')$ to the $[0,\alpha_{N'}]$ part of the right side can be re-routed to instead terminate in the~$[\alpha_{N'},N']$-part of the right side. The re-routing will be achieved by showing existence of a path winding around an annulus of inner ``radius'' at least~$\alpha_{N'}$ centered at the point $\o_{N'}:=(N',0)$.

We will focus on~$N'$ of the form $N':=b^n N$, where~$b:=8$ and $n\ge1$. Fix such an~$n$ (and thus~$N'$) and, for $k=1,\dots,n$, let $B_{ n,k}:=\o_{N'}+B(b^kN)$. Consider also the annulus $A_{ n,k}:=\o_{N'}+B(4b^kN)\smallsetminus B(2b^kN)^\circ$ and define the conditional field
\begin{equation}
\chi_{4N',k;v}:=\chi_{4N', v} - \E\biggl(\chi_{4N', v} \,\bigg|
\, \sigma\Bigl(\chi_{4N',u}\colon u\in\bigcup_{n-k\le j\le n}\partial B_{ n,j}\Bigr)\biggr). 
\end{equation}
By the Gibbs-Markov property of the GFF, $\{\chi_{4N',k;v}\colon v\in A_{ n,k}\}$ has the law of the values on~$A_{ n,k}$ of the GFF in~$B(b^{k+1}N)\smallsetminus B(b^k N)^\circ$  with Dirichlet boundary condition. Let $R_{A_{ n,k}; \chi_{4N', k}}$ denote the sum of the resistances between the shorter sides of the four maximal rectangles contained in $A_{ n,k}$, in the field $\chi_{4N',k}$.

Assuming $\alpha_{N'}\le N$, Lemma~\ref{lem:Tassion3} in conjunction with Corollary~\ref{cor:field_smoothness1} and Lemma~\ref{lem:resistance_chaining1} show that, for some $C'_{4}\in(0,\infty)$ and $c_2>0$:
\begin{equation}
\label{eq:Tassion_main1}
\P\bigl(R_{A_{ n,k}; \chi_{4N', k}} \leq C'_{4}\e^{\cspecial \log\log N'}\bigr) \geq c_2\,.
\end{equation}
Let~$m$ be the smallest integer such that $(1 - c_2)^{m} \leq 0.01$, let~$C_1$ be as in the first inequality in (P2) in Lemma~\ref{lem:Tassion1} and let~$\wt C$ be the constant from Lemma~\ref{lem:ellipicity_upper}. Define
\begin{equation}
\wt M_{ n,k} :=\min_{v\in A_{ n,k}}\E\biggl(\chi_{4N', v} \,\bigg|
\, \sigma\Bigl(\chi_{N',u}\colon u\in\bigcup_{n-k\le j\le n}\partial B_{j,n}\Bigr)\biggr).
\end{equation}
Lemma~\ref{lem:ellipicity_upper}  (dealing with the LIL for the sequence $M_{n,k}$) and Lemma~\ref{lem:anchoring}  (dealing with the deviations~$\Delta_n$) tell us that there is a positive integer $m' > 100$ satisfying
\begin{multline}
\label{eq:Tassion_main2}
\qquad
\P\biggl(\#\Bigl\{k =1,\dots,m'-1\colon \gamma \wt M_{k, m'} \geq 0.5\log \frac{C'_{4}}{C_1} + \log 5 + \wt C\gamma\sqrt{\log m'}\Bigr\} < m\biggr) \\ 
\le 0.01 + 0.01 = 0.02 \,.
\qquad
\end{multline}
Putting together \eqref{eq:Tassion_main1}, \eqref{eq:Tassion_main2}, the choices of $m$ and $m'$ along with Lemmas~\ref{lem:anchoring} and \ref{lem:resistance_decomp} we get  for all $N$ such that $c\log \big(1 + \tfrac{(m' + 1)\log 8}{\log N} \big) \leq \log 5$,
\begin{equation}
\P\big(\exists k \in \{1,\dots,m'\}\colon R_{A_{k, m'}; \chi_{C_{5}N}} \geq C_1\e^{\cspecial \log\log N}\big) \leq 0.02 + 0.01 = 0.03\,,
\end{equation}
where $C_{5} := 8^{m' + 1}$.

\begin{figure}[t]
\centerline{\includegraphics[width=0.65\textwidth]{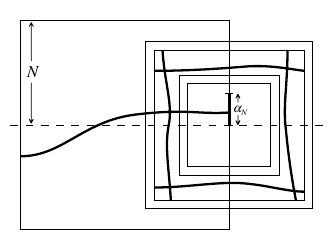}
}
\small 
\vglue0.0cm
\caption{
\label{fig9}
The geometric setting for a key argument in the proof of Proposition~\ref{lem:Tassion_main}: Once~$\alpha_N$ is less than the inner radius of the depicted annulus, $R_{B(N)_{\chi}}(\mathcal P_{N; [\alpha_{N}, N]})$ is bounded by $R_{B(N)_{\chi}}(\mathcal P_{N; [0,\alpha_{N}]})$ plus the sum of the resistances between the shorter sides of the four maximal rectangles contained in the annulus.
}
\normalsize
\end{figure}

We are now ready to derive the desired contradiction.
Lemma~\ref{lem:series_law} 
gives us that if $\alpha_{N'} \leq N$ for all $8N \leq N' \leq C_{5}N$, then
\begin{equation}
\P\Bigl(R_{B(N')_{\chi_{N'}}}(\mathcal P_{N'; [\alpha_{N'}, N']}) \leq R_{B(N')_{\chi_{N'}}}(\mathcal P_{N'; [0, \alpha_{N'}]}) + C_1\e^{\cspecial \log\log N}\Bigr) \\
\geq 1 - 0.03 = 0.97\,.
\end{equation}
From the second inequality in (P2) in Lemma~\ref{lem:Tassion1} we have
\begin{equation}
\P\bigl(R_{B(N')_{\chi_{N'}}}(\mathcal P_{N'; [0, \alpha_{N'}]}) \leq  4\e^{\cspecial \log\log N'}\bigr) \geq 0.17\,.
\end{equation}
The last two displays and the FKG imply
\begin{equation}
\P\Bigl(R_{B(N')_{\chi_{N'}}}(\mathcal P_{N'; [\alpha_{N'}, N']}) \leq (4 + C_1)\,\e^{\cspecial \log\log N'}\Bigr) > 0.17 \times 0.97 > 0.16\,.
\end{equation}
in contradiction with the first inequality in (P2) in 
Lemma~\ref{lem:Tassion1}. The claim follows. 
\end{proofsect}

\subsection{Resistance across rectangles and annuli}
\label{sec-wrap-up}
As a consequence of the above arguments, we are now ready to state our first \emph{unrestricted} general upper bound on the effective resistance across rectangles:

\begin{proposition}
\label{prop-4.7}
There are constants $C_{6}, c_3\in(0,\infty)$ and $N_1\ge1$ such that for all $N \geq N_1$, all $M\ge16N$ and for every translate~$S$ of $B(4N,N)$ contained in~$B(M/2)$, we have
\begin{equation}
\label{E:5.55b}
\P\bigl(R_{\LR;S,{\chi_{M}}} \leq C_6\e^{\cspecial \log\log(M)}\bigr) \geq c_3\,.
\end{equation}
The same applies to $R_{\UD;S,{\chi_{M}}}$ for translates~$S$ of $B(N,4N)$ with~$S\subseteq B(M/2)$.
\end{proposition}

We begin by showing that \eqref{E:5.32a} holds (with the same constants) along an exponentially growing sequence of~$N$. This is where Proposition~\ref{lem:Tassion2} and 
Proposition~\ref{lem:Tassion_main} come together.  Later we use Lemma~\ref{lem:series_law} to extend this bound to all the intermediate scales by  combining  the effective resistances along a sequence of appropriately placed 
rectangles.

\begin{lemma}
\label{lemma-4.7}
Let $c_0$ and~$C_3$ be as in Proposition~\ref{lem:Tassion2}. There is $c\in(0,\infty)$ and an increasing sequence $\{N_k\colon k\ge1\}$ of positive integers such that, for each $k\ge1$, we have
\begin{equation}
\label{E:5.54}
14N_{k}-1 \leq N_{k+1} \leq c N_{k}
\end{equation}
and the bound
\begin{equation}
\label{E:5.55}
\P\bigl(R_{\LR;S,{\chi_{16N_k}}} \leq C_3\e^{\cspecial \log\log(16N_k)}\bigr) \geq c_0
\end{equation}
holds for all translates~$S$ of $B(4N_k,N_k)$ contained in~$B(8N_k)$.
\end{lemma}

\begin{proofsect}{Proof}
We will  construct $\{N_k\colon k\ge1\}$  by induction. Suppose that $N_1,\dots,N_k$ have already been defined. Since \eqref{E:5.55} holds for~$N_k$, Proposition~\ref{lem:Tassion_main} shows the existence of an $L \in [8N_k, C_{5}N_{k}]$ with $\alpha_L\ge N_k$. Define a sequence $\{L_j\colon j\ge0\}$ by $L_0:=L$ and $L_{j+1}:=\min\{L\in\N\colon \lfloor 4L/7\rfloor = L_j\}$ and note that $L_j\le c(7/4)^j L$ for some numerical constant~$c'\in(0,\infty)$. Now if $\alpha_{L_{i+1}}>2\alpha_{L_i}$ is true for $i=0,\dots,j-1$, then
\begin{equation}
\label{E:5.56}
2^j N_k\le 2^j\alpha_L<\alpha_{L_j}\le L_j\le c'(7/4)^j L\le c'(7/4)^j C_5 N_k.
\end{equation}
The fact that $7/4<2$ implies that this must fail  once~$j$ is sufficiently large; i.e.,  for some~$j\in\{0,\dots,C_5'\}$, where $C_5'$ depends only on~$C_5$. We thus let~$j\ge1$ be the smallest such that $\alpha_{L_j}\le2\alpha_{L_{j-1}}$ and set $N_{k+1}:=L_j$. Then \eqref{E:5.54} holds by the inequality on the right of \eqref{E:5.56} and the fact that $N_{k+1}\ge L_1\ge(7/4)L-1\ge14N_k-1$. The bound \eqref{E:5.55} is implied by Proposition~\ref{lem:Tassion2}.

To start the induction, we just take the above sequence $\{L_j\}$ with $L:=1$ and find the first index~$j$ for which $\alpha_{L_j}\le2\alpha_{L_{j-1}}$. Then we set~$N_1:=L_j$ and argue as above.
\end{proofsect}

From here we now conclude:

\begin{proofsect}{Proof of Proposition~\ref{prop-4.7}}
Let~$\{N_k\}$ be the sequence from Lemma~\ref{lemma-4.7}. 
Invoking Corollary~\ref{cor:field_smoothness1}, 
the bound \eqref{E:5.55} shows that, for each $M\ge16N_k$ and any translate~$S$ of $B(4N_k,N_k)$ contained in~$B(M/2)$,
\begin{equation}
\label{E:5.55a}
\P\bigl(R_{\LR;S,{\chi_{M}}} \leq C_3'\e^{\cspecial \log\log(M)}\bigr) \geq c_0'.
\end{equation}
holds with some constants~$C_3',c_0'\in(0,\infty)$ independent of~$k$ and~$M$.
By invariance of the law of~$\chi_M$ with respect to rotations of~$B(M)$, the same holds for the resistance~$R_{\UD;S,{\chi_{M}}}$ for all rotations of  $B(4N_k,N_k)$ contained in~$B(M/2)$. 

Now pick~$N\ge N_1$ and let~$k$ be such that $N_k\le N<N_{k+1}$. For $M\ge 16 N\ge 16N_k$, consider a translate~$S$ of $B(4N,N)$ contained in $B(M/2)$. Let~$m:=\min\{r\in\N\colon (3 r+1)N\ge N_{k+1}\}$; by \eqref{E:5.54} this $m$ is bounded uniformly in~$k$. We then find rectangles $S_i$, $i=1,\dots,m$ that are translates of $B(4N_k,N_k)$ such that $S_{i+1}=3Ne_1+S_i$ for each $i=1,\dots,m-1$ and are centered along the same horizontal line as~$S$ and positioned  in such a way that they all lie inside~$B(M/2)$. Next we find translates $S_1',\dots,S_{m-1}'$ of $B(N_k,4N_k)$ such that $S_i\cap S_{i+1}$, which is a translate of~$B(N)$, is contained in~$S_i'$ for each~$i=1,\dots,m-1$. We can again position these so that $S_i'\subseteq B(M/2)$ for each~$i$.

It is clear from the construction that if, for each $i=1,\dots, m$, we are given a path in~$S_i$ and, for each $i=1,\dots,m-1$, a path in~$S_i'$ and these paths connect the shorter sides of the rectangle they lie in, then the graph union of all these paths contains a path in~$S$ between the left side and right side thereof. Lemma~\ref{lem:series_law} then gives
\begin{equation}
\label{E:4.62}
R_{\LR;S,{\chi_{M}}}\le\sum_{i=1}^m R_{\LR;S_i,{\chi_{M}}}+\sum_{i=1}^{m-1}R_{\UD;S'_i,{\chi_{M}}}.
\end{equation}
All the rectangles lie in~$B(M/2)$ and so \eqref{E:5.55a} applies to the resistances on the right of \eqref{E:4.62}. Lemma~\ref{lem:resistance_chaining1} then readily gives \eqref{E:5.55} with the constants given by $C_6:=(2m-1)C_3'$ and $c_3:=(c_0')^{2m-1}$.
\end{proofsect}

 In addition to resistance across rectangles, the proofs in Section~\ref{sec-5} will also require an lower bound for resistances across annuli. For $N<M$, let $A(N,M):=B(M)\smallsetminus B(N)^\circ$ and denote
\begin{equation}
\bdryin A(N,M):=\partial B(N)
\quad\text{\rm and}\quad
\bdryout A(N,M):=\partial B(M)^\circ
\end{equation}
Note that $\bdryin A(N,M)\subset A(N,M)$ as well as $\bdryout A(N,M)\subset A(N,M)$. We have:

\begin{lemma}
\label{lemma-annuli}
There $C_7,c_4\in(0,\infty)$ such that for all large-enough~$N$ and $A:=A(N,2N)$,
\begin{equation}
\P\Bigl(R_{A_{\chi_{4N}}}(\bdryin A,\bdryout A)\ge C_7\e^{-3\cspecial \log\log(4N)}\Bigr)\ge c_4.
\end{equation}
\end{lemma}

\begin{proofsect}{Proof}
Let~$S_1,S_2,S_3,S_4$ denote the four maximal rectangles contained in~$A$. We assume that the rectangles are labeled clockwise starting from the one on the right.  Now observe that  every path in~$A$ from $\bdryin A$ to~$\bdryout A$ contains a path that is contained in, and connects the longer sides of, one of  the rectangles $S_1,S_2,S_3,S_4$. It follows that 
\begin{equation}
R_{A_{\chi_{4N}}}(\bdryin A,\bdryout A)
\ge R_{\LR,S_1,\chi_{4N}}+R_{\UD,S_2,\chi_{4N}}+R_{\LR,S_3,\chi_{4N}}+R_{\UD,S_4,\chi_{4N}}\,.
\end{equation}
The claim will follow from the FKG inequality if we can show that, for some~$p>0$ and~$C_7'>0$,
\begin{equation}
\label{E:5.77a}
\P\bigl(R_{\LR,S,\chi_{4N}}\ge C_7'\e^{-3\cspecial \log\log(4N)}\bigr)\ge p
\end{equation}
holds for all translates~$S$ of $([0,N]\times[0,4N])\cap\Z^2$ contained in~$B(2N)$ and all~$N$ sufficiently large. (Indeed, then $c_4:=p^4$ and $C_7:=4C_7'$.)

We will show this using the duality in Lemma~\ref{lem:dual_couple}, but for that we will first need to invoke the decomposition $\chi_{4N}=Y_{4N}+Z_{4N}$ from Lemma~\ref{lem:field_decomp_smooth}. First, for any $r,A>0$,
\begin{equation}
\P\bigl(R_{\LR,S,\chi_{4N}}\ge r\bigr)\ge \P\bigl(R_{\LR,S,Z_{4N}}\ge r/A\bigr)-\P\bigl(R_{\LR,S,\chi_{4N}}< A R_{\LR,S,Z_{4N}}\bigr)
\end{equation}
Passing over to conductances, from Lemma~\ref{lem:resistance_decomp} we then get, as before,
\begin{equation}
\P\bigl(R_{\LR,S,\chi_{4N}}< A R_{\LR,S,Z_{4N}}\bigr)\le \frac1A\e^{\cspecial \log\log(4N)},
\end{equation}
while the duality in Lemma~\ref{lem:dual_couple} gives, as in the proof of Proposition~\ref{prop:zero_mean}, 
\begin{equation}
\P\Bigl(R_{\LR,S,Z_{4N}}\,R^\star_{\UD,S,Z_{4N}}\ge \e^{-2\gamma c_1}/64\Bigr)\ge 1-\epsilon.
\end{equation}
Finally, we use Lemma~\ref{lem:resistance_decomp} one more time to get
\begin{equation}
\P\bigl(R^\star_{\UD,S,Z_{4N}}\le\tilde r\bigr)\ge \P\bigl(R_{\UD,S,\chi_{4N}}\le\tilde r/A\bigr)-\frac1{A}\e^{\cspecial \log\log(4N)}.
\end{equation}
If we set $\tilde r/A:=C_6\e^{\cspecial \log\log(4N)}$, Proposition~\ref{prop-4.7} bounds the first probability below by~$c_3$. Now take $A:=C\e^{3\cspecial \log\log(4N)}$ for~$C$ large and work your way back to get \eqref{E:5.77a}.
\end{proofsect}

\subsection{Gaussian concentration and upper bound on point-to-point resistances}
To get the tail estimate on the effective resistance in Theorem~\ref{thm-effective-resistance}, we need to invoke  a concentration-of-measure argument for the quantity at hand.  Recall the notation $R_{A_\chi}(\mathcal P)$ for the effective resistance in network~$A_{\chi}$ restricted to the collection of paths in~$\mathcal P$.

\begin{proposition}
\label{lem:Gaussian_conc}
Suppose $\chi$ is a Gaussian field on $B(N)$ with $\var(\chi_x)\le c_1\log N$   for all $u\in B(N)$ and~$c_1$  independent of~$N$. Let $A_\chi$ be a subnetwork of $B(N)_\chi$ and let $\mathcal P$ be a  finite 
collection of paths within $A$ between some given source 
and destination. 
There is a constant~$c_2\in(0,\infty)$ such that for all~$N\ge1$, all~$t\ge0$  and all~$\gamma>0$, 
\begin{equation}
\mathbb{P}\Bigl(\bigl|\log{ R_{A_{\chi}} (\mathcal P)}- \E \log{ R_{A_{\chi}} (\mathcal P)}\bigr|  \geq t\sqrt{\log N}\Bigr) \leq 
2 \e^{- c_2\gamma^{-2}t^2 }\,.
\end{equation}
\end{proposition}

For the proof, we will need:

\begin{lemma}
\label{lem:Lipschitz}
Let $A$ be a subnetwork of $B(N)$ and $\mathcal P$ be a  finite  collection of paths within~$A$ between 
some given source and destination. Let $g\colon \R^{V(A)} \rightarrow \R$ be defined by
\begin{equation}
g(\mathbf{x}) := \max_{\mathbf q \in \mathcal Q}\log \Big(\sum_{P \in \mathcal P} \frac{1}
{\sum\limits_{e\in P} \mathrm{e}^{-\gamma(x_{e_-} + x_{e_+})}q_{e, P}}\Big)\,,
\end{equation}
where $\mathcal Q$ is the set of all $\mathbf q = (q_{e, P})_{e \in E(A), P\in 
\mathcal P} \in \R_+^{E(A)\times \mathcal P}$ such that 
\begin{equation}
\sum_{P\in \mathcal P} \frac{1}{q_{e, P}} \leq 1 \,\text{\rm\ for all }\, e \in  E(A)\,.
\end{equation}
Then $g$ is Lipschitz in the $\ell^{\infty}$ norm on $\R^{V(A)}$ with Lipschitz constant 
$2\gamma$.
\end{lemma}

\begin{proofsect}{Proof}
Define a new real-valued function, also denoted by $g$, on $\R^{V(A)} \times \R_+^{E(A)\times 
\mathcal P}$ via
\begin{equation}
g(\mathbf{x} ,\mathbf{q}) := \log \Big(\sum_{P \in \mathcal P} \frac{1}
{\sum\limits_{e\in P} \mathrm{e}^{-\gamma(x_{e_-} + x_{e_+})}q_{e, P}}\Big)\,.
\end{equation}
Then for any $\mathbf q \in \mathcal Q$ and $\mathbf x, \mathbf y \in \R^{V(A)}$ it is clear that
\begin{equation}
|g(\mathbf{x} ,\mathbf{q}) - g(\mathbf{y} ,\mathbf{q})| \leq 2\gamma||\mathbf x - \mathbf y||
_{\infty}\,.
\end{equation}
Hence $g(\mathbf{x}) = \max_{\mathbf q \in \mathcal Q}g(\mathbf{x}, \mathbf q)$ is $2\gamma$-
Lipschitz relative to the $\ell^{\infty}$ norm as well.
\end{proofsect}

\begin{proofsect}{Proof of Proposition~\ref{lem:Gaussian_conc}}
 The ``standard'' Gaussian concentration inequality (see \cite{ST74,Borell75} and also Lemma~\ref{lem:Borell_ineq}) states that for any centered Gaussian process $\mathbf X := \{X_v\colon v \in A\}$ indexed by points in a finite set~$A$ and for any $g\colon\R^A\to\R$ obeying $|g(\mathbf x)-g(\mathbf y)|\le c\Vert\mathbf x-\mathbf y\Vert_\infty$ for all~$\mathbf x,\mathbf y\in\R^A$,
\begin{equation}
\P\Bigl(\,\bigl|   g(\mathbf X)  - \E  g(\mathbf X)\bigr|  \geq \lambda\Bigr) \leq 2 \e^{-\frac{\lambda^2}{ 2 c^2\sigma^2}}\,,\quad\lambda>0,
\end{equation}
where $\sigma^2 := \max_{v \in A}\E (X_v^2)$. The claim then follows via Lemma~\ref{lem:Lipschitz}.
\end{proofsect}

We are now ready to give  a version of the upper bound in Theorem~\ref{thm-effective-resistance}, albeit for a network arising from a GFF on a finite subset of~$\Z^2$: 

 \begin{lemma}
\label{prop:thm1.4}
There is $c_1 \in (0, \infty)$ depending only on $\gamma$ and a constant $c'' \in (0, \infty)$ such that 
\begin{equation}
\P\Bigl(R_{B(N)_{\chi_M}}(u,v)\ge c_1(\log M)\,\e^{t\sqrt{\log M}}\Bigr)\le 2c_1(\log M)\,\e^{-c'' t^2}
\end{equation}
 holds  for all $N \geq 1$,  all  $M \geq 32 N$ and  all  $t \geq 0$.
\end{lemma}

\begin{proofsect}{Proof}
 We first derive a similar bound for effective resistances across rectangles by Proposition~\ref{lem:Gaussian_conc} and the estimates on its quantiles given by Proposition~\ref{prop-4.7} and Corollary~\ref{cor-4.3}. Then we use a gluing argument like the one used in the proof of Proposition~\ref{prop-4.7} to obtain the required upper bound for point-to-point effective resistances. Let us start with the following observation.  Combining Proposition~\ref{prop-4.7} with Corollary~\ref{cor-4.3}, for each $\epsilon>0$ there is $N_0''=N_0''(\gamma,\epsilon)$ such that if $N\ge N_0''$, $M\ge 32 N$ and~$S$ is a translate of $B(4N,N)$ contained in~$B(M/2)$, then we have
\begin{equation}
\P\Bigl(\bigl|\log R_{\LR;S,\chi_{2M}}\bigr|\le 2\cspecial \log\log(2M)+\log C_6\Bigr)\ge \epsilon.
\end{equation}
Decomposing $\chi_{2M}$ on $B(M)$ into a fine field~$\chi_{M}^f$ and a coarse field $\chi_M^c$, the fact that 
\begin{equation}
|\log R_{\LR;S,\chi_{2M}}\bigr|\ge |\log R_{\LR;S,\chi_{M}^f}\bigr|-2\gamma\max_{u\in S}\bigl|\chi_{M}^c\bigr|
\end{equation}
along with $\chi_M^f\laweq\chi_M$ shows
\begin{equation}
\P\Bigl(\bigl|\log R_{\LR;S,\chi_{M}}\bigr|\le2\cspecial \log\log(2M) +\log C_6+2\tilde c\gamma\Bigr)
\ge \epsilon-\P\Bigl(\max_{u\in S}\bigl|\chi_{M}^c\bigr|>\tilde c\Bigr).
\end{equation}
The last probability tends to zero as~$\tilde c\to\infty$ uniformly in~$M\ge1$ and so, by choosing~$\tilde c$ large, there is a constant $C_7\in(0,\infty)$ such that, for all $N\ge N_0''$,
\begin{equation}
\label{E:4.72}
\P\Bigl(\bigl|\log R_{\LR;S,\chi_{M}}\bigr|\le2\cspecial \log\log(2M)+\log C_7\Bigr)\ge \epsilon/2
\end{equation}
holds for all $M\ge32N$ and all translates of $B(4N,N)$ contained anywhere in~$B(M)$.

Since \eqref{E:4.72} gives us an interval of width of order $\log\log M$ where $\bigl|\log R_{\LR;S,\chi_{M}}\bigr|$ keeps a uniformly positive mass, the Gaussian concentration in Proposition~\ref{lem:Gaussian_conc} shows that, for some constants $c',c''\in(0,\infty)$, 
\begin{equation}
\E\bigl|\log R_{\LR;S,\chi_M}\bigr|\le c'\sqrt{\log M}
\end{equation}
and also
\begin{equation}
\label{E:4.74}
\P\Bigl(\bigl|\log R_{\LR;S,\chi_M}\bigr|>t\sqrt{\log M}\Bigr)\le 2\e^{-c'' t^2}
\end{equation}
hold for every~$t\ge0$.
The proof has so far assumed $N\ge N_0''$; to eliminate this assumption we note that $\var(\chi_{M,v})\le \tilde c\log M$ uniformly in~$v\in B(M)$ and so the union bound gives
\begin{equation}
\P\bigl(\max_{v\in S}\bigl|\chi_{M,v}|>t\sqrt{\log M}\Bigr)\le 2|S|\e^{-\frac12\tilde c^{-2}t^2}.
\end{equation}
Since $|S|\le(4N_0''+1)^2$ while $|\log R_{\LR;S,\chi_M}|$ is at most $2\gamma \max_{v\in S}\bigl|\chi_{M,v}|$ times an $N_0''$-dependent constant, by adjusting~$c''$ we make \eqref{E:4.74} hold for all $N\ge1$. Due to rotation symmetry, the same bound holds also for $R_{\UD;S,\chi_M}$ and any translate~$S$ of~$B(N,4N)$ contained in~$B(M)$.

\begin{figure}[t]
\centerline{\includegraphics[width=0.65\textwidth]{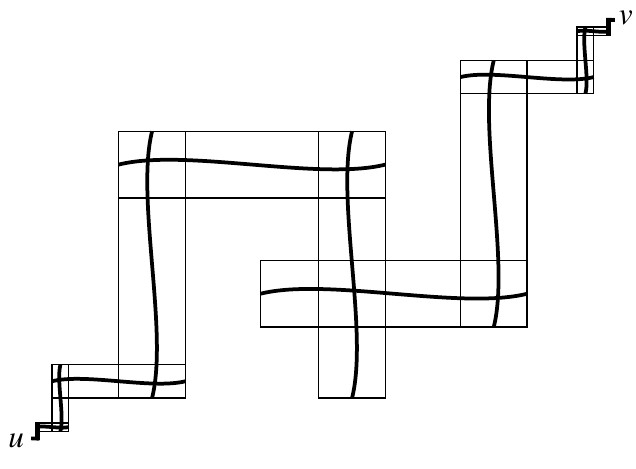}
}
\small 
\vglue0.0cm
\caption{
\label{fig9u}
 A collection of rectangles $4N\times2N$ or $2N\times4N$ rectangles with paths that ensure a connection between two given points at distance order~$M\ge32N$.
}
\normalsize
\end{figure}

Now fix~$M\ge32$ and let~$u,v\in B(M)$. Then one can find  (see Fig.~\ref{fig9u})  a collection of rectangles of the form $B(N,4N)$ or $B(4N,N)$ with $32N\le M$ that are contained in $B(M)$ and satisfy:
\settowidth{\leftmargini}{(11)}
\begin{enumerate}
\item[(1)] There are at most $c_1\log M$ of such rectangles with $c_1\in(0,\infty)$ independent of~$M$.
\item[(2)] If a path is chosen connecting the shorter sides in each of these rectangles, then the graph union of these paths contains a path from $u$ to~$v$.
\end{enumerate}
By Lemma~\ref{lem:series_law}, this construction dominates $R_{B(N)_{\chi_M}}(u,v)$ by the sum of the resistances between the shorter sides of these rectangles. The FKG inequality, \eqref{E:4.74} and a union bound then imply 
\begin{equation}
\P\Bigl(R_{B(N)_{\chi_M}}(u,v)\ge c_1(\log M)\,\e^{t\sqrt{\log M}}\Bigr)\le 2c_1(\log M)\,\e^{-c'' t^2}
\end{equation}
This is the desired claim.
\end{proofsect}

\subsection{ Adapting the bounds to the full-plane field }
In order to extend Lemma~\ref{prop:thm1.4} to the network with the underlying field~$\eta$, we first note: 

\begin{lemma}
Let~$\eta$ denote the GFF on~$\Z^2$ pinned at the origin. There are $C_1, c_1\in(0,\infty)$ and $N_1\ge1$ such that for all $N \geq N_1$, all $M\ge16N$ and for every translate~$S$ of $B(4N,N)$ contained in~  
$B(M/2)$,   we have
\begin{equation}
\label{E:5.1}
\P\bigl(R_{\LR;S,\eta} \leq C_1\e^{2\cspecial \log\log(M)}\bigr) \geq c_1\,.
\end{equation}
The same applies to $R_{\UD;S,{\chi_{M}}}$ for translates~$S$ of $B(N,4N)$ with~$S\subset B(M/2)$.
\end{lemma}

\begin{proofsect}{Proof}
We will assume that~$M$ is the minimal  integer  such that~$S\subset B(M/2)$. Note that this means that $M/N$ is bounded.
We proceed in two steps, first reducing~$\eta$ to the GFF in~$\Lambda:= B(M) 
\smallsetminus\{0\}$ and then relating this field to~$\chi_M$. Using the Gibbs-Markov property, the field~$\eta$ can be written as $\chi_\Lambda+\chi^c$, where $\chi_\Lambda$, the fine field, has the law of the GFF on~$\Lambda$ while the coarse field~$\chi^c$ is~$\eta$ conditional on its values outside of~ $B(M)$.
Now pick an~$x\in B(M)\smallsetminus B(M/2)^\circ$ such that~$x$ is at least~$M/6$ lattice steps from both~$B(M/2)$ and~$B(M)^\cc$. For any $r,A>0$ we then have
\begin{equation}
\begin{aligned}
\label{E:5.2}
\P\bigl(R_{\LR;S,\eta}\le r\bigr)&\ge \P\bigl(R_{\LR;S,\eta}\le r,\,\eta_x^c\ge0\bigr)
\\
&\ge\P\bigl(R_{\LR;S,\eta_\Lambda}\le r/A,\,\eta^c_x\ge0\bigr)
-\P\bigl(R_{\LR;S,\eta}>AR_{\LR;S,\eta_\Lambda},\,\,\eta^c_x\ge0\bigr)
\end{aligned}
\end{equation}
Noting that both events are increasing functions of~$\eta$, for the first probability on the right we get
\begin{equation}
\P\bigl(R_{\LR;S,\eta_\Lambda}\le r/A,\,\eta^c_x\ge0\bigr)\ge\frac12\P\bigl(R_{\LR;S,\eta_\Lambda}\le r/A\bigr)
\end{equation}
using the FKG inequality. For the second probability we set
\begin{equation}
\varphi_u:=\eta^c_u-\frac{\cov(\eta^c_u,\eta^c_x)}{\var(\eta^c_x)}\eta^c_x,\qquad u\in B(M/2),
\end{equation}
and note, since $\cov(\eta^c_u,\eta^c_x)\ge0$, we have
\begin{equation}
R_{\LR;S,\eta}\le R_{\LR;S,\eta_\Lambda+\varphi}\quad\text{on }\{\eta^c_x\ge0\bigr\}.
\end{equation}
But the above definition ensures that~$\varphi$ is independent of~$\eta^c_x$ and a calculation using the explicit form of the law of~$\eta^c$ gives that $\max_{v\in\Lambda}\var(\varphi_v)$ is bounded by a constant independent of~$M$. Markov's inequality and \eqref{E:4.2} then bound the last probability in \eqref{E:5.2} by~$c'/A$ for some constant~$c'\in(0,\infty)$ independent of~$A$ or~$M$.

\newcommand{\frakg}{\mathfrak g}

Next let $\frakg_M\colon\Z^2\to[0,1]$ be discrete harmonic on~$\Lambda$ with $\frakg_M(0):=1$ and~$\frakg_N(u):=0$ whenever $u\not\in B(M)$. Let $\tilde\chi$ have the law of $\chi_M(0)\frakg(\cdot)$ but assume that~$\tilde\chi$ is independent of~$\chi_\Lambda$. The Gibbs-Markov property shows
\begin{equation}
\tilde\chi+\chi_\Lambda\laweq\chi_M.
\end{equation}
A direct use of Lemma~\ref{lem:resistance_decomp} is hampered by the fact that $\var(\tilde\chi(0))$ is of order~$\log M$. However, this is not a problem when~$S$ is at least distance~$\delta M$ from the origin because then $\frakg_N(x)=O(1/\log M)$. Letting $K:=\lfloor N/3\rfloor$, we now note that each translate~$S$ of $B(4N,N)$ contains a translate $\tilde S$ of $B(4N,K)$ which is at least distance~$N$ from the origin and is aligned with one of the longer side of~$S$. Lemma~\ref{lem:resistance_decomp} then gives, for any $b\in\R$,
\begin{equation}
R_{\LR;\tilde S,\chi_\Lambda+\tilde\chi}\ge \e^{-c'' b}R_{\LR; \tilde S,\chi_\Lambda}\ge  \e^{-c'' b} R_{\LR; S,\chi_\Lambda},
\qquad\text{on }\{\tilde\chi(0)\le b\log N\}
\end{equation}
for some~$c''>0$.
Hence
\begin{equation}
\label{E:5.5}
\P\bigl(R_{\LR; S,\chi_\Lambda}\le r/A\bigr)\ge
\P\bigl(R_{\LR; \tilde S,\chi_\Lambda+\tilde\chi}\le \e^{-\tilde c b} r/A\bigr)-\P\bigl(\tilde\chi(0)> b\log N\bigr).
\end{equation}
Now set $r:=C_1\e^{2\cspecial \log\log(M)}$, $A:=\e^{\cspecial \log\log(M)}$ and pick any~$b>0$. Then the last probability in both \eqref{E:5.2} and \eqref{E:5.5} tends to zero as~$N\to\infty$, while, as soon as~$C_1$ is large enough, the first  probability on the right of \eqref{E:5.5} is uniformly positive by Proposition~\ref{prop-4.7} and a routine use of the FKG inequality (to get us from rectangles of the form $B(4N,K)$ to those with aspect ratio~$4$).  The claim follows. 
\end{proofsect}

Using exactly the same argument as in the proof of Lemma~\ref{prop:thm1.4}, we then get:

\begin{corollary}
\label{cor-4.18}
Let~$\eta$ be the GFF in~$\Z^2\smallsetminus\{0\}$. There are $C,C'\in(0,\infty)$ such that 
\begin{equation}
\P\Bigl(R_{B(N)_\eta}(u,v)\ge C\,\e^{Ct\sqrt{\log N}}\,\Bigr)\le C'\e^{-t^2}\log N
\end{equation}
 holds  for all $N \geq 1$ and  all  $t \geq 0$.
\end{corollary}

This is one half of Theorem~\ref{thm-effective-resistance}; the other half will be shown in Section~\ref{sec-proofs}.

\section{Random walk computations}
\label{sec-5}\noindent
Here we use the techniques developed earlier in this paper to finally prove our main results. We begin with some preparatory claims; the actual proofs start to appear in Section~\ref{sec-5.3}.

\subsection{Points with moderate resistance to origin}
\label{subsec:spectral_dim_upper}
Our proofs will require restricting to subsets of~$\Z^2$ of points with only a moderate value of the effective resistance to the origin and/or the boundary of a box centered there in. Here we give the needed bounds on cardinalities of such sets.

\begin{lemma}
\label{lemma-5.3}
Denote $A(N,2N):=B(2N)\smallsetminus   B(N)^\circ $.  For any~$\delta>0$, we have
\begin{equation}
\label{E:5.7}
\P\biggl(\,\sum_{v\in A(N,2N)}\pi_\eta(v)\,1_{\{R_{B(N)_\eta}(0,v)>\e^{(\log N)^{1/2+\delta}}\}}> N^{\psi(\gamma)}\e^{-(\log N)^{\delta}}\biggr)\le\e^{-(\log N)^\delta}
\end{equation}
as soon as~$N$ is sufficiently large.
\end{lemma}

\begin{proofsect}{Proof}
Abbreviate,  as  in~\eqref{E:3.14}, $g:=2/\pi$. We will proceed by a straightforward first-moment estimate, but  first  we have to localize the problem to a finite box. Write $\eta =\eta^f+\eta^c$ where $\eta^f$ is the fine field on the box~$B(4N)$. Since $\var(\eta^c_v)\le\var(\eta_v)$, the variance of~$\eta^c$ is bounded by a constant times $\log N$ uniformly on~$B(N)$ and so, combining Corollary~\ref{cor:field_smoothness1} with a bound at one vertex,
\begin{equation}
\P\Bigl(\,\min_{v\in A(N,2N)}\eta_v^c\le-(\log N)^{1/2+\delta/2}\Bigr)\le c\e^{-\tilde c(\log N)^\delta}.
\end{equation}
On the event when $\eta^c\ge-(\log N)^{1/2+\delta/2}$ we have
\begin{equation}
R_{B(N)_\eta}(0,v)\le R_{B(N)_\eta^f}(0,v)\,\e^{2\gamma (\log N)^{1/2+\delta/2}}
\end{equation}
and so comparing this with the restriction on the effective resistance in~\eqref{E:5.7} we may as well estimate the probability in~\eqref{E:5.7} for~$\eta$ replaced by~$\chi_{4N}$.

Here we will still need to employ a truncation to keep the field~$\chi_{4N}$ below its typical maximum scale. The following crude estimate based on a union bound is sufficient,
\begin{equation}
\P\Bigl(\,\max_{v\in B(N)}\chi_{4N,v}\ge 2\sqrt g\log N+(\log N)^\delta\Bigr)\le c\e^{-\tilde c(\log N)^\delta}
\end{equation}
for some constants~$c,\tilde c\in(0,\infty)$.
 Writing $F_N$ for  the complementary event  and  inserting $F_N$ in the probability in~\eqref{E:5.7} with~$\eta$ replaced by~$\chi_{4N}$, Markov's inequality bounds the result by
\begin{equation}
\label{E:5.8}
N^{-\psi(\gamma)}\e^{(\log N)^\delta}\sum_{v\in A(N,2N)}\E\Bigl(\pi_{\chi_{4N}}(v)1_{\{R_{B(N)_{\chi_{4N}}}(0,v)>\e^{(\log N)^{1/2+\delta}}\}}\,\Big|\,F_N\Bigr)\,.
\end{equation}
Now $\eta\mapsto\pi_\eta(v)$ is increasing while $\{R_{B(N)_\eta}(0,v)>\e^{(\log N)^\delta}\}$ is a decreasing event. Since the conditioning on~$F_N$ preserves the FKG inequality, the quantity in~\eqref{E:5.8} is no larger than
\begin{equation}
\frac1{\P(A_N)^2}
N^{-\psi(\gamma)}\e^{(\log N)^\delta}\sum_{v\in B(N)}\E\bigl(\pi_{\chi_{4N}}(v);\,F_N\bigr) \P\Bigl(R_{B(N)_{\chi_{4N}}}(0,v)>\e^{(\log N)^{1/2+\delta}}\Bigr).
\end{equation}
 Corollary~\ref{cor-4.18}  bounds the last probability by $\e^{-\tilde c(\log N)^{2\delta}}$ so we just have to compute the sum of the expectations of $\pi_{\eta_{4N}}(v)$'s. 

Pick a pair of nearest neighbors~$u$ and~$v$, with~$v\in A(N,2N)$, and let $X:=\chi_{4N,u}+\chi_{4N,v}$. Disregarding the event~$F_N$, a moment computation using $\var(\chi_{4N,v})\le g\log N+c$ for $v\in A(N,2N)$ shows
\begin{equation}
\label{E:5.13}
\E\bigl(\e^{\gamma X}\bigr)=\e^{\frac12\gamma^2\var(X)}
\le c N^{2\gamma^2 g},\qquad v\in A(N,2N).
\end{equation}
On the other hand, a change of measure argument gives
\begin{equation}
\label{E:5.14}
\begin{aligned}
\E\bigl(\e^{\gamma X};F_N\bigr)
&\le\e^{\frac12\gamma^2\var(X)}\P\Bigl(X\le4\sqrt g\log N+2(\log N)^\delta-\gamma\var(X)\Bigr)
\\
&\le c N^{2\gamma^2 g}\,\P\Bigl(X\le4(\sqrt g-\gamma g)\log N+3(\log N)^\delta\Bigr)\end{aligned}
\end{equation}
For $\gamma>\gamma_\cc:=1/\sqrt g$, the probability itself decays as $N^{-2(1-\gamma/\gamma_\cc)^2}\e^{c'(\log N)^\delta}$.  Invoking the definition of~$\psi(\gamma)$ in \eqref{E:1.4}, the inequalities  \twoeqref{E:5.13}{E:5.14} thus give
\begin{equation}
\E\bigl(\pi_{\chi_{4N}}(v);\,F_N\bigr)
\le c N^{\psi(\gamma)-2}\e^{c'(\log N)^\delta},
\qquad v\in A(N,2N).
\end{equation}
Summing over~$v\in A(N,2N)$, the claim follows.
\end{proofsect}

Consider now the set
\begin{equation}
\label{E:XiT}
\Xi_N:=\{0\}\cup\Bigl\{v\in A(N,2N)\colon R_{B(4N)_\eta}(0,v)\le\e^{(\log T)^{1/2+\delta}}\Bigr\}.
\end{equation}
With the help of the above lemma we then show:

\begin{lemma}
\label{lem-good-volume}
For each~$\delta>0$, there is $c>0$ such that for all~$N$ sufficiently large,
\begin{equation}
\P\Bigl(\,\pi_\eta(\Xi_N)\le N^{\psi(\gamma)}\e^{-(\log N)^\delta}\Bigr)\le \frac{c}{(\log N)^2}.
\end{equation}
\end{lemma}

\begin{proofsect}{Proof}
In light of Lemma~\ref{lemma-5.3}, it suffices to show that 
\begin{equation}
\label{E:5.18}
\P\biggl(\,\sum_{v\in A(N,2N)}\pi_\eta(v)\,\le 3N^{\psi(\gamma)}\e^{-(\log N)^\delta}\biggr)\le \frac{c}{(\log N)^2}.
\end{equation}
Thanks to the Gibbs-Markov property, it actually suffices to show this (with~$\delta$ replaced by $\delta/2$) for~$\eta$ replaced by~$\chi_N$ and $A(N,2N)$ replaced by a box~$B(N)$. (Indeed, we just need to take a translate~$B$ of $B(N)$ with~$B\subset A(N,2N)$ and then use the Gibbs-Markov property on a translate of~$B(\lfloor 3N/2\rfloor)$ centered at the same point as~$B$. The contribution of the coarse field is estimated using Corollary~\ref{cor:field_smoothness1}.) 

The argument for \eqref{E:5.18} is different depending on the relation between~$\gamma$ and~$\gamma_\cc$. For~$\gamma\ge\gamma_\cc$ we use that the maximum of the GFF has doubly-exponential lower tails (see~\cite{Ding11}). Invoking the Gibbs-Markov property we then conclude that, with probability at least $\e^{-(\log N)^c}$, for some~$c>0$, there is at least one point~$u$ where
\begin{equation}
\label{E:5.18a}
\chi_{N,u}\ge 2\sqrt g\log N-\cspecial \log\log N
\end{equation}
for some large enough~$C>0$. As~$\chi_{N,u}-\chi_{N,v}$, for~$u$ and~$v$ neighbors, have bounded (in fact, stationary) variances, a union bound shows that \eqref{E:5.18a} will hold also for the neighbors of~$u$. On this event, and denoting by~$v$ a neighbor of~$u$,
\begin{equation}
\sum_{v\in B(N)}\pi_{\chi_N}(v)\ge \e^{\gamma(\chi_{N,u}+\chi_{N,v})} = N^{4\sqrt g\gamma}\e^{-c'\log\log N}.
\end{equation}
Since $4\sqrt g\gamma = 4(\gamma/\gamma_\cc)$ equals~$\psi(\gamma)$ for~$\gamma\ge\gamma_\cc$, we are done here.

Concerning $\gamma<\gamma_\cc$, here we will apply Theorem~\ref{thm-levelset} for $\alpha:=\gamma/\gamma_\cc$. Recall the notation $\mathcal A_{N, \alpha}$ for the level set in~\eqref{eq-def-level-set}. A straightforward computation using the explicit form of the Gaussian probability density shows
\begin{equation}
\P\bigl(x\in \mathcal A_{N, \alpha}\bigr)\ge\frac c{\log N}\,N^{-2\alpha^2},
\end{equation}
and so $\E(|\mathcal A_{N, \alpha}|)\ge c N^{\psi(\gamma)}/\log N$. Theorem~\ref{thm-levelset} ensures that $|\mathcal A_{N, \alpha}|\ge\delta\E(|\mathcal A_{N, \alpha}|)$ occurs with probability $O(\delta^c)$. This statements permits even setting $\delta:=1/(\log N)^{c'}$, whereby the claim readily follows.
\end{proofsect}

We also record an upper estimate on the total volume of $\pi_\eta$:

\begin{lemma}
\label{lem-volume-upper}
For any~$\delta>0$, we have
\begin{equation}
\label{E:5.7a}
\P\biggl(\,\,\sum_{v\in B(N)}\pi_\eta(v)\,> N^{\psi(\gamma)}\e^{(\log N)^\delta}\biggr)\le\e^{-(\log N)^\delta}
\end{equation}
as soon as~$N$ is sufficiently large.
\end{lemma}

\begin{proofsect}{Proof}
This follows directly from the Markov inequality and the calculations in the formulas \twoeqref{E:5.13}{E:5.14}.
\end{proofsect}


\subsection{Upper bound on heat-kernel and exit time}
\label{sec-5.3}\noindent
The starting point  of  our proofs is an upper bound on the return probability for the random walk. We remark that numerous methods  exist in the literature  to derive such bounds. Some of these  are  based on geometric properties of the underlying Markov graph such as isoperimetry  and volume growth,  others  are  based on resistance estimates. The most natural approach to use would be that of \cite{BCK05} (see also~\cite{Kumagai14}); unfortunately, this does not seem possible due to our lack of  required  uniform control  of the  resistance growth. Instead, we base our presentation on the general strategy outlined in \cite[Chapter 21.5]{LPW08}. 
 We begin by restating, and proving, one half of Theorem~\ref{thm_main_spectral}: 

\begin{lemma}
\label{lemma-5.6}
For each $\delta>0$, 
\begin{equation}
\label{E:5.23}
\lim_{T \to \infty}\,\P\Bigl(P_\eta^0(X_{2T} = 0) \leq \e^{ (\log T)^{1/2+\delta}}T^{-1}\Bigr) = 1\,. 
\end{equation}
\end{lemma}

\begin{proofsect}{Proof}
Pick $\delta>0$ and a large integer~$T$, and  recall the notation $\Xi_T$ for the set  in~\eqref{E:XiT}.
Consider the random walk~$\{\wt X_t\colon t\ge0\}$ on the network $B(4T)_\eta$; this walk starts at~$0$ and moves around $B(4T)$ indefinitely using the transition probabilities \eqref{eq-def-transition2} that are modified on the boundary of~$B(4T)$ so that jumps outside~$B(4T)$ are suppressed. Let $\{Y_t\colon t\ge0\}$ record the successive visits of~$\wt X$ to~$\Xi_T$. Then~$Y$ is a Markov chain on~$\Xi_T$ with stationary distribution
\begin{equation}
\label{E:5.19ua}
\nu(x):=\frac{\pi_\eta(x)}{\pi_\eta(\Xi_T)}.
\end{equation}
Let~$\tau_0:=0$, $\tau_1$, $\tau_2$, etc be the times of the successive visits of~$Y$ to~$0$. Define
\begin{equation}
\hat\sigma:=\inf\bigl\{k\ge1\colon\tau_k\ge T\,\,\text{and}\,\, Y_k=0\bigr\}.
\end{equation}
Then we have
\begin{equation}
\begin{aligned}
\label{E:5.21ua}
(T/2) \cdot P^0(\wt X_T=0)&\le E^0\Bigl(\,\,\sum_{t=0}^{T-1}\,1_{\{\wt X_t=0\}}\Bigr)
\\
&\le E^0\Bigl(\,\,\sum_{k=0}^{T-1}\,1_{\{Y_k=0\}}\Bigr)
\le E^0\Bigl(\,\,\sum_{k=0}^{\hat\sigma-1}\,1_{\{Y_k=0\}}\Bigr),
\end{aligned}
\end{equation}
where the first inequality comes from the monotonicity of $2T\mapsto P^0(\wt X_{2T}=0)$ (see, e.g., \cite[Proposition~12.8]{LPW08}) and the second inequality reflects the fact that $0\in\Xi_T$. Since~$Y_{\hat\sigma}=0$, by, e.g., \cite[Lemma 10.5]{LPW08} we have
\begin{equation}
\label{E:5.22ua}
E^0\Bigl(\,\sum_{k=0}^{\hat\sigma-1}\,1_{\{Y_k=x\}}\Bigr) = E^0(\hat\sigma)\nu(x).
\end{equation}
(This is  proved  by noting that the object on the left is a stationary measure for the walk~$Y$ of total mass $E^0(\hat\sigma)$.) By conditioning on~$Y_T$ we  further  estimate 
\begin{equation}
E^0(\hat\sigma)\le T+\max_{u\in\Xi_T}E^u(\sigma_0),
\end{equation}
where $\sigma_0:=\inf\{k\ge0\colon Y_k=0\}$  and note that 
\begin{equation}
E^u(\sigma_0) \le \pi_\eta(\Xi_T)\,R_{B(4T)_\eta}(0,u)\le\pi_\eta(\Xi_T)\,\e^{(\log T)^{1/2+\delta}},
\qquad u\in\Xi_T,
\end{equation}
by the commute-time identity (usually attributed to \cite{CRRST} but appearing already in \cite{BL90}; see also \cite[Corollary 2.21]{LP16}). Combining this with \twoeqref{E:5.21ua}{E:5.22ua} and \eqref{E:5.19ua} we then get 
\begin{equation}
P^0(\wt X_T=0)\le\frac1T\,\pi_\eta(0)\,\e^{(\log T)^{ 1/2+\delta}},
\end{equation}
 which proves  \eqref{E:5.23}  because, due to the jumps being only to nearest neighbors, the walk~$\wt X$ coincides with the walk~$X$ up to time at least~$4T$.
\end{proofsect}

This now permits to give:

\begin{proofsect}{Proof of Theorem~\ref{thm_main_diffusive_exponent}}
A standard calculation based on reversibility and the Cauchy-Schwarz inequality yields
\begin{equation}
\label{E:5.32b}
\begin{aligned}
P^0(X_{2T}=0)&\ge\sum_{x\in B(N)}P^0(X_T=x)P^x(X_T=0)
\\
&=\pi_\eta(0)\sum_{x\in B(N)}\frac{P^0(X_T=x)^2}{\pi_\eta(x)}
\ge\pi_\eta(0)\frac{P^0\bigl(X_T\in B(N)\bigr)^2}{\pi_\eta\bigl(B(N)\bigr)}.
\end{aligned}
\end{equation}
Invoking the upper bound on the heat-kernel and Lemma~\ref{lem-volume-upper}, we get that with probability tending rapidly to one as~$N$ and~$T$ tend to infinity, we have
\begin{equation}
\label{E:5.33b}
P^0\bigl(X_T\in B(N)\bigr)
\le \Bigl[\frac 1 T \e^{(\log T)^{1/2+\delta}}\,N^{\psi(\gamma)}\e^{(\log N)^\delta}\Bigr]^{1/2}\,.
\end{equation}
Setting $T:=N^{\psi(\gamma)}\e^{(\log N)^{1/2+2\delta}}$ gives the desired claim.
\end{proofsect}

The same conclusion could in fact be inferred from the following claim  which constitutes one half of Theorem~\ref{thm_main_exit_time}: 

\begin{lemma}
\label{lemma-5.7}
For each~$\delta>0$ and all $N$ sufficiently large,
\begin{equation}
\P\Bigl(E^0(\tau_{B(N)^\cc})>N^{\psi(\gamma)}\e^{(\log N)^{1/2+\delta}}\Bigr)\le \e^{-(\log N)^\delta}.
\end{equation}
\end{lemma}

\begin{proofsect}{Proof}
By the hitting time identity (or, alternatively, the commute time identity)
\begin{equation}
E^0(\tau_{B(N)^\cc})\le R_{B(N+1)_\eta}\bigl(0,\partial B(N)\bigr)\pi_\eta\bigl(B(N)\bigr)
\end{equation}
The claim then follows from  Corollary~\ref{cor-4.18}  and Lemma~\ref{lem-volume-upper}.
\end{proofsect}

\subsection{Bounding the voltage from below}
\label{subsec:spectral_dim_lower}\noindent
We now move to the proofs of the required lower bounds. Here the focus will be  directed to  the expected exit time which we write using the hitting time identity as
\begin{equation}
\label{E:5.34}
E^0(\tau_{B(N)^\cc}) = R_{B(N+1)_\eta}\bigl(0,\partial B(N)\bigr)\sum_{v\in B(N)}\pi_\eta(v)\phi(v),\,
\end{equation}
where,  using our convention that~$\partial B(N)$ is the external boundary of~$B(N)$, 
\begin{equation}
\phi(v):=P^v(\tau_0<\tau_{\partial B(N)})
\end{equation}
is the electrostatic potential, a.k.a.\ voltage, in~ $B(N)$  
with   $\phi(0)=1$  and $\phi$~vanishing on~$\partial B(N)$. Estimating \eqref{E:5.34} from below naturally requires finding a sufficiently good lower bound on~$\phi$. The idea is to recast the problem using a simple electric network and invoke suitable effective resistance estimates. The following computation will be quite useful:

\begin{lemma}
Consider a resistor network with three nodes, $\{1,2,3\}$, and for each $i,j$ let $c_{ij}$ denote the conductance of the edge $(i,j)$. Let $R_{ij}$ denote the effective resistance between node~$i$ and node~$j$. Then,
\begin{equation}
\label{E:5.35b}
\frac{c_{12}}{c_{12}+c_{13}}=\frac{R_{13}+R_{23}-R_{12}}{2R_{23}}.
\end{equation}
\end{lemma}

\begin{proofsect}{Proof}
 By the Parallel and Series Laws we get,
\begin{equation}
R_{12} = \frac{1}{c_{12} + (c_{13}^{-1} + c_{23}^{-1})^{-1}}.
\end{equation}
Computing $R_{13}$ and $R_{23}$ in a similar way and plugging them into the right hand side of \eqref{E:5.35b}, we deduce the claim.
\end{proofsect}

Using this lemma we then get:

\begin{corollary}
\label{cor-5.7}
For any~$v\in B(N)\smallsetminus\{0\}$ and~$\phi$ as above,
\begin{multline}
\label{E:5.39}
\quad
2R_{B(N+1)_\eta}\bigl(0,\partial B(N)\bigr)\phi(v) 
\\
=R_{B(N+1)_\eta}\bigl(0,\partial B(N)\bigr)+R_{B(N+1)_\eta}\bigl(v,\partial B(N)\bigr)-R_{B(N+1)_\eta}(0,v).
\quad
\end{multline}
\end{corollary}

\begin{proofsect}{Proof}
As~$v\not\in\{0\}\cup\partial B(N)$,  we may apply the network reduction principle to represent the problem on an effective network of three nodes, with node~$1$ labeling~$v$, node~$2$ marking the origin and node~$3$ standing for~$\partial B(N)$.
Since~$\phi$ is harmonic on~$B(N)\smallsetminus\{0\}$, it is also harmonic on the effective network. But there $\phi(v)$ is just the probability that the random walk at~$v$ jumps right to~$0$ in the first step. Using conductances, this probability is exactly the expression on the left of~\eqref{E:5.35b}. Plugging in the effective resistances, the claim follows.
\end{proofsect}

A key point is to bound the expression involving effective resistances on the right of \eqref{E:5.39} from below. This is the subject of:

\begin{proposition}
\label{prop-5.8}
Let $D_{N,\eta}(v)$ denote the difference on the right of \eqref{E:5.39}, i.e., 
\begin{equation}
D_{N,\eta}(v) := R_{B(N+1)_\eta}\bigl(0,\partial B(N)\bigr)+R_{B(N+1)_\eta}\bigl(v,\partial B(N)\bigr)-R_{B(N+1)_\eta}(0,v).
\end{equation}
For any $\delta\in(0,1)$, we then have
\begin{equation}
\label{E:5.40}
\lim_{N\to\infty}\,\P\Bigl(\,\,\min_{v\in B(\lfloor N\e^{-(\log N)^\delta}\rfloor)}\,\,D_{N,\eta}(v)\ge \log N\Bigr)=1.
\end{equation}
\end{proposition}
 The main idea behind the proof is to find an annulus between $B_N$ and $B(\lfloor N\e^{-(\log N)^\delta}\rfloor)$ with high probability which provides a ``shortcut'' for diverting the flow from 0 to $v$.  To this end we recall the annulus decomposition of the GFF from Section~\ref{subsec:smoothness}. Let~$b:=8$ and for a given~$N\ge1$ and $n\in\N$, set $N':=b^n N$. Define the annuli
\begin{equation}
\label{E:5.36ua}
A'_{n,k}:=B(b^{n-k+1}N)\smallsetminus B(b^{n-k}N)^\circ,\qquad k=1,\dots,n-1.
\end{equation}
and
\begin{equation}
\label{E:5.37ua}
A_{n,k}:=B(4b^{ n-k}N)\smallsetminus B(2b^{n-k}N)^\circ,\qquad k=1,\dots,n-1.
\end{equation}
Note that $A_{n,k}\subset A'_{n,k}$. Write $\eta=\eta^c+\chi_{2N'}$, where $\eta^c$ is the coarse field on~$B(2N')$ and $\chi_{2N'}$ is the corresponding fine field. Denote
\begin{equation}
\Delta':=\max_{v\in B(N')}|\eta^c_v|.
\end{equation}
Define $M_{n,k}$ as in~\eqref{E:3.66} and for $1\le \ell<m\le n$ let
\begin{equation}
\Delta_{\ell,m} := \max_{k=\ell,\dots,m-1}\,\,\max_{v \in A_{ n,k}}\,\Bigl|
M_{ n,k} - 
\E\bigl(\chi_{N', v} \,\big|\,\chi_{N', v}: v \in {\textstyle\bigcup_{n \geq j \geq n - k}}\partial B(b^jN)\bigr)\Bigr|.
\end{equation}
(Both objects are measurable with respect to~$\eta$.) Similarly to Lemma~\ref{lem:anchoring} we get
\begin{equation}
\label{E:5.47a}
\P\Bigl(\Delta_{\ell,m}\ge\wt C\sqrt{m-\ell}\,\Bigr)\le\frac1{(m-\ell)^2}
\end{equation}
as soon as $m-\ell$ is sufficiently large.

Let $\chi_{k,v}^f$ denote the fine field on~$A_{n,k}'$,
\begin{equation}
\chi_{k,v}^f:=\E\Bigl(\chi_{2N',v}\,\Big|\,\chi_{2N',u}\colon u\in\partial A_{n,k}\Bigr),\qquad v\in A_{n,k},
\end{equation}
(we think of $\chi_k^f$ as set to zero outside $A'_{n,k}$) and $\chi_k^c:=\chi_{2N'}-\chi_k^f$ be the corresponding coarse field. The definitions ensure
\begin{equation}
\max_{k=\ell,\dots,m}\,\max_{v\in A_{n,k}}\,\bigl|\eta_v-(\chi_{k,v}^f+M_{n,k})\bigr|\le\Delta_{\ell,m}+\Delta'.
\end{equation}
Note also that $M_{n,k}$ and~$\chi_{k'}^f$ are independent as long as $k\ge k'$.

Next recall that $R_{A,\eta}$, for~$A$ an annulus in~$\Z^2$, denotes the sum of the effective resistances in network $A_\eta$ between the shorter sides of the four maximal rectangles contained in~$A$. Recall also that $R_{A,\eta}(\bdryin A,\bdryout A)$ denotes the effective resistance in~$A_\eta$ between the inner and outer boundaries of annulus~$A$. We define the events:
\begin{multline}
\label{E:5.47}
\mathcal E_{n,k}^\star:=\Bigl\{R_{A_{n,k},\chi_k^f}(\bdryin A_{n,k},\bdryout A_{n,k})\ge \e^{-3\cspecial \log\log(b^{-k}N')}\Bigr\}
\\
\cap\Bigl\{M_{n,k}\le -C^\star\sqrt{k\log\log(k)}\Bigr\}
\end{multline}
and
\begin{equation}
\label{E:5.48}
\mathcal E_{n,k}^{\star\star}:=\Bigl\{R_{A_{n,k},\chi_k^f}\le \e^{\cspecial \log\log(b^{-k}N')}\Bigr\}
\cap\Bigl\{\min_{v\in A_{n,k}}\eta_{k,v}^c\ge -\log\log(N')\Bigr\}.
\end{equation}
Here~$ \cspecial $ is the constant Proposition~\ref{prop:zero_mean} and~$C^\star$ is fixed via:

\begin{lemma}
\label{lemma-5.11}
For each~$\delta>0$ there are~$n_0\ge1$, $N_0\ge1$, $c_1\in(0,\infty)$ such that one can choose $C^\star\in(0,\infty)$ in the definitions of $\mathcal E_{n,k}^\star$ and $\mathcal E_{n,k}^{\star\star}$ so that, for all $N\ge N_0$ and all $n\ge n_0$,
\begin{equation}
\P\Bigl(\exists k^\star,k_\star\colon \e^{\sqrt{\log n}}<k^\star<k_\star<n,\,\, \mathcal E_{n,k^\star}^\star\cap \mathcal E_{n,k_\star}^{\star\star}
\text{\rm\ occurs}\Bigr)\ge1-\frac{c_1}{\log\log n}.
\end{equation}
\end{lemma}

\begin{proofsect}{Proof}
Abbreviate by~$E_k^\star$ the first event on the right of \eqref{E:5.47}. This event is measurable with respect to $\chi_k^f$ and so $\{E_k\colon k=1,\dots,n\}$ are independent. By Lemma~\ref{lemma-annuli}, $\P(E_k^\star)\ge p$ holds for some~$p>0$ and all~$k$ as soon as~$N\ge N_0$. We are first interested in a simultaneous occurrence of~$E_k^\star$ and $\{M_{n,k}\le -C^\star\sqrt{k\log\log(k)}\}$.

Recalling that $k\mapsto M_{n,k}$ is a random walk, define the stopping time
\begin{equation}
T_n:=\inf\bigl\{k\colon \e^{\sqrt{\log(n)}}\le k \le n,\, M_{n,k}\le -2C^\star\sqrt{k\log\log(k)}\bigr\}.
\end{equation}
Then, for~$C^\star$ sufficiently small, Lemma~\ref{lem:finite_LIL} shows
\begin{equation}
\P\bigl(T_n> n/4\bigr)\le\frac{c_1}{\log\log n}
\end{equation}
for some constant~$c_1\in(0,\infty)$.
Since the increments of~$M_{n,k}$ are independent centered Gaussians with a uniform bound on their tail, for the event
\begin{equation}
\mathcal G_{n,k}:=\Bigl\{M_{n,k+j+1}-M_{n,k+j}\le \log(k)\colon 0\le j\le\log(k)^2\Bigr\}
\end{equation}
the fact that $T_n\ge\e^{\sqrt{\log(n)}}$ yields
\begin{equation}
\P\Bigl(\bigl\{T_n\le n/4\bigr\}\cap\mathcal G_{n,T_n}\Bigr)\ge 1-\frac{2c_1}{\log\log n}
\end{equation}
as soon as~$n$ is larger than a positive constant. Under a similar restriction on~$n$, we then also have
\begin{equation}
\bigl\{T_n\le n/4\bigr\}\cap\mathcal G_{n,T_n}
\subseteq\bigcap_{T_n\le k\le T_n+(\log T_n)^2}\bigl\{M_{n,k}\le -C^\star\sqrt{k\log\log(k)}\bigr\}.
\end{equation}
Therefore, on the event on the left,~$E_k^\star\cap\{M_{n,k}\le -C^\star\sqrt{k\log\log(k)}\}$ will not occur for some~$k<n/2$ only if the sequence $\{1_{E_k^\cc}\colon 1\le k\le n\}$ contains a run of 1's of length at least $\log(n)^2$. This has probability~$n (1-p)^{\lfloor\log(n)\rfloor^2}$. As~$p>0$, we get
\begin{equation}
\P\Bigl(\bigcup_{1\le k<n/2}E_k^\star\cap\bigl\{M_{n,k}\le -C^\star\sqrt{k\log\log(k)}\bigr\}\Bigr)\ge 1-\frac{2c_1}{\log\log n}
\end{equation}
as soon as~$n$ is larger than some positive constant.

For  event  $\mathcal E_{n,k}^{\star\star}$, the fact that the coarse field~$\eta^c$ on~$A_{n,k}$ has uniformly bounded variances implies, via Corollary~\ref{cor:field_smoothness1},
\begin{equation}
\P\biggl(\,\bigcup_{0\le k-n/2\le (\log n)^2}\Bigl\{\min_{v\in A_{n,k}}\eta_{k,v}^c\ge -\log\log(N')\Bigr\}\biggr)\ge1-c'(\log n)^2\e^{-c''(\log\log N')^2}
\end{equation}
for some~$c',c''>0$.
Proposition~\ref{prop-4.7} in turn shows that the first event on the right of \eqref{E:5.48} has a uniformly positive probability. The claim then follows as before.
\end{proofsect}

We are ready to complete:

\begin{proofsect}{Proof of Proposition~\ref{prop-5.8}}
Fix~$N'\ge1$ large and, given~$\delta\in(0,1)$, let $n$ be the largest integer such that $N:=b^{-n}N'>N'\e^{-(\log N')^\delta}$. (We are assuming the setting of Lemma~\ref{lemma-5.11}.) Abbreviate $k_n:=\e^{\sqrt{\log n}}$ and suppose that the event
\begin{equation}
\label{E:5.60}
\mathcal E_{n,k^\star}^{\star}\cap\mathcal E_{n,k_\star}^{\star\star}\cap\bigl\{\Delta'\le \log\log(N')\bigr\}\cap\bigcap_{k_n\le k\le n}\bigl\{\Delta_{k_n,k}\le \wt C\sqrt{k}\bigr\}
\end{equation}
 occurs for some $k^\star,k_\star$ with $k_n\le k^\star<k_\star\le n$. Then
\begin{equation}
\label{E:5.61}
\begin{aligned}
R_{A_{n,k^\star},\eta}\bigl(\bdryin A_{n,k^\star},\bdryout A_{n,k^\star}\bigr)
&\ge \e^{-2\gamma(\Delta'+M_{n,k^\star}+\Delta_{k_n,k^\star})}
R_{A_{n,k^\star}^\star,\chi_{k^\star}^f}\bigl(\bdryin A_{n,k^\star},\bdryout A_{n,k^\star}\bigr)
\\
&\ge\e^{2\gamma[C^\star\sqrt{k^\star\log\log k^\star}-\wt C\sqrt{k^\star}-\log\log(N')\,]}\,\e^{-3\cspecial \log\log(N')}
\\
&\ge \e^{\tilde c\sqrt{k^\star\log\log k^\star}}
\end{aligned}
\end{equation}
holds for some constant $\tilde c>0$, where we used that $k^\star\ge k_n$ implies $\sqrt{k^\star}\gg\log\log(N')$ as soon as~$N'$ is sufficiently large.
Similarly, abbreviating $\mathfrak m_{n,k}:=\min_{v\in A_{n,k}}\eta_{k,v}^c$, we get
\begin{equation}
\label{E:5.62}
\begin{aligned}
R_{A_{n,k_\star},\eta}
\le \e^{-2\gamma(\mathfrak m_{n,k_\star}-\Delta')}R_{A_{n,k_\star},\chi_{k_\star}^f}
&\le \e^{4\gamma\log\log(N')}\,\e^{\cspecial \log\log(N')}
\\&\le R_{A_{n,k^\star},\eta}\bigl(\bdryin A_{n,k^\star},\bdryout A_{n,k^\star}\bigr) - \log(N'),
\end{aligned}
\end{equation}
where we again used that $\sqrt{k^\star}\gg\log\log(N')$.

Next observe that if $v\in B(N)$, then the Nash-Williams estimate implies
\begin{equation}
\label{E:5.63}
R_{B(N'),\eta}\bigl(v,\partial B(N')\bigr)\ge
R_{B(N'),\eta}\bigl(v,\bdryin A_{n,k^\star}\bigr)+R_{A_{n,k^\star},\eta}\bigl(\bdryin A_{n,k^\star},\bdryout A_{n,k^\star}\bigr)
\end{equation}
while the Series Law gives
\begin{equation}
R_{B(N'),\eta}(0,v)\le
R_{B(N'),\eta}\bigl(0,\bdryout A_{n,k_\star}\bigr)+R_{B(N'),\eta}\bigl(v,\bdryout A_{n,k_\star}\bigr)
+R_{A_{n,k_\star},\eta}.
\end{equation}
Since $k^\star<k_\star$ implies that $A_{n,k^\star}$ lies outside $A_{n,k_\star}$, we also have
\begin{equation}
\label{E:5.65}
R_{B(N'),\eta}\bigl(v,\bdryin A_{n,k^\star}\bigr)
\ge
R_{B(N'),\eta}\bigl(v,\bdryout A_{n,k_\star}\bigr).
\end{equation}
Combining \twoeqref{E:5.62}{E:5.65} we thus get that $D_{N',\eta}(v)\ge\log N'$ for all $v\in B(N)$ as soon as the event in~\eqref{E:5.60} occurs. The claim follows (for~$N$ replaced by~$N'$) from \eqref{E:5.47a} and Lemma~\ref{lemma-5.11}.
\end{proofsect}

\subsection{Proofs of  the  main results}
\label{sec-proofs}\noindent
We  will now move to prove the remaining part of our main  results. Fix~$\delta\in(0,\infty)$ small,  abbreviate $ N_\delta:=N\e^{(\log N)^\delta}$  and consider the set
\begin{multline}
\Xi_{N}^\star:=\{0\}\cup\partial B(N)
\\
\cup\Bigl\{v\in A(N_\delta,2 N_\delta)\colon R_{B(N+1)_\eta}(0,
v)\vee R_{B(N+1)_\eta}(v,
\partial B(N))\le\e^{(\log N)^{1/2+\delta}}\Bigr\}.
\end{multline}
We again claim:

\begin{lemma}
\label{lem-good-volume2}
For each~$\delta>0$, there is $c>0$ such that for all~$N$ sufficiently large, 
\begin{equation}
\label{E:5.32c}
\P\Bigl(\,\pi_\eta(\Xi_{N}^\star)\le N^{\psi(\gamma)}\e^{-(\log N)^{\delta}}\Bigr)\le \frac{c}{(\log N)^2}.
\end{equation}
\end{lemma}

\begin{proofsect}{Proof}
Using the same proof, Lemma~\ref{lemma-5.3} applies also for resistance $R_{B(N)_\eta}(v,\partial B(N))$. In light of 
\begin{equation}
R_{\,B(N+1)_\eta}\bigl(v,\partial B(N)\bigr)\le R_{\,B(N+1)_\eta}(v,u),\qquad u\in \partial B(N),
\end{equation}
 Corollary~\ref{cor-4.18}  applies to $R_{B(N+1)_\eta}(v,\partial B(N))$ just as well. Combining this with \eqref{E:5.18}, we proceed as in the proof of  Lemma~\ref{lem-good-volume}  to get the result. 
\end{proofsect}

We are now ready to give:

\begin{proofsect}{Proof of Theorem~\ref{thm_main_exit_time}}
The upper bound has already been shown in Lemma~\ref{lemma-5.7}, so we just need to derive the corresponding lower bound. For this we write \eqref{E:5.34} as a bound and apply \eqref{E:5.39} with Proposition~\ref{prop-5.8} to get  that, with probability tending to one as~$N\to\infty$, 
\begin{equation}
\label{E:5.69}
\E^0(\tau_{ B(N)^\cc}) \ge R_{B(N+1)_\eta}\bigl(0,\partial B(N)\bigr)\sum_{v\in \Xi_N^\star}\pi_\eta(v)\phi(v)
\ge\pi_\eta(\Xi_N^\star)\log(N).
\end{equation}
The claim then follows from Lemma~\ref{lem-good-volume2}.
\end{proofsect}

We then use the lower bound on the expected exit time to also get:

\begin{proofsect}{Proof of Theorem~\ref{thm_main_spectral}}
The upper bound on the return probability has already been proved in Lemma~\ref{lemma-5.6}, so we will focus on the lower bound and recurrence. Consider again the random walk~$\wt X$ on~$B(N+1)$ and let~$Y$ be its trace on~$\Xi_N^\star$. Let~$\hat\tau_{\partial B(N)}:=\inf\{k\ge0\colon Y_k\in\partial B(N)\}$. Then
\begin{equation}
\begin{aligned}
\quad \E^0(\hat\tau_{\partial B(N)})
&\le T \P^0(\hat\tau_{\partial B(N)}\le T)+\P^0(\hat\tau_{\partial B(N)}>T)\bigl(T+\max_{v\in \Xi_N^\star\smallsetminus\partial B(N)}\E^v(\hat\tau_{\partial B(N)})\bigr)
\\
&=T+\P^0(\hat\tau_{\partial B(N)}>T)\max_{v\in \Xi_N^\star\smallsetminus\partial B(N)}\E^v(\hat\tau_{\partial B(N)}).
\end{aligned}
\end{equation}
The hitting time estimate in conjunction with the definition of~$\Xi_N^\star$ gives
\begin{equation}
\E^v(\hat\tau_{\partial B(N)})\le \pi_\eta(\Xi_N^\star)\,\e^{(\log N)^{1/2+\delta}},\qquad v\in \Xi_N^\star\smallsetminus\partial B(N)
\end{equation}
whereby we get
\begin{equation}
\P^0(\hat\tau_{\partial B(N)}>T)
\ge \pi_\eta(\Xi_N^\star)^{-1}\e^{-(\log N)^{1/2+\delta}}\bigl(\E^0(\hat\tau_{\partial B(N)})
-T\bigr).
\end{equation}
Since \eqref{E:5.69} applies also for the expectation of $\hat\tau_{\partial B(N)}$, the choice $N:=T^{1/\psi(\gamma)}\e^{(\log N)^\delta}$ implies $\E^0(\hat\tau_{\partial B(N)})\ge 2T$ and thus, using \eqref{E:5.69} one more time,
\begin{equation}
\P^0(\hat\tau_{\partial B(N)}>T)\ge \e^{-(\log N)^{1/2+\delta}}.
\end{equation}
But $\hat\tau_{\partial B(N)}\le\tau_{\partial B(N)} :=\inf\{k\ge0\colon X_k\in\partial B(N)\}$  and so we get
\begin{equation}
\P^0\bigl(X_T\in B(N)\bigr)\ge \P^0(\tau_{\partial B(N)}>T)\ge\e^{-(\log N)^{1/2+\delta}}
\end{equation}
as well. Using this in~\eqref{E:5.32b}, the desired lower bound then follows from, e.g., \eqref{E:5.18}.

It remains to show recurrence. Here we note that \eqref{E:5.61} and \eqref{E:5.63} along with Lemma~\ref{lemma-5.11} imply that $R_{B(N),\eta}(0,\partial B(N))\to\infty$ in probability along a sufficiently rapidly growing deterministic sequence of~$N$'s.  Since  the sequence of resistances is increasing in~$N$, the convergence holds almost surely. By a well known criterion, this implies recurrence.
\end{proofsect}

It remains to give:

\begin{proofsect}{Proof of Theorem~\ref{thm-effective-resistance}}
The bound \eqref{E:1.10ua} has already been shown in Corollary~\ref{cor-4.18}. 
A completely analogous argument produces also the estimate
\begin{equation}
\P\Bigl(R_{\Z^2_\eta}(0,B(N)^\cc)\ge C\e^{C t\sqrt{(\log N)\log\log N}}\,\Bigr)\le C'(\log N)^{1-t^2}.
\end{equation}
Taking~$t$ sufficiently large and invoking the Borel-Cantelli lemma then proves \eqref{E:1.11ua} for~$N$ restricted to powers of~$2$. The monotonicity of $N\mapsto R_{\Z^2_\eta}(0,B(N)^\cc)$ extends this to general~$N$.

It remains to prove \eqref{E:1.12ua}.  We will use a decomposition of~$\eta$ from \cite[Proposition~3.12]{BL3}. Let~$b:=8$ and consider the annuli $A_k':=B(b^{k+1})\smallsetminus B(b^k)^\circ$ and $A_k:=B(4b^k)\smallsetminus B(2 b^k)$ for all $k\ge0$. Then
\begin{equation}
\label{E:5.69ua}
\eta_v = \sum_{k\ge0}\Bigl[\mathfrak b_k(v)X_k+\psi_{k,v}+\eta^f_{k,v}\Bigr],
\end{equation}
where $\mathfrak b_k\colon\Z^2\to\R$ is a function such that
\begin{equation}
\label{E:5.70ua}
\mathfrak b_k(v)=-1\,\,\text{if}\,\, v\not\in B(b^k)\quad\text{\rm and}\quad
\bigl|\mathfrak b_k(v)\bigr|\le cb^{\ell-k}\,\,\text{if}\,\, v\in B(b^\ell)\subseteq B(b^k),
\end{equation}
while $\{X_k\colon k\ge0\}$ are random variables and $\{\psi_k\colon k\ge0\}$ and $\{\eta_k^f\colon k\ge0\}$ are random fields (all measurable with respect to~$\eta$) that are independent of one another and distributed as centered Gaussian with the specifics of the law determined as follows:
\settowidth{\leftmargini}{(111)}
\begin{enumerate}
\item[(1)] $\lim_{k\to\infty}\var(X_k)=g\log b$,
\item[(2)] writing $\chi_k^c$ for the coarse field obtained as the conditional expectation of the GFF on~$B(b^k)$ given its values on $\partial B(b^{k-1})$, we have
\begin{equation}
\psi_k\laweq\chi_k^c-\E(\chi_k^c|\chi_{k,0}^c),
\end{equation}
\item[(3)] $\eta_k^f$ is the fine field on~$A_k'$.
\end{enumerate}
For~$\psi_k$, we in addition have the following variance estimate,
\begin{equation}
\label{E:5.72ua}
\var(\psi_{k,v})\le c b^{\ell-k},\qquad v\in B(b^\ell)\subseteq B(b^k).
\end{equation}
See~\cite[Lemma~3.7]{BL3} for \eqref{E:5.70ua} and~\cite[Lemma~3.8]{BL3} for \eqref{E:5.72ua}. 

Clearly, only one of the fine fields $\chi_k^f$ can contribute in \eqref{E:5.69ua} for each given~$v$ and $\chi_{k,v}=0$ unless~$v\in B(b^k)$. Setting (with some abuse of our earlier notation),
\begin{equation}
\Delta_k:=\max_{v\in A_k}\,\Bigl|\sum_{j>k}\mathfrak b_j(v) X_j+\sum_{j\ge k}\psi_{j,v}\Bigr|
\end{equation}
\cite[Lemma~3.8]{BL3} shows that, for some constants~$c,c'\in(0,\infty)$,
\begin{equation}
\label{E:5.73}
\P\bigl(\Delta_k\ge c+t\bigr)\le \e^{-c' t^2},\qquad t\ge0.
\end{equation}
The first half of \eqref{E:5.70ua} then lets us write
\begin{equation}
\eta_v +\sum_{j=0}^k X_k -\chi_{k,v}^f\le \Delta_k,\qquad v\in A_k.
\end{equation}
We now set $S_k:= \sum_{j=0}^k X_k$ and note that the Nash-Williams estimate  and Lemma~\ref{lem:resistance_decomp} imply
\begin{equation}
\label{E:5.75}
R_{B(N+1)_{\eta}}\bigl(0,\partial B(N)\bigr)\ge \max_{k=1,\dots,n-1}\Bigl[\,\e^{-2\gamma(\Delta_k-S_k)}\, R_{A_{k},\eta_k^f}\bigl(\bdryin A_{n,k},\bdryout A_{n,k}\bigr)\Bigr]
\end{equation}
where $n:=\max\{k\ge0\colon b^k\le N\}$.

The Borel-Cantelli lemma, straightforward estimates for Gaussian random variables and the tail bound \eqref{E:5.73} show the existence of~a random variable $n_0$ with~$\P(n_0<\infty)=1$ such that
\begin{equation}
\label{E:5.77}
\max_{0\le k\le n}\Delta_k\le(\log n)^2,\quad n\ge n_0,
\end{equation}
and
\begin{equation}
\label{E:5.78}
\max_{\begin{subarray}{c}
0\le k,l\le n\\|k-\ell|\le (\log n)^{5/4}
\end{subarray}}
|S_k-S_l|\le (\log n)^2,\quad n\ge n_0.
\end{equation}
Denoting
\begin{equation}
F_k:=\Bigl\{R_{A_{k},\eta_k^f}\bigl(\bdryin A_{n,k},\bdryout A_{n,k}\bigr)\ge C\e^{-3\cspecial \log(b^k)}\Bigr\}.
\end{equation}
Lemma~\ref{lemma-annuli} tells us $\inf_{k\ge1}\P(F_k)>0$. The events~$\{F_k\}$ are also independent, we may assume that~$n_0$ is also such that, for~$n\ge n_0$, the longest run of~$k=1,\dots,n$ where~$F_k$ fails is at most a constant time~$\log n$. Since by \eqref{E:5.78},~$S$ changes by at most~$(\log n)^2$ over that interval, from \eqref{E:5.77} and $b^k\le N$ we get
\begin{equation}
\label{E:5.80}
R_{B(N+1)_{\eta}}\bigl(0,\partial B(N)\bigr)
\ge\Bigl(\,\max_{0\le k\le n}\e^{2\gamma S_k}\Bigr)\e^{-4\gamma(\log n)^2}\,C\e^{-3\cspecial \log N}
\end{equation}
as soon as~$N$ (and thus~$n$) is sufficiently large.
The growth of the maximum is controlled by Chung's Law of the Iterated Logarithm \cite{Chung} which gives
\begin{equation}
\liminf_{n\to\infty}\,\frac{\max_{k\le n}S_k}{\sqrt{n/\log n}}>0
\end{equation}
The claim follows by using this in \eqref{E:5.80} and invoking that~$n$ is order~$\log N$.
\end{proofsect}

\def\cprime{$'$}

\end{document}